\documentclass[11pt, letterpaper, hidelinks]{article}

\usepackage[margin=1in]{geometry}
\usepackage{authblk}

\usepackage{url}
\usepackage{mathtools}
\usepackage[dvipsnames]{xcolor}
\usepackage{amsmath,amsthm, amsfonts, amssymb}
\renewcommand{\proofname}{\textit{\textbf{Proof}}}

\makeatletter
\renewenvironment{proof}[1][\proofname]{\par
  \pushQED{\qed}%
  \normalfont \topsep6\p@\@plus6\p@\relax
  \trivlist
  \item[\hskip\labelsep
        #1\@addpunct{}]\ignorespaces % Modified this line
}{%
  \popQED\endtrivlist\@endpefalse
}
\makeatother
\theoremstyle{plain}
\newtheorem{theorem}{Theorem}[section]
\newtheorem{lemma}{Lemma}[section]
\newtheorem{proposition}{Proposition}[section]
\newtheorem{corollary}{Corollary}[section]
\theoremstyle{definition}
\newtheorem{definition}{Definition}[section]

\newtheorem{remark}{Remark}[section]
\usepackage[hypertexnames=false]{hyperref}
\newcommand\myshade{85}
\colorlet{mylinkcolor}{violet}
\colorlet{mycitecolor}{YellowOrange}
\colorlet{myurlcolor}{Aquamarine}
\hypersetup{
  linkcolor  = mylinkcolor!\myshade!black,
  citecolor  = mycitecolor!\myshade!black,
  urlcolor   = myurlcolor!\myshade!black,
  colorlinks = true,
}
\usepackage{enumitem, comment, xifthen}
\usepackage{graphicx}
\usepackage{etoolbox}
\usepackage{tikz}
\usetikzlibrary{math}
\usepackage{mathabx} % gives \vvver
\usepackage{subfig}
\graphicspath{{figs/}}
\DeclareMathAlphabet{\mathsf}{OT1}{qhv}{m}{n}
\DeclareMathDelimiter{\vvvert} {0}{matha}{"7E}{mathx}{"17}%

\makeatletter
\def\letterdef#1#2#3{\def\letterdef@##1{\expandafter\def\csname #1\endcsname{#2}}%
  \letterdef@@#3{?\@car{}}\@nil}
\def\letterdef@@#1{\@gobble#1\letterdef@{#1}\letterdef@@}
\makeatother

\DeclarePairedDelimiterX{\klx}[2]{(}{)}{%
  #1\;\delimsize\|\;#2%
}
\DeclarePairedDelimiterX{\quantklx}[3]{(}{)}{%
  #1\;\delimsize\|\;#2\;\delimsize\vert\;#3%
}
\DeclarePairedDelimiterX{\inner}[2]{\langle}{\rangle}{%
  #1,#2%
}

\newcommand{\R}{\mathbf R} % reals
\letterdef{c#1}{\mathcal{#1}}{ABCDEFGHIJKLMNOPQRSTUVWXYZ}
\letterdef{b#1}{\mathbb{#1}}{ABCDEFGHIJKLMNOPQRSTUVWXYZ}
\letterdef{bf#1}{\mathbf{#1}}{ABCDEFGHIJKLMNOPQRSTUVWXYZ}
\newcommand{\twomax}[2]{\ensuremath{#1 \lor #2}}
\newcommand{\twomin}[2]{\ensuremath{#1 \land #2}}
\newcommand{\ceil}[1]{\left\lceil #1 \right\rceil}
\newcommand{\floor}[1]{\left\lfloor #1 \right\rfloor}

\newcommand{\e}{\mathrm{e}}
\DeclareMathOperator{\sign}{\bf sign}
\newcommand{\ud}[0]{\mathrm{d}}  % upright d for diff., int.
\newcommand{\1}{\mathbf 1} % ones
\let\ones\1
\newcommand{\half}{\frac12} 
\let\epsilon\varepsilon
\newcommand{\eps}{\varepsilon}

\let\subseteq\subset

\renewcommand{\le}{\leqslant}
\renewcommand{\ge}{\geqslant}
\renewcommand{\leq}{\leqslant}
\renewcommand{\geq}{\geqslant}

\newcommand{\argmin}{\mathop{\rm arg\,min}}
\newcommand{\argmax}{\mathop{\rm arg\,max}}

\DeclareMathOperator{\diam}{diam}
\DeclareMathOperator{\dist}{\bf dist}

\newcommand{\ind}{\textnormal{ind.}}
\newcommand{\Var}{\mathrm{Var}}

\makeatletter
\newcommand{\E}{\operatorname*{\mathbf{E}}\ilimits@}
\makeatother
\makeatletter
\renewcommand{\P}{\operatorname*{\mathbf{P}}\ilimits@}
\makeatother

\newcommand{\ie}{\textit{i}.\textit{e}., }

\newcounter{algorithmctr}
\renewcommand{\thealgorithmctr}{\arabic{algorithmctr}}
   {\refstepcounter{algorithmctr}\begin{list}{}{%
       \setlength{\rightmargin}{0\linewidth}%
       \setlength{\leftmargin}{0\linewidth}}%
       \rmfamily\small
       \item[]{\setlength{\parskip}{0ex}\hrulefill\par%
        \nopagebreak{\bfseries\textsf{Algorithm \thealgorithmctr~}}}}%
   {{\setlength{\parskip}{-1ex}\nopagebreak\par\hrulefill} \end{list}}

\makeatletter
\long\def\@makecaption#1#2{
        \vskip 0.8ex
        \setbox\@tempboxa\hbox{\small {\bf #1.} #2}
        \parindent 1.5em 
        \dimen0=\hsize
        \advance\dimen0 by -3em
        \ifdim \wd\@tempboxa >\dimen0
                \hbox to \hsize{
                        \parindent 0em
                        \hfil 
                        \parbox{\dimen0}{\def\baselinestretch{0.96}\small
                                {\bf #1.} #2
                                } 
                        \hfil}
        \else \hbox to \hsize{\hfil \box\@tempboxa \hfil}
        \fi
        }
\makeatother

\newcommand{\Normal}[2]{\mathsf{N}\left(#1, #2\right)}

\newcommand{\dimension}{d}
\DeclareMathOperator{\rad}{rad}

\newcommand{\thetastar}{\theta^\star}
\newcommand{\etastar}{{\eta^\star}}
\newcommand{\MetricChar}{\overline{\eps}}
\newcommand{\SeriesFixedPoint}{\cS}
\newcommand{\TruncIntegral}{\cJ}
\newcommand{\LocEnt}{h^{\rm loc}}
\newcommand{\GlobEnt}{h}

\newcommand{\Minoration}{\Psi}

\newcommand{\ChatFixed}{r}
\newcommand{\BoundOne}{\cT}

\let\hat\widehat
\newcommand{\GenPackSet}[1]{\cM_{#1}}
\newcommand{\NoiseVec}{\xi}
\let\phi\varphi

\newcommand{\BallProj}{\Pi_{B^d_1}}
\newcommand{\lambdastar}{\lambda^\star}

\DeclareMathOperator{\inrad}{inrad}
\newcommand{\PeakIndex}[1]{i^\star_{#1}}

\renewcommand{\ind}[1]{\ensuremath{\1\left[\,#1\,\right]}}
\newcommand{\defn}{=}
\usepackage{cleveref}
\usepackage{autonum}
\renewcommand{\Pr}{\P}
\title{\Large \bfseries Gaussian Width of Convex Sets via Integral Decompositions, Projections, and the Distribution of Intrinsic Volumes}

\author[1,2]{Reese Pathak}
\author[1]{Nikita Zhivotovskiy}
\affil[1]{Department of Statistics, University of California, Berkeley}
\affil[2]{School of Operations Research and Information Engineering (ORIE), Cornell University}

\date{\today}

\begin{document}
\maketitle
\begin{abstract}
We revisit the problem of bounding the expected supremum of a canonical Gaussian process indexed by a convex set \(T \subset \R^d\).
We develop two decompositions for the Gaussian width, based on the geometry of the index set.
The first decomposition involves metric projections of Gaussians onto rescaled copies of $T$. The second involves fixed points arising from a quadratically penalized variant of the local width. Neither decomposition directly invokes generic chaining constructions.

Our results make use of recent work in geometric analysis and Gaussian processes. 
The work of Chatterjee~[\emph{Ann.\ Statist.}, %\textbf{42}(6):2340--2381, 
2014] characterizes the behavior of the metric projection of a Gaussian random vector onto rescaled copies of $T$ with a variational problem involving localized Gaussian widths. We use these bounds to develop decompositions of the Gaussian width using the local metric structure of $T$.
Second, we leverage the work of Vitale [\emph{Ann.\ Probab.}, %\textbf{24}(4):2172--2178
1996] to form a connection between the Wills functional (and hence the intrinsic volumes of $T$) and the first terms that appear in our decompositions. 
Finally, invoking recent work by Mourtada~[\emph{J.\ Eur.\ Math.\ Soc.}, 2025] on the logarithm of the Wills functional, we show that the width is controlled by a single, ``peak index'' of the intrinsic volumes.
In the worst case, our bound recovers a local form of the classical Dudley integral.
\end{abstract}

\tableofcontents 

\section{Introduction}

Bounding the Gaussian width of a set $T \subset \R^d$ is an intensively studied problem with applications across probability, geometry, signal processing, statistics, and machine learning. Recall that the Gaussian width of $T$ is defined as
\[
w(T) = \E\sup_{x \in T} \; \langle x, g\rangle,
\]
where $g \sim \Normal{0}{I_d}$ is a standard Gaussian vector in $\R^d$. Talagrand's majorizing measures theorem \cite[Theorem 2.10.1]{Talagrand2021} implies that
\begin{equation}
\label{eq:majorizingmeasure}
w(T) \asymp \gamma_2(T),
\end{equation}
where $a \asymp b$ means that there are absolute constants $c, C > 0$ such that $ca \le b \le Ca$, and
\[
\gamma_2(T)
= \inf_{(\mathcal A_n)} \; \sup_{x \in T}\sum_{n=0}^\infty 2^{n/2}\,\diam\bigl(A_n(x)\bigr).
\]
Here the infimum is taken over all admissible, nested sequences of partitions $(\mathcal A_n)_{n\ge0}$ of $T$ in the sense of \cite[Definition 2.7.1]{Talagrand2021}. In particular, $\mathcal A_0=\{T\}$, $\mathcal A_{n+1}$ refines $\mathcal A_n$, and $|\mathcal A_n|\le 2^{2^n}$. Moreover, $A_n(x)$ denotes the unique element of the partition $\mathcal A_n$ that contains $x$, and $\diam(A)=\sup_{u,v\in A}\|u-v\|_2$.
Despite providing a complete characterization of the Gaussian width, exhibiting an optimal admissible sequence of partitions for a general $T$ can be extremely challenging. In fact, only a few such explicit constructions are known. As Talagrand notes in his monograph \cite{Talagrand2021}, even for ellipsoids, constructing an optimal admissible sequence is ``surprisingly non-trivial.'' These observations motivate developing alternative characterizations of $w(T)$ that avoid explicit generic chaining constructions.

This paper pursues such an alternative by developing connections with the Gaussian sequence model, as studied in statistics. For $\sigma > 0$ and $g \sim \Normal{0}{I_d}$, we observe
\begin{equation}
\label{eq:gaussianseqmodel}
Y = \theta^\star + \sigma g.
\end{equation}
Above, $\theta^\star$ is an unknown target vector belonging to a known parameter set $T \subset \R^d$. Three central objects of study within the Gaussian sequence model are the least squares estimator (LSE), its variance, and the optimal statistical rate of estimation. Geometrically, the LSE---defined in~\cref{eqn:def-LSE}---is the metric projection of $Y$ onto a closed, convex constraint set $T$:
\begin{subequations}
\label{eqn:key-objects-GSM}
\begin{equation}
\label{eqn:def-LSE}
\Pi_T(Y) = \argmin_{\vartheta \in T} \|Y - \vartheta\|_2^2.
\end{equation}
When $T$ is additionally centrally symmetric---\ie if $T = -T$---and $\theta^\star = 0$, the projection is mean-zero and thus the variance of the LSE simplifies to
\begin{equation}
\label{eqn:variance-of-LSE}
\E_{Y \sim \Normal{0}{\sigma^2 I_d}}
\bigl\|\Pi_T(Y) - \E\Pi_T(Y)\bigr\|_2^2 = \sigma^2 \E\bigl\|\Pi_{T/\sigma}( g)\bigr\|_2^2.
\end{equation}
Notably, the latter quantity $\E\bigl\|\Pi_{T/\sigma}( g)\bigr\|_2^2$, in the special case that $T$ is a closed convex cone, is referred to as the \emph{statistical dimension} and was previously studied~\cite{amelunxen2014living,bellec2018sharp,lotz2020concentration,han2023noisy}. Finally, the optimal rate of estimation within the model~\eqref{eq:gaussianseqmodel}, also commonly referred to as the \emph{minimax rate}, is given by
\begin{equation}
\label{eqn:def-minimax-rate-intro}
\big(\eps_\star(\sigma)\big)^2
=
\inf_{\hat \theta} \sup_{\theta \in T}
\E_{Y \sim \Normal{\theta}{\sigma^2 I_d}}
\Big[\|\hat \theta(Y) - \theta\|_2^2\Big].
\end{equation}
\end{subequations}
The three objects in eqns.~\eqref{eqn:key-objects-GSM} will play a key role in our decompositions of the Gaussian width.

It is well-understood that the difficulty of estimating $\theta^\star$ from $Y$ is related to the geometry of $T$.
In particular, if $T$ is closed and convex, then the risk and thus the variance of the least squares estimator (LSE), e.g.,~\eqref{eqn:variance-of-LSE}, are tightly connected with fixed points of the localized Gaussian width, as shown by Chatterjee \cite{Cha14}. Note that fixed-point descriptions of optimal estimation rates are prevalent; see~\cite{birge1993rates,barron1999risk,vandegeer2000empirical,massart2000some}. More generally, the optimal estimation rate~\eqref{eqn:def-minimax-rate-intro} in the model~\eqref{eq:gaussianseqmodel} is essentially determined by the metric structure of $T$, via an appropriately defined local packing number of $T$ (see \Cref{sec:info-and-stat}).

Gaussian width bounds are prevalent in statistical analysis of particular estimators. A main message of this paper, however, is that the ``reverse direction'' is also fruitful: statistical rates and information-theoretic bounds can produce new insights into the behavior of the Gaussian width itself. 
Our approach consists in varying the noise level $\sigma$, and leveraging the profile of the risk of the LSE and the optimal statistical rate, to yield a decomposition of the width itself. Surprisingly, this approach leads to new identities for the width and, in particular, sharp characterizations that bypass generic chaining constructions.

Our analysis of $w(T)$ leverages several deep results regarding canonical Gaussian processes:
\begin{itemize}
    \item \textbf{Wills functional and the distribution of intrinsic volumes.}
    Connections between the Wills functional and suprema of Gaussian processes go back to Vitale \cite{vitale1996wills}.
    More recently, Mourtada \cite{mourtada2025universal} characterized the logarithm of the Wills functional via fixed point equations involving localized Gaussian widths and covering number quantities closely related to those that appear in this paper.
    In our approach, the Wills functional enters as an intermediate geometric object for bounding the Gaussian width: combining our width decomposition with the observation from \cite{mourtada2025universal} that this functional is effectively governed by a single peak intrinsic volume index yields bounds on $w(T)$ in terms of the intrinsic volume profile of suitable rescalings of $T$.
    We exploit these results to obtain a sharp characterization of $w(T)$; see~\Cref{thm:indexstarbound}.
    \item \textbf{Chatterjee’s analysis of Gaussian projections.}
    Chatterjee \cite{Cha14} analyzes the least squares estimator in the Gaussian sequence model and shows, in particular, that fixed points of localized Gaussian width govern LSE performance.
    We exploit two outcomes of this analysis: fixed points control the risk of the LSE across noise levels, moreover, the variance of the LSE can be bounded using only the local metric structure of $T$, bypassing direct appeals to $w(T)$.
    In the worst case, the bounds developed by this technique include the Dudley integral upper bound on $w(T)$, in a fashion that neither invokes  $\gamma_2(T)$ nor any related generic chaining constructions.
\end{itemize}

We note that in the finite-dimensional setting in which $T \subset \R^d$, the $\gamma_2$ functional can be simplified. Specifically, it can be split into two parts, indexed by any integer $p \ge 1$. 
Indeed, for any admissible sequence of partitions $(\mathcal A_n)$,
\[
\gamma_2(T)
\leq
\underbrace{\sup_{x \in T} \sum_{n=0}^{p-1} 2^{n/2}\,\diam\bigl( A_n(x)\bigr)}_{\text{``head'' part}} 
\; + \; 
\underbrace{\sup_{x \in T} \sum_{n=p}^{\infty} 2^{n/2}\,\diam\bigl( A_n(x)\bigr)}_{\text{``tail'' part}}
,
\]
where $A_n(x) \in \mathcal{A}_n$ is the unique set containing $x$. 
For bounded sets in $\R^d$, and for $p$ with $2^p \gtrsim d\log d$, the ``tail'' part can be controlled by standard volumetric and entropy estimates:
\[
 \inf\limits_{(\mathcal{A}_n)}\sup_{x \in T} \sum_{n \ge p} 2^{n/2}\,\diam\bigl(A_n(x)\bigr)
\lesssim  \diam(T)
\lesssim w(T).
\]
One interpretation of our main results, \Cref{thm:decomp-of-gauss-width,thm:decomp-via-projections,thm:indexstarbound}, is that in many cases, it is possible to make earlier, geometry-driven splits, with $p \ll \log_2(d\log d)$. The 
``tail'' part is then controlled by geometric or information-theoretic tools, such as the intrinsic volumes or the Wills functional structure. Note that this is precisely in the regime where covering numbers alone do not correctly capture the scale of the Gaussian width. The ``head'' part is then correspondingly handled by statistical risk bounds, for instance from the LSE or the minimax rate, which are never worse than the classical Dudley entropy integral estimates.
A comparison of our results with generic chaining is provided in \Cref{sec:generic-chaining}.

For the remainder of the paper, we primarily focus on the case where $T$ is a nonempty, compact, convex set containing the origin. In many places we additionally assume that $T$ is centrally symmetric\footnote{Throughout, we say that $T\subset\R^d$ is \emph{centrally symmetric} if $T=-T$.}. This is because this case essentially subsumes all others; for instance, by standard properties of Gaussian width \cite[Proposition 7.5.2]{vershynin2018high}, for any bounded set $T \subset \R^d$,
\[
w(T) = w(\operatorname{cl\,conv}(T)) = \frac{1}{2} \, w(\operatorname{cl\,conv}(T-T)),
\]
with $\operatorname{cl \, conv}(T-T)$ being convex, closed, and centrally symmetric. Additionally, $w(T)$ is finite if and only if $T$ is bounded. 
Although the majorizing measures theorem does not directly invoke convexity, our approach is phrased in terms of convexity, which we exploit to obtain simplifications via the resulting geometry of metric projections and risk identities in the Gaussian sequence model.
Notably, this is not the first paper to leverage convexity to obtain substitutes for the generic chaining. Recent related work includes~\cite{van2018chaining,van2018chainingtwo}, which develop contraction and interpolation principles that yield entropy-based upper bounds on chaining functionals in terms of local ``neighborhoods,'' offering an alternative to constructing admissible partitions.

As a byproduct of our analysis, we obtain several results that may be of independent interest:
\begin{itemize}
\item \textbf{Local and global Dudley and Sudakov bounds are equivalent up to constants.}
A basic lesson from statistical estimation is that, due to the difference between ``local'' and ``global'' forms of the packing and covering entropies, optimal estimation generally depends sensitively on which form is being used to construct estimators or to establish lower bounds~\cite{bshouty2009using,Mendelson2017LocalGlobal,zhivotovskiy2018localization,Ney23}. In contrast, for Dudley-type upper bounds and Sudakov-type lower bounds the resulting expressions remain unchanged, up to universal constants, when global entropy is replaced by its local analogue. In fact, this holds without any structural assumptions on $T$; see~\Cref{sec:localandglobal}.

\item \textbf{Process lower bounds via information-theoretic tools.}
Combining our local-global comparison for Sudakov minoration with Fano-type arguments, we obtain information-theoretic proofs of both the classical and dual Sudakov inequalities; see \Cref{sec:sudakov}. While for the dual version we additionally use a Gaussian shift-type result, namely Anderson's lemma, for the proof of the primal version no Gaussian comparison inequalities are used; the Gaussian input enters only through bounding the Kullback-Leibler divergence for Gaussian measures. In fact, this observation also leads to a simplified proof and generalized statement of the lower bound in the majorizing measures theorem; see~\Cref{rem:KL-properties,rem:generalized-mmt}. These approaches may be of interest in extending such minorations and process lower bounds to the non-Gaussian setting, as studied in~\cite{Latala2014Sudakov,mmn2019,bednorz2023sudakov,liu2022minoration}. Related, but different, work on information-theoretic approaches to Sudakov inequalities appeared in~\cite{liu2025simple,liu2022minoration}.

\item \textbf{A statistical explanation of when Dudley is loose.}
We relate the potential suboptimality of the Dudley entropy integral to the discrepancy between the risk of the LSE at a fixed target (in particular $\theta^\star=0$ for centrally symmetric bodies) and the minimax risk $\eps_\star(\sigma)^2$ at the worst target in $T$; see \Cref{thm:upper-bound-minimax}. Of some statistical interest, we show that the LSE over a suitably chosen global covering is already minimax rate-optimal.
\end{itemize}

As with generic chaining, particular examples pose their own challenges; our methods are not always the simplest. Nonetheless, we illustrate some of our approaches in two basic examples: the standard crosspolytope (see~\Cref{sec:examples}) and an ellipsoid (see~\Cref{rem:width-of-the-ellipse}), both in $\R^d$.
Although the Dudley entropy integral gives an inaccurate estimate for these two bodies, our approaches are able to produce sharp estimates. 

\section{Decompositions of suprema of stochastic processes}

Our first main result is a decomposition of the suprema of stochastic processes. For a random vector 
$\xi \in \R^\dimension$ we define for a nonempty set $T \subset \R^\dimension$,
\[
w_\xi(T) = \E_{\xi}\Big[ \sup_{x \in T} \; \langle \xi, x \rangle \Big].
\]
We will refer to this as the $\xi$-width of the set $T$. In the case $\xi \sim \Normal{0}{I_\dimension}$, this quantity is typically referred to as the Gaussian width (or complexity) of the set $T$. In this case we simply write $w(T)$. 

We often make the assumption that $T$ is closed, convex, and contains the origin, which is essentially no restriction.

\begin{remark}[Convexity and containment of the origin]
The assumption that $T$ is convex and contains the origin is essentially without loss of generality. Indeed, for nonempty $T \subset \R^d$, define $T_0 = \mathrm{cl\,conv}(T - t_0)$ where $t_0 \in T$ is fixed. Then, we clearly have 
\[
w_\xi(T) = w_\xi(T_0) + \E_\xi~\langle t_0, \xi \rangle,
\]
provided that $\E \, \langle t_0, \xi \rangle > -\infty$. Hence, in this way, it is typically possible to reduce to the setting in which the process is indexed by a closed, convex set $T_0$ containing the origin.
\end{remark}

Our first decomposition involves an auxiliary variational problem, indexed by $\sigma > 0$:
\begin{equation}
\label{def:chat-fixed-point}
\ChatFixed(\sigma) 
\defn 
\argmax_{r \geq 0} \Big\{\, 
w_\xi\big(T \cap r B^d_2\big) 
- 
\frac{r^2}{2\sigma} \, \Big\}.
\end{equation}
It is not \emph{a priori} clear that $r(\sigma)$ is well-defined; hence, we will tacitly assume that this is the case in the remainder of the paper. The next result characterizes precisely when this occurs. 

\begin{proposition} 
\label{prop:fixed-points-exist-unique}
For any random vector $\xi \in \R^\dimension$, and any convex set $T \subset \R^\dimension$ containing the 
origin, the following are equivalent: 
\begin{enumerate}[label=(\roman*)]
\item 
\label{item:existence-and-uniqueness}
the fixed point $r(\sigma)$ exists and is 
unique for all $\sigma > 0$; and, 
\item 
\label{item:finiteness-at-one-radius}
there exists $r_0 > 0$ such that 
$w_\xi(T \cap r_0 B^d_2) \in [0, \infty)$; and, 
\item 
\label{item:finiteness-at-all-radii}
for every $r > 0$, it holds that 
$w_\xi(T \cap r B^d_2) \in [0, \infty)$.
\end{enumerate}
\end{proposition} 
\begin{proof}
We use the shorthand $\omega(r) = w_\xi(T \cap r B^d_2)$. Note that $\omega(r) \geq \omega(0) = 0$ 
for all $r \geq 0$.
Assume~\ref{item:finiteness-at-one-radius}. 
Then, for $r \geq r_0$, 
we have from the inclusion $T \cap r B^d_2 
\subset \tfrac{r}{r_0}(T\cap r_0 B^d_2)$ that 
$\omega(r) \leq \tfrac{r}{r_0} \omega(r_0) < \infty$. On the other hand, 
for $r\leq r_0$, the inclusion $T \cap r B^d_2 \subset T \cap r_0 B^d_2$ implies $\omega(r) \leq \omega(r_0)$, thereby establishing~\ref{item:finiteness-at-all-radii}. Assume~\ref{item:finiteness-at-all-radii} holds. Then, 
$r \mapsto \omega(r) - \tfrac{r^2}{2\sigma}$ is a 
finite, $\tfrac{1}{\sigma}$-strongly concave function 
on $[0, \infty)$ and hence it achieves its maximum uniquely, yielding~\ref{item:existence-and-uniqueness}. On the other hand, assume~\ref{item:finiteness-at-all-radii} fails. That is, for some $r_0 > 0$, it holds that $\omega(r_0) = \infty$. Then, $\omega(r) = \infty$ for all $r \geq r_0$, and hence $r(\sigma)$ is not unique, and therefore~\ref{item:existence-and-uniqueness} fails, as needed.
\end{proof}

Note that, although we assume that $r(\sigma)$ exists and is unique for all $\sigma > 0$, a simpler sufficient condition is that $\E \|\xi\|_2 < \infty$, which, for instance, holds in the Gaussian setting.
\begin{theorem} [Decomposition via fixed points]
\label{thm:decomp-of-gauss-width}
Let $\xi$ denote a random vector in $\R^\dimension$ and let $T \subset \R^\dimension$ be a convex set containing the origin.\footnote{In particular, $0 \in T$ ensures $w_\xi(T) \in \R_+ \cup \{+\infty\}$.} Assume that $r(\sigma)$ exists and is unique for all $\sigma > 0$. Then, for each $\sigma > 0$,
\begin{subequations}
\begin{equation}
\label{eqn:fp-decomposition-sig-pos}
w_{\xi}(T) = \Big\{ w_\xi\big(T \cap \ChatFixed(\sigma) B^d_2\big) - \frac{\ChatFixed(\sigma)^2}{2\sigma} \Big\} + \frac{1}{2}\int_{\sigma}^{\infty}\frac{r(\nu)^2}{\nu^2}  \ud \nu.
\end{equation}
In particular, it holds that 
\begin{equation}
\label{eqn:fp-decomposition-sig-zero}
w_{\xi}(T) = \frac{1}{2}\int_{0}^{\infty}\frac{r(\nu)^2}{\nu^2} \, \ud \nu.
\end{equation}
\end{subequations}
\end{theorem} 
The proof is deferred to~\Cref{sec:proof-fixed-point-decomposition}.
At this stage, \Cref{thm:decomp-of-gauss-width} should be viewed as a general decomposition identity: it represents the expected supremum $w_\xi(T)$ via a family of ``local'' variational problems, indexed by the scalar parameter $\nu > 0$.

Decompositions similar to those in~\Cref{thm:decomp-of-gauss-width} appeared previously in empirical process theory, mainly via the so-called peeling device, due to K. Alexander~\cite{alexander1987rates} and popularized in statistical literature~\cite{vandegeer2000empirical, massart2000some, gine2006concentration}. The original peeling device splits the index set into certain level sets and the process is computed for each level separately, upper bounding the expected supremum (a concrete example of this is given in~\Cref{lem:peeling}). These peeling decompositions, in general, are not sharp.
In contrast, the decompositions in~\Cref{thm:decomp-of-gauss-width} crucially exploit the convexity of the index set, using the fixed points $r(\sigma)$ to define level sets which depend on $T$, and lead to an exact decomposition of the underlying stochastic process.  

Specializing to the case where $\xi$ is Gaussian, one convenient (although, not in general tight) interpretation of~\Cref{thm:decomp-of-gauss-width} is via the following form of the Dudley integral upper bound. For a centrally symmetric convex body \(T\subset\R^\dimension\), for any \(\delta>0\) (see, e.g., \cite[Theorem 5.22]{Wai19} for explicit constants):
\[
w(T)
\;\le\;
2\,w\Big(T\cap \tfrac{\delta}{2}B_2^\dimension\Big)
\;+\;
16\int_{\delta/4}^{\infty}
\sqrt{\log N\big(T,\varepsilon B_2^\dimension\big)} \, \ud\varepsilon,
\]
where \(N\big(T,\varepsilon B_2^\dimension\big)\) is the Euclidean covering number of \(T\) at scale \(\varepsilon\).
Taking \(\delta=2r(\sigma)\) gives
\[
w(T)
\;\le\;
2\underbrace{\left(w\Big(T\cap r(\sigma)B_2^\dimension\Big) - \frac{r(\sigma)^2}{2\sigma}\right)}_{\text{First term in~\cref{eqn:fp-decomposition-sig-pos}}}
\;+\;
\underbrace{\frac{r(\sigma)^2}{\sigma} \;+\; 16\int_{r(\sigma)/2}^{\infty}
\sqrt{\log N\big(T,\varepsilon B_2^\dimension\big)} \, \ud\varepsilon}_{\text{Loose upper bound on }\;\frac{1}{2}\int\nolimits_{\sigma}^{\infty}\frac{r(\nu)^2}{\nu^2} \, \ud \nu}.
\]

\medskip 

Our next decomposition involves the metric projections onto a nonempty, closed, convex set $T \subset \R^d$. We recall that it is defined at $x\in\R^d$ by
\[
\Pi_T(x) \defn \argmin_{t \in T} \, \|x - t\|^2_2.
\]

\begin{theorem} [Decomposition via metric projections]
\label{thm:decomp-via-projections}
Let $\xi$ denote a random vector in $\R^\dimension$ and let $T \subset \R^\dimension$ be a closed, convex set containing the origin. Then, 
for every $\sigma > 0$, it holds that
\begin{subequations}
\begin{equation}
\label{eqn:projection-decomposition-sig-pos}
w_{\xi}(T) = \E \bigg[\,\sup_{t \in T} \; \Big\{\langle t, \xi\rangle-\frac{\|t\|_2^2}{2\sigma} \Big\}\, \bigg] \; + \; 
\frac{1}{2}
\int_{\sigma}^{\infty} 
\E\|\Pi_{T/\nu}(\xi)\|_2^2 \, \ud\nu.
\end{equation}
In particular, it holds that 
\begin{equation}
\label{eqn:projection-decomposition-sig-zero}
w_{\xi}(T) = \half \int_{0}^{\infty} 
\E\|\Pi_{T/\nu}(\xi)\|_2^2 \, \ud\nu.
\end{equation}
\end{subequations}
\end{theorem} 
\noindent The proof is presented in~\Cref{sec:proof-projection-decomposition-thm}.

\medskip

The results appearing in~\Cref{thm:decomp-of-gauss-width,thm:decomp-via-projections}, though related, are not the same. Indeed, the first terms appearing on the right-hand sides in~\cref{eqn:fp-decomposition-sig-pos,eqn:projection-decomposition-sig-pos}
are related as follows:
\begin{equation}\label{eqn:upper-bound-between-first-terms}
w_\xi\big(T \cap \ChatFixed(\sigma) B^d_2\big) - \frac{\ChatFixed(\sigma)^2}{2\sigma} \; \le \; 
 \E \bigg[\,\sup_{t \in T} \; \Big\{\langle t, \xi\rangle-\frac{\|t\|_2^2}{2\sigma} \Big\}\, \bigg].
\end{equation}
Notably, in the Gaussian case, for centrally symmetric convex $T$ the bound can be reversed:
\begin{equation}
\label{eq:mourtadapeeling}
w\big(T \cap \ChatFixed(\sigma) B^d_2\big) - \frac{\ChatFixed(\sigma)^2}{2\sigma} \gtrsim \E\left[\sup\limits_{x \in T }\left(\langle x, g\rangle-\frac{\|x\|^2}{2\sigma}\right)\right].
\end{equation}
Above, $g \sim \Normal{0}{I_d}$. The proof of~\cref{eq:mourtadapeeling}, which appears in~\Cref{sec:equivalenceoffirstterms}, exploits Gaussian concentration and a peeling argument; see also the recent paper~\cite{mourtada2025universal} for a related result. 

A similar comparison can be drawn between the second terms in~\cref{eqn:fp-decomposition-sig-pos,eqn:projection-decomposition-sig-pos}. Combining~\Cref{thm:decomp-of-gauss-width,thm:decomp-via-projections} and the inequality~\eqref{eqn:upper-bound-between-first-terms}, we immediately obtain 
\begin{equation}
\label{eq:upperboundonintegrals}
\int_{\sigma}^{\infty} \, 
\E \|\Pi_{T/\nu}(\xi)\|_2^2 \, \ud \nu \;
\le\;
\int_{\sigma}^{\infty}\frac{r(\nu)^2}{\nu^2}\, \ud\nu.
\end{equation}
Again, in the Gaussian case a more precise characterization is available. Letting $\diam(T)$ denote the Euclidean diameter of $T$, it can be shown (see~\Cref{sec:chaterjeecomparison}) that for any $\sigma \ge 0$,
\begin{equation}
\label{eq:chattejeeintegrals}
\int_{\sigma}^{\infty} \, 
\E \|\Pi_{T/\nu}(g)\|_2^2 \, \ud \nu + \diam(T) \asymp \int_{\sigma}^{\infty}\frac{r(\nu)^2}{\nu^2} \, \ud\nu + \diam(T), 
\end{equation}
where $g \sim \Normal{0}{I_d}$.
This follows via a result of Chatterjee~\cite{Cha14}, which implies that
\begin{equation}
\label{eq:chatterjeeidentity}
r(\sigma)^2\ \asymp\ \E\|\Pi_{T}(\sigma g)\|^2,
\qquad\text{whenever } r(\sigma)\gtrsim \sigma,
\end{equation}
for any centrally symmetric convex body $T \subset \R^d$. It should be noted that the above relation does not hold, in general, without the constraint $r(\sigma) \gtrsim \sigma$.

\medskip

For the rest of the paper we will primarily focus on the Gaussian case.
Summarizing the discussion above, the results in \cref{eq:mourtadapeeling,eq:chattejeeintegrals} mean that, disregarding multiplicative constant factors, in the Gaussian case the two decompositions of \Cref{thm:decomp-of-gauss-width} and \Cref{thm:decomp-via-projections} are essentially equivalent. Nonetheless, we typically exploit one of the decompositions in order to obtain sharper leading constants and more streamlined arguments.

\subsection{Gaussian width via the index of the largest intrinsic volume}

Another consequence of \Cref{thm:decomp-of-gauss-width} links the Gaussian width to a purely geometric feature of the convex body, complementing the Gaussian projection structure. In this section convexity is sufficient; central symmetry is not needed. Let us first recall the classical definition of intrinsic volumes. For a convex body $T \subset \R^d$, one can use the Steiner formula
\begin{equation}
\label{eq:intrvolumesformulas}
\operatorname{Vol}_d(T + rB_2^d) = \sum_{i = 0}^d \kappa_{d - i}\, V_i(T)\, r^{d - i},
\end{equation}
where $\operatorname{Vol}_d(\cdot)$ denotes the volume in $\R^d$, $r > 0$, $B_2^d$ denotes the Euclidean unit ball in $\R^d$, $\kappa_j = \operatorname{Vol}_j(B_2^j)$ and $V_j(T)$ are the \emph{intrinsic volumes} of $T$. Some of the intrinsic volumes have a particularly simple structure: namely,
\begin{equation}
\label{eq:intrvolumes}
V_0(T) = 1, \quad V_1(T) = \sqrt{2\pi}\,w(T),\quad V_{d - 1}(T) = \tfrac{1}{2}\operatorname{Sf}_{d-1}(\partial T),\quad V_d(T) = \operatorname{Vol}_d(T),
\end{equation}
where $\operatorname{Sf}_{d-1}(\partial T)$ denotes the surface area of $T$. The direct connection between $V_{1}(T)$ and $w(T)$ due to Sudakov~\cite{Sudakov1976} reduces our problem to understanding the first intrinsic volume.

A scale-free aspect of the Gaussian width that will be central in the context of \Cref{thm:decomp-of-gauss-width} is closely related to the distribution of the intrinsic volumes $\{V_i(T)\}_{i=0}^d$. This sequence is unimodal~\cite{lotz2020concentration, AravindaMarsigliettiMelbourne2022ULC}.\footnote{A closely related question, namely the distribution of the normalized intrinsic volumes, is studied in~\cite{lotz2020concentration}.}  We will only need the integer index of the largest intrinsic volume; namely, we set
\begin{equation}
\label{eq:istardef}
\PeakIndex{\sigma} \in \argmax_{i \in \{0, 1, \ldots, d\}}\left\{ V_i\left(\frac{T}{\sigma\sqrt{2\pi}}\right)\right\},
\end{equation}
which will be called the \emph{peak intrinsic index}. We assume in what follows that $\PeakIndex{\sigma}$ is the smallest maximizer if there is more than one maximizer. 
Of course, finding the peak intrinsic index is a nontrivial problem in general. However, from \eqref{eq:intrvolumes} and unimodality we have the following limiting cases:
\begin{equation}
\label{eq:peakindexprops}
\PeakIndex{\sigma} = d \quad \text{iff}\quad \sigma < \sqrt{\frac{2}{\pi}}\frac{\operatorname{Vol}_d(T)}{\operatorname{Sf}_{d-1}(\partial T)},
\qquad\text{and}\qquad
\PeakIndex{\sigma} = 0 \quad \text{iff}\quad \sigma \ge w(T).
\end{equation}
Thus, the interesting regime $\PeakIndex{\sigma}>0$ corresponds to $\sigma \in (0, w(T))$. It is straightforward to argue that $\sigma \mapsto \PeakIndex{\sigma}$ is a nonincreasing, integer-valued map that changes only at scales where two consecutive intrinsic volumes of $\frac{T}{\sigma\sqrt{2\pi}}$ are equal.

Before stating the main result, let us recall a basic comparison between the Gaussian width of a set and its diameter. For a bounded set $T \subset \R^d$, we have (see \cite[Proposition~7.5.2]{vershynin2018high})
\begin{equation}
\label{eq:vershyninrelation}
    \frac{1}{\sqrt{2\pi}}\diam(T)\le w(T) \le \sqrt{d}\,\diam(T).
\end{equation}
Based on these bounds, it is natural to wonder if, up to the diameter, 
$[w(T)]^2$ behaves as some form of an ``effective dimension,'' which takes values between a constant and $d$, depending on the geometry of the underlying set. Our next result shows
that the mode of the intrinsic volumes plays precisely this role.

\begin{theorem}
\label{thm:indexstarbound}
For a compact convex set $T$ in $\R^d$ containing the origin, and for any $\sigma \in (0, w(T))$, it holds that
\[
w(T) \asymp \sigma \PeakIndex{\sigma}+\int_{\sigma}^{\infty}\E\|\Pi_{T/\nu}(g)\|_2^2 \, \, \ud\nu, \quad \text{and} \quad w(T) \asymp \sigma \PeakIndex{\sigma}+\int_{\sigma}^{\infty}\frac{r(\nu)^2}{\nu^2}\ud\nu,
\]
where the peak intrinsic index $\PeakIndex{\sigma}$ is given by \eqref{eq:istardef}. In particular, assuming $\diam(T) \neq 0$, define
\begin{equation}
\label{eq:istar}
i^\star \in \argmax_{i \in \{1,\ldots,d\}}\left\{ V_i\left(\frac{T}{\diam(T)}\right)\right\}.
\end{equation}
Then, it holds that
\begin{equation}
\label{eq:istarbound}
w(T) \asymp i^\star\diam(T).
\end{equation}
\end{theorem}
The full proof of this result is deferred to \Cref{sec:wills}. 
The upper bound in \Cref{thm:indexstarbound} only requires an approximate index-location $\PeakIndex{\sigma}$, that is, up to multiplicative constant factors. 

\begin{remark}[Probabilistic interpretation of~\Cref{thm:indexstarbound}]
The distribution of intrinsic volumes was studied in~\cite{lotz2020concentration}. There, the authors consider the discrete random variable $Z_T$ supported on $\{0,\ldots,d\}$ with probability mass function
\[
\Pr\big\{Z_T=i\big\}=\frac{V_i(T)}{W(T)},\quad \mbox{for}~i=0,1,\ldots, d.
\]
The work \cite{lotz2020concentration} establishes concentration of $Z_T$ and analyzes particular choices of the set $T$. On the other hand, \Cref{thm:indexstarbound} 
shows that $w(T)$ depends on the {mode} of $Z_T$. In this sense, more fine-grained information regarding the behavior of $Z_T$ is not necessary to obtain a characterization of the width.
\end{remark}

From the proof of~\Cref{thm:indexstarbound}, the peak intrinsic index $i^\star$ in~\eqref{eq:istar} satisfies
\[
i^\star 
\le \left\lfloor \sqrt{2\pi d}\right\rfloor.
\]
 The next result, which follows quickly from~\cref{eq:istarbound}, shows that when $T$ is additionally in the correct position, it must at least scale polylogarithmically in dimension. 
\begin{corollary}
\label{cor:T-lowner}
Suppose that $T \subset \R^d$ is a centrally symmetric convex body which is in the Löwner position (\ie the minimum volume enclosing ellipsoid is the ball $\rad(T) B^d_2$). Then, the peak intrinsic index $i^\star$ satisfies
\[
\sqrt{\log(\e d)} \lesssim i^\star \lesssim \sqrt{d}.
\]
\end{corollary}

The proof of \Cref{cor:T-lowner} is presented in \Cref{sec:T-lowner}. Note that both inequalities are sharp: by considering the case when $T = B^d_1$ or $T = B^d_2$, up to scaling. 

\begin{remark}[Failure of~\Cref{cor:T-lowner} without position]
The assumption of position is critical for the lower bound in Corollary~\ref{cor:T-lowner}. For instance, consider the case of the ellipsoid $\cE_\lambda = \{x \in \R^d : x_1^2 + \tfrac{d-1}{\lambda}(x_2^2 + \cdots + x_d^2)\leq 1\}$ for $\lambda > 0$. It is not hard to verify that $\rad(\cE_\lambda) = \max\{1, \sqrt{\lambda/(d-1)}\}$, while $w(\cE_\lambda) \asymp\sqrt{1 + \lambda}$, and hence by~\Cref{thm:indexstarbound}, we have $i^\star \asymp 1$ as $\lambda \to 0^+$.
\end{remark}

\subsection{Proof of~\Cref{thm:decomp-of-gauss-width}}
\label{sec:proof-fixed-point-decomposition}

The remainder of this section completes the proofs of \Cref{thm:decomp-of-gauss-width}; applications and further developments of the identity appear in subsequent sections. Throughout, we make use of the following shorthand notation for any $\sigma > 0$:
\[
\psi(\sigma) = \sup_{r \geq 0 }\Big\{\,
w_\xi(T \cap r B^d_2) - \frac{r^2}{2\sigma}
\,\Big\}.
\]
The following lemma, proved in~\Cref{sec:proof-lem-smoothed-local-width}, is a consequence of convex analysis. 
\begin{lemma}
\label{lem:smoothed-local-width}
Suppose $T$ is a convex set containing the origin. Then, 
\begin{enumerate}[label=(\roman*)]
\item \label{item:limit-at-infinity-smoothed-width}
the limit relation $\lim_{\sigma \to \infty} \psi(\sigma) = w_{\xi}(T)$ holds; and, 
\item \label{item:limit-at-zero-smoothed-width} the limit relation $\lim_{\sigma \to 0^+} \psi(\sigma) = 0$ holds; and, 
\item \label{item:monotone-r-sigma}
the map $\sigma \mapsto r(\sigma)$ is nondecreasing on $(0, \infty)$; and,
\item \label{item:derivative-of-smooth-width}
the function $\psi$ is differentiable on $(0, \infty)$ with $\psi'(\sigma) = \frac{r^2(\sigma)}{2\sigma^2}$.
\end{enumerate}
\end{lemma}
\noindent By the fundamental theorem of calculus and~\Cref{lem:smoothed-local-width}\ref{item:limit-at-infinity-smoothed-width}~and~\ref{item:derivative-of-smooth-width}, for any $\sigma > 0$ it holds that 
\[
w_{\xi}(T) = \psi(\sigma) + \lim_{\nu \to \infty}\big[\, \psi(\nu) - \psi(\sigma)\,\big] = 
\psi(\sigma) + \half \int_\sigma^\infty \frac{r^2(\nu)}{\nu^2} \, \ud \nu,
\]
thereby establishing~\eqref{eqn:fp-decomposition-sig-pos}. Moreover, 
if $0 \in T$, we take the limit as $\sigma \to 0^+$ and using~\Cref{lem:smoothed-local-width}\ref{item:limit-at-zero-smoothed-width}, 
we obtain~\eqref{eqn:fp-decomposition-sig-zero}.

\subsection{Proof of~\Cref{thm:decomp-via-projections}}
\label{sec:proof-projection-decomposition-thm}
For any $x \in \R^d$ and $\sigma > 0$, we define the shorthand notations 
\[
h(x) \equiv h_T(x) = \sup_{t \in T} \, \langle x, t \rangle, 
\quad
h_\sigma(x) = 
\sup_{t \in T} \, \Big\{ \langle x, t \rangle - \frac{\|t\|_2^2}{2\sigma} \Big\}, 
\quad \mbox{and} \quad 
\dot{h}_\sigma(x)=\frac{\ud}{\ud \nu} h_\nu(x) \Big|_{\nu = \sigma}.
\]
The following lemma, proved in~\Cref{sec:proof-lem-smoothed-support-function}, can be established using convex analysis. 
\begin{lemma}
Suppose $T$ is a closed, convex set containing the origin. Then for any $x \in \R^d$,
\begin{enumerate}[label=(\roman*)]
\item 
\label{item:derivative-of-smoothed-support}
the map $\sigma \mapsto h_\sigma(x)$ is differentiable on $(0, \infty)$, with $\dot{h}_\sigma(x) = \tfrac{1}{2} \|\Pi_{T/\sigma}(x)\|_2^2$; and,
\item 
\label{item:limit-relation-support-inf}
the limit relation $\lim_{\sigma \to \infty} h_\sigma(x) = h(x)$ holds; and, 
\item 
\label{item:limit-relation-support-0}
the limit relation 
$\lim_{\sigma \to 0^+} h_\sigma(x) = 0$ holds. 
\end{enumerate}
\label{lem:smoothed-support-function}
\end{lemma}
\noindent By the fundamental theorem of calculus and~\Cref{lem:smoothed-support-function}\ref{item:derivative-of-smoothed-support}, for each $x \in \R^d$ and any $\eta \geq \sigma > 0$,
\[
h_\eta(x) = h_\sigma(x) +  \int_{\sigma}^{\eta} \dot{h}_\nu(x) \, \ud \nu = h_\sigma(x) + \half \int_{\sigma}^{\eta} \|\Pi_{T/\nu}(x)\|_2^2 \, \ud \nu. 
\]
Passing to the limit as $\eta \to \infty$ (resp.,   $\sigma \to 0^+$) and applying~\Cref{lem:smoothed-support-function}\ref{item:limit-relation-support-inf}
(resp., \Cref{lem:smoothed-support-function}\ref{item:limit-relation-support-0}), it holds for any $x \in \R^d$ and every $\sigma > 0$  that 
\begin{subequations}
\label{subeqns:pointwise-relations-support}
\begin{align}
\label{eqn:pointwise-integral-with-sigma}
h(x) &= h_\sigma(x) + \frac{1}{2} \int_\sigma^\infty 
\|\Pi_{T/\nu}(x)\|_2^2 \, \ud \nu, 
\quad\mbox{and,} \\
\label{eqn:pointwise-integral-no-sigma}
h(x) &= \frac{1}{2} \int_0^\infty 
\|\Pi_{T/\nu}(x)\|_2^2 \, \ud \nu, 
\quad \mbox{provided $0 \in T$}.
\end{align}
\end{subequations}
The claim~\eqref{eqn:projection-decomposition-sig-pos} (resp., \eqref{eqn:projection-decomposition-sig-zero}) now follows by taking $x = \xi$ in the pointwise relation~\eqref{eqn:pointwise-integral-with-sigma} (resp., \eqref{eqn:pointwise-integral-no-sigma}), integrating with respect to the law of $\xi$, and applying Fubini's theorem. \qed

\subsection{Proofs of~\Cref{lem:smoothed-local-width,lem:smoothed-support-function}}

For the proofs of the following lemmas, we need
to recall the notion of \emph{lower-$C^1$ continuity}~\cite[Definition 10.29]{RocWets98}. A map $f \colon (0, \infty) \to \R$ 
is said to be \emph{lower-$C^1$}, provided that 
for each $\sigma > 0$, there exists an open interval $I_\sigma \subset (0, \infty)$, a compact space $T_\sigma$, and map $g \colon T_\sigma \times I_\sigma \to \R$ for which it holds that: 
\begin{enumerate}[label=(\alph*)]
\item for every $\nu \in I_\sigma$ we 
can write 
\begin{equation}\label{eqn:representation-of-lower-C1}
f(\nu) = \max_{t \in T_\sigma} g(t, \nu); \qquad \mbox{and,}
\end{equation}
\item for each $t \in T_\sigma$, the map $g(t, \cdot) \colon I_\sigma \to \R$ is of class $C^1$; and, 
\item the mappings $(t, \sigma) \mapsto g(t, \sigma)$ and $(t, \sigma) \mapsto \partial_\sigma g(t, \sigma)$
are continuous on $T_\sigma \times I_\sigma$.
\end{enumerate}

We then recall the following result, which is a form of an envelope theorem. It is a special case of~\cite[Theorem~10.30]{RocWets98}.
\begin{lemma} [{\cite{RocWets98}}]
\label{lem:differentiability-lower-c1}
Let $f \colon (0, \infty) \to \R$ be lower-$C^1$. 
If, in the representation~\eqref{eqn:representation-of-lower-C1}, the maximum is attained uniquely at some $t^\star_\sigma \in T_\sigma$, then $f'(\sigma) = \partial_\sigma g(t^\star_\sigma, \sigma)$.
\end{lemma}

\subsubsection{Proof of Lemma~\ref{lem:smoothed-local-width}}
\label{sec:proof-lem-smoothed-local-width}

For $r > 0$, we use the following shorthands for the local width and the corresponding random variable,
\[
\omega(r) = w_\xi(T \cap r B^d_2), \quad \mbox{and} \quad 
X_r = h_{T \cap r B^d_2}(\xi).
\]

\paragraph{Claim~\ref{item:limit-at-infinity-smoothed-width}:}

For every $\sigma > 0$, it clearly holds that 
$\psi(\sigma) \leq w_\xi(T)$. On the other hand, for each $r > 0$, we have 
\[
\lim_{\sigma \to \infty} \psi(\sigma) \geq 
\lim_{\sigma \to \infty} \Big\{ \, 
\omega(r) - \frac{r^2}{2\sigma}\, \Big\} = 
\omega(r). 
\]
Therefore, passing to the limit as $r \to \infty$, 
\[
\lim_{\sigma \to \infty} 
\psi(\sigma) \geq \lim_{r \to \infty} \omega(r)  = \lim_{r\to \infty} \E X_r = w_\xi(T).
\]
The last equality holds by the monotone convergence theorem, as the nonnegative sequence of random variables $X_r$ is nondecreasing in $r$, and $X_r \to h(\xi)$ almost surely as $r \to \infty$.

\paragraph{Claim~\ref{item:limit-at-zero-smoothed-width}:}

Using $0 \in T$, $\psi(\sigma) \geq 0$ for each $\sigma > 0$. Note that $\omega(r) \to 0$ as $r \to 0^+$. Indeed, this follows by applying the dominated convergence theorem to $\{X_r\}_{r \geq 0}$, and the fact that for some $r_0 > 0$, it holds that $\E X_{r_0} <\infty$ under the assumption that $r(\sigma)$ exists for all $\sigma > 0$; see~\Cref{prop:fixed-points-exist-unique}.

Let $\eps > 0$. Then, we fix $r > 0$ such that $\omega(r) \leq \eps$. As $0 \in T$, convexity yields that $T \cap s B^d_2 \subset \tfrac{s}{r} (T \cap r B^d_2)$ for all $s \geq r$. 
Consequently, for all sufficiently small $\sigma > 0$, 
\begin{equation}
\psi(\sigma) 
= \sup_{s \geq 0} \, \Big\{\, \omega(s) - \frac{s^2}{2\sigma} \,\Big\}
\leq \max\Big\{\omega(r),\; \sup_{s \geq r} \Big( s\frac{\omega(r)}{r} - \frac{s^2}{2\sigma}\Big)\Big\} \leq 
\max\Big\{\omega(r),\; \frac{\sigma}{2} \Big(\frac{\omega(r)}{r}\Big)^2\Big\} \leq \eps,
\end{equation} 
The inequality above completes the proof, as $\eps > 0$ was arbitrary.

\paragraph{Claim~\ref{item:monotone-r-sigma}:}
The monotonicity follows by noting that for $\nu \geq \sigma > 0$ we have, by definition, 
\[
\omega(r(\nu)) - \frac{r^2(\nu)}{2\sigma} \leq 
\omega(r(\sigma)) - \frac{r^2(\sigma)}{2\sigma}
\quad \mbox{and} \quad 
\omega(r(\sigma)) - \frac{r^2(\sigma)}{2\nu} \leq 
\omega(r(\nu)) - \frac{r^2(\nu)}{2\nu}.
\]
Adding these two inequalities yields 
\[
(\sigma^{-1} - \nu^{-1}) r^2(\sigma) \leq 
(\sigma^{-1} - \nu^{-1}) r^2(\nu),
\]
which obviously implies $r(\sigma) \leq r(\nu)$, as required.

\paragraph{Claim~\ref{item:derivative-of-smooth-width}:}

We can write 
\[
\psi(\sigma) = 
\sup_{r \geq 0} g(r, \sigma)
\quad \mbox{where}\quad 
g(r, \sigma) = \omega(r) - \frac{r^2}{2\sigma}.
\]
Let $\dot{g}(r, \sigma) = \partial_\sigma g(r, \sigma) = \frac{r^2}{2\sigma^2}$. Clearly 
$g, \dot{g} \colon \R_+ \times (0, \infty)$ are continuous; the maps $g(r,\cdot)$ are $C^1$ on $(0, \infty)$ for $r \geq 0$. Moreover, by the existence of $r(\sigma)$, we can write $\psi(\sigma) = g(r(\sigma), \sigma)$. 
Let $I_\sigma = (\tfrac{\sigma}{2}, \tfrac{3\sigma}{2})$.
By claim~\ref{item:monotone-r-sigma}, the map $\sigma \mapsto r(\sigma)$ is nondecreasing in $\sigma$. Thus we may take $T_\sigma = [0, r(3\sigma/2)] \subset [0, \infty)$, and we have $\psi(\nu) = \sup_{r \in T_\sigma} g(r, \nu)$ for each $\nu \in I_\sigma$. Hence, the 
claim now follows from~\Cref{lem:differentiability-lower-c1}, which yields 
\[
\psi'(\sigma) = \dot{g}(r(\sigma), \sigma) = \frac{1}{2} \frac{r^2(\sigma)}{\sigma^2},
\]
as needed.

\subsubsection{Proof of~\Cref{lem:smoothed-support-function}}
\label{sec:proof-lem-smoothed-support-function}

\paragraph{Claim~\ref{item:derivative-of-smoothed-support}:}

For a fixed $x \in \R^d$, we can write 
\[
f(\sigma) \equiv h_\sigma(x) = \sup_{t \in T} g(t, \sigma), 
\quad \mbox{where} 
\quad 
g(t, \sigma) \defn \langle t, x \rangle - \frac{\|t\|_2^2}{2\sigma}.
\]
Note that $\dot{g}(t, \sigma) \defn 
\partial_\sigma g(t, \sigma) = \frac{\|t\|_2^2}{2\sigma^2}$. 
Clearly, $g, \dot{g} \colon T \times (0, \infty) \to \R$ are continuous; moreover the maps $g(t, \cdot)$ are $C^1$ on $(0, \infty)$ for $t \in T$. 
Moreover, we may write 
\[
g(t, \sigma) = \frac{\sigma}{2} \|x\|_2^2 - \frac{1}{2\sigma}\|t - \sigma x\|_2^2, 
\quad \mbox{and, therefore} 
\quad 
f(\sigma) = g(t^\star_\sigma, \sigma),~t^\star_\sigma = \Pi_T(\sigma x).
\]
The maximizer is unique, as $T$ is a nonempty, closed and convex set. Thus, now set $\sigma > 0$ and set $I_\sigma = (\tfrac{\sigma}{2}, \tfrac{3\sigma}{2})$. Since projections are nonexpansive, $\|\Pi_T(\nu x) - \Pi_T(0)\|_2 \leq \nu \|x\|_2 \leq \tfrac{3\sigma}{2} \|x\|_2$, for each $\nu \in I_\sigma$. Therefore, if we set 
$R_\sigma = \|\Pi_T(0)\|_2 +\tfrac{3\sigma}{2} \|x\|_2$, then it follows that the representation~\eqref{eqn:representation-of-lower-C1} holds with the compact set $T_\sigma = T \cap R_\sigma B^d_2$, and by the triangle inequality $t^\star_\nu \in T_\sigma$ for each $\nu \in I_\sigma$. Hence, $\sigma \mapsto h_\sigma(x)$ is lower-$C^1$, with a unique maximizer, whence the claim follows by~\Cref{lem:differentiability-lower-c1}: 
\[
f'(\sigma)=\dot{g}(t^\star_\sigma, \sigma)=
\half \frac{\|\Pi_T(\sigma x)\|_2^2}{\sigma^2} = 
\half \|\Pi_{T/\sigma}(x)\|_2^2.
\]

\paragraph{Claim~\ref{item:limit-relation-support-inf}:}

Fix $x \in \R^d$. Clearly $h_\sigma(x) \leq h(x)$ for each $\sigma > 0$. 
On the other hand, for any $t_0 \in T$, 
\[
\lim_{\sigma \to \infty} h_\sigma(x)
\geq 
\lim_{\sigma \to \infty} \Big\{\, \langle x, t_0 \rangle - \frac{\|t_0\|_2^2}{2\sigma}\, \Big\} = \langle x, t_0 \rangle.
\]
Passing to the supremum over $t_0 \in T$ yields $\lim_{\sigma \to \infty} h_\sigma(x) \geq h(x)$, as needed. 

\paragraph{Claim~\ref{item:limit-relation-support-0}:}

Fix $x \in \R^d$. As $0 \in T$, 
$h_\sigma(x) \geq 0$ for all $\sigma > 0$. On the other hand, 
\[
h_\sigma(x) = 
\sup_{t \in T} \Big\{\, \langle x, t\rangle - \frac{\|t\|_2^2}{2\sigma}\, \Big\} 
\leq 
\sup_{t \in \R^d} \Big\{\, \langle x, t\rangle - \frac{\|t\|_2^2}{2\sigma}\, \Big\} 
= \frac{\sigma}{2} \|x\|_2^2.
\]
Passing to the limit as $\sigma \to 0^+$ yields the claim.

\section{Equivalence of~\Cref{thm:decomp-of-gauss-width,thm:decomp-via-projections} in the Gaussian case}

In this section, we compare the decompositions which appeared in \Cref{thm:decomp-of-gauss-width,thm:decomp-via-projections} in the Gaussian case.

\subsection{Equivalence of the first terms of decompositions}
\label{sec:equivalenceoffirstterms}
We start with the following definitions.
\begin{definition}[Critical radius]
\label{def:critical-radius}
    For a convex set $T \subseteq \R^d$ containing the origin, define the \emph{critical radius}  
    \[
    r^\star(\sigma) = \sup\left\{r > 0: w\big(T \cap r B^d_2\big) \ge \frac{r^2}{\sigma}\right\},
    \]
    for each $\sigma > 0$.
\end{definition}

The next definition is a shorthand for the first term which appeared in~\Cref{thm:decomp-of-gauss-width}.

\begin{definition}
\label{def:first-term}
    For a convex set $T \subset \R^d$ containing the origin, we define 
    \[
    \BoundOne(\sigma) = w\big(T \cap \ChatFixed(\sigma) B^d_2\big) - \frac{\ChatFixed(\sigma)^2}{2\sigma} = \sup_{r \geq 0} \left\{\, 
w\big(T \cap r B^d_2\big) 
- 
\frac{r^2}{2\sigma} \right\},
    \]
    for each $\sigma > 0$.
\end{definition}

Note that although both~\Cref{def:critical-radius,def:first-term} depend on the underlying set $T \subset \R^d$, we suppress this from the notation as it will always be clear from context when we use these notions.
Versions of $r^\star(\sigma)$ and $\BoundOne(\sigma)$ play a prominent role in the analysis of least squares estimators in statistical literature (see \cite[Chapter 13]{Wai19} for an account of related results). Note that the critical radius need not satisfy $r^\star(\sigma) \le \rad(T)$, in contrast to the fixed point $r(\sigma)$. The next result quantifies some relations (and in particular~\cref{eq:mourtadapeeling}) between the quantities $r, r^\star$, and $\cT$, more formally.

\begin{proposition}[First terms of decompositions in the Gaussian case]
\label{prop:equivalentfixedpoints}
    Let $T$ be a bounded, nonempty, centrally symmetric, convex set in $\R^d$. Fix any $\sigma > 0$. Then, it holds that
    \begin{equation}
    \label{ineq:elementary-relations-bound-one}
    \BoundOne(\sigma) \le \E\left[\sup\limits_{x \in T }\left(\langle x, g\rangle-\frac{\|x\|^2}{2\sigma}\right)\right] \le 140 \; \frac{(r^\star(\sigma))^2}{\sigma} \le 280 \; \BoundOne(\sigma).
    \end{equation}
    Furthermore, it holds that 
    \begin{equation}
    \label{eq:relationonrs}
    r(\sigma) \le \min\left\{r^\star(2\sigma),\, \rad(T) \right\} \le \min\left\{2\sqrt{\sigma\BoundOne(2\sigma)},\, \rad(T) \right\}.
    \end{equation}
\end{proposition}
The inequalities~\eqref{ineq:elementary-relations-bound-one} show that, up to multiplicative constants, the first terms of \Cref{thm:decomp-of-gauss-width,thm:decomp-via-projections} coincide. The second inequality,~\eqref{eq:relationonrs} should be viewed as a crude relation between the two fixed points $r, r^\star$, which in general cannot be reversed. A significantly sharper form of these bounds will be developed in~\Cref{cor:improvedfixedpointformula}. 
There, we show that the gap between $r(\sigma)$ and $r^\star(\sigma)$ is related to the difference between the Sudakov minoration and the local Gaussian width of $T$.

\begin{proof} 
First, observe that
\begin{equation}
\label{eq:rstarfixedpoint}
r^\star(\sigma) = \sigma\sup\{r > 0: w\big(T/\sigma \cap r B^d_2) \ge r^2\}.
\end{equation}
    We also have the following inequality
    \begin{equation}
\label{eq:boundontauone}
                \BoundOne(\sigma) =  \max_{r > 0}\left\{w\big(T \cap r B^d_2)-\frac{r^2}{2\sigma}\right\} \le \E\left[\sup\limits_{x \in T }\left(\langle x, g\rangle-\frac{\|x\|^2}{2\sigma}\right)\right]
        =\sigma\E\bigg[\sup\limits_{x \in T/\sigma } \Big(\langle x, g\rangle-\frac{\|x\|^2}{2}\Big)\bigg]
    \end{equation}
Now, Proposition 4.1 in \cite{mourtada2025universal} shows via the peeling argument combined with the Gaussian concentration inequality that for a convex set $T/\sigma$ it holds that
\[
\E\bigg[\sup_{x \in T/\sigma }\; \Big(\langle x, g\rangle-\frac{\|x\|^2}{2}\Big)\bigg] \le 140\left(\sup\{r > 0: w\big(T/\sigma \cap r B^d_2) \ge r^2\}\right)^2.
\]
Note that Proposition 4.1 is stated for a slightly different notion of $r^\star$, which does not require central symmetry. However, by Lemma 8.2 in \cite{mourtada2025universal}, in the case of centrally symmetric convex bodies the result holds as claimed above.
Using the above chain of inequalities and \eqref{eq:rstarfixedpoint}, we get
$
\BoundOne(\sigma) \le \frac{140 (r^\star(\sigma))^2}{\sigma}. 
$ Finally, by the definition we have
\[
\frac{(r^\star(\sigma))^2}{2\sigma} \le w\big(T \cap r^\star(\sigma) B^d_2\big)- \frac{(r^\star(\sigma))^2}{2\sigma} \le \sup\limits_{r > 0}\left\{w\big(T \cap r B^d_2)- \frac{r^2}{2\sigma}\right\} = \BoundOne(\sigma). 
\]

To show the second claim we note that  $w\big(T \cap \ChatFixed(\sigma) B^d_2\big) - \frac{\ChatFixed(\sigma)^2}{2\sigma} \ge 0$, and hence by definition $r(\sigma) \le r^\star(2\sigma)$. Finally, we use again $\frac{(r^\star(2\sigma))^2}{4\sigma} \le \BoundOne(2\sigma)$ and $r(\sigma) \le \operatorname{rad}(T)$. The claim follows.
\end{proof}

\subsection{Projections, inradius and Stein's identity}

This subsection collects a few elementary estimates when $\sigma$
is small. Throughout, we have $g\sim\Normal{0}{I_d}$. For a set $T \subset \R^d$ we define the \emph{inradius} of $T$ by
\[
\inrad(T)=\sup\{r>0:\ rB_2^d\subseteq T\}.
\]
Above, we interpret $\sup \emptyset = 0$.

\begin{proposition}
\label{prop:balltype-small-sigma}
Let $T\subset\R^d$ be a centrally symmetric convex body.\footnote{By \emph{convex body}, we mean a compact, convex set with nonempty interior.} Then, the following hold.
\begin{enumerate}[label=(\roman*)]
\item For every $\sigma>0$,
\begin{equation}\label{eq:ball-upper-bounds}
\BoundOne(\sigma)\le \frac{\sigma d}{2},
\qquad
r^2(\sigma)\le 4\sigma^2 d,
\qquad
\E\|\Pi_T(\sigma g)\|^2\le \sigma^2 d.
\end{equation}
\item If $\sigma\le 2\inrad(T)/\sqrt d$, then
\begin{equation}\label{eq:ball-lower-bound-B1}
\BoundOne(\sigma)\ge \frac{\sigma d}{8}.
\end{equation}
\item If $\sigma\le \inrad(T)/(2\sqrt{d})$, then
\begin{equation}\label{eq:ball-lower-bound-proj}
\E\|\Pi_T(\sigma g)\|^2 \gtrsim \sigma^2d,
\qquad\text{and}\qquad
r^2(\sigma)\gtrsim \sigma^2d.
\end{equation}
\end{enumerate}
\end{proposition}

\begin{proof}
We use the standard bounds $\sqrt{d-\tfrac12}\le \E\|g\|\le \sqrt d$.
For \eqref{eq:ball-upper-bounds}, note that $w(T\cap rB_2^d)\le w(rB_2^d)=r\E\|g\|$, hence
\[
\BoundOne(\sigma)
=\max_{r>0}\Big\{w(T\cap rB_2^d)-\frac{r^2}{2\sigma}\Big\}
\le \max_{r>0}\Big\{r\E\|g\|-\frac{r^2}{2\sigma}\Big\}
=\frac{\sigma(\E\|g\|)^2}{2}
\le \frac{\sigma d}{2}.
\]
The bounds for $r^2(\sigma)$ follow by combining $r^2(\sigma)\le 4\sigma\,\BoundOne(2\sigma)$ from
\Cref{prop:equivalentfixedpoints} with $\BoundOne(2\sigma)\le (2\sigma)d/2$. Finally,
$\E\|\Pi_T(\sigma g)\|^2\le \E\|\sigma g\|^2=d\sigma^2$ follows from $\|\Pi_T(y)\|\le \|y\|$. 

\noindent For \eqref{eq:ball-lower-bound-B1}, set $r'=\frac{\sigma\sqrt d}{2}$. If $r'\le \inrad(T)$, then
$T\cap r'B_2^d=r'B_2^d$ and
\[
\BoundOne(\sigma)\ge r'\,w(B_2^d)-\frac{(r')^2}{2\sigma}
= r'\E\|g\|-\frac{(r')^2}{2\sigma}
\ge \frac{r'\sqrt d}{2}-\frac{(r')^2}{2\sigma}
=\frac{\sigma d}{8}.
\]
For \eqref{eq:ball-lower-bound-proj}, assume $\sigma\le \inrad(T)/(2\sqrt{d})$ and set
$A=\{\|\sigma g\|_2\le \inrad(T)\}$. By Markov's inequality,
\[
\Pr(A^c)=\Pr(\|\sigma g\|_2\ge \inrad(T))\le \frac{\E\|\sigma g\|_2^2}{\inrad(T)^2}
=\frac{d\sigma^2}{\inrad(T)^2}\le \frac14.
\]
Since $\inrad(T) B_2^d\subseteq T$ and $0\in T$, we have $\Pi_T(\sigma g)=\sigma g$ on $A$, hence
\[
\E\|\Pi_T(\sigma g)\|_2^2\ge \E\big[\|\sigma g\|_2^2\ind{A}\big]
\ge \big(\E[\|\sigma g\|_2\ind{A}]\big)^2
\]
by Jensen's inequality. Moreover, by the Cauchy-Schwarz inequality, it holds that
\[
\E[\|\sigma g\|_2\ind{A}]
\ge \E\|\sigma g\|_2-\E[\|\sigma g\|_2\ind{A^c}]
\ge \sigma\,\E\|g\|-\big(\E\|\sigma g\|_2^2\big)^{1/2}\Pr(A^c)^{1/2}
\ge \sigma\Big(\sqrt{d-\tfrac12}-\frac{\sqrt d}{2}\Big),
\]
which implies $\E[\|\sigma g\|_2\ind{A}] \gtrsim \sigma\sqrt d$ for all $d\ge 1$. Squaring yields
$\E\|\Pi_T(\sigma g)\|^2\gtrsim \sigma^2 d$. 

\noindent Finally, for the second inequality in~\eqref{eq:ball-lower-bound-proj}, note again that for any $r\le \inrad(T)$ we have $w(T\cap rB_2^d)=r\E\|g\|$.
Let $q(r) = r\,\E\|g\|-\frac{r^2}{2\sigma}$, which is maximized at
$\widetilde{r}=\sigma\E\|g\|$.
If $\widetilde{r}\le \inrad(T)$ (in particular, if $\sigma\le \inrad(T)/(2\sqrt d)$ since $\E\|g\|\le \sqrt d$), then
$q(\widetilde{r})=w(T\cap \widetilde{r}B_2^d)-\frac{(\widetilde{r})^2}{2\sigma}$.
Moreover, for all $r\ge 0$,
\[
w(T\cap rB_2^d)-\frac{r^2}{2\sigma}\le w(rB_2^d)-\frac{r^2}{2\sigma}=q(r)\le q(\widetilde{r}),
\]
so $r(\sigma)=\widetilde{r}$. Therefore,
$
r(\sigma)^2=\sigma^2(\E\|g\|)^2 \ge \sigma^2\Bigl(d-\tfrac12\Bigr)\ge \tfrac12 d\sigma^2.
$
The claim follows.
\end{proof}
 
It is useful to compare~\Cref{prop:balltype-small-sigma} to the intrinsic volume threshold in~\Cref{thm:indexstarbound}. 
It is easy to see that $\PeakIndex{\sigma}=d$ is characterized by the inequality 
$
\sigma<\sqrt{\tfrac{2}{\pi}}\,
\tfrac{\operatorname{Vol}_d(T)}{\operatorname{Sf}_{d-1}(\partial T)}.$
On the other hand, by isoperimetric considerations (e.g., \cite[Lemma~2.1]{giannopoulos2018inequalities}) it holds that 
\begin{equation}\label{eq:inrad-vol-sf}
\frac{\operatorname{Vol}_d(T)}{\operatorname{Sf}_{d-1}(\partial T)}\ge \frac{\inrad(T)}{d}.
\end{equation}
Thus, if $\sigma<\sqrt{2} \inrad(T)/(d \sqrt{\pi})$, then
$\PeakIndex{\sigma} = d$.
Consequently, the regime $\sigma \lesssim \inrad(T)/\sqrt d$ considered in~\Cref{prop:balltype-small-sigma}
is sufficient for the estimates in that proposition; it is generally weaker than (and does not imply) the condition
$\sigma<\sqrt{2}\,\inrad(T)/(d\sqrt{\pi})$ that guarantees $\PeakIndex{\sigma}=d$.

\begin{remark}[Stein's identity and divergence of $\Pi_T$]
\label{rem:stein-div}
Let $F \colon \R^d\to\R^d$ be an almost differentiable map in the sense of Stein~\cite{stein1981estimation}. Then, Stein's lemma, which is also commonly referred to as Gaussian integration by parts, implies
\[
\E \, \langle g, F(g)\rangle=\E\big[\operatorname{div}F(g)\big],
\]
whenever the expectations are finite. Applying this with $F(x)=\Pi_T(\sigma x)$---see below for a discussion regarding the differentiability---we obtain
\begin{equation}\label{eq:stein-div-form}
\E \, \langle \sigma g,\Pi_T(\sigma g)\rangle
=\sigma^2\E\big[\operatorname{div}\Pi_T(\sigma g)\big].
\end{equation}
Combining \eqref{eq:stein-div-form} with the projection identity
$\langle \Pi_T(y),\Pi_T(y)-y\rangle\le 0$ (valid for all $y$ when $0\in T$) yields
\begin{equation}
\E\|\Pi_T(\sigma g)\|^2
=\E \, \langle \Pi_T(\sigma g),\sigma g\rangle
+\E \, \langle \Pi_T(\sigma g),\Pi_T(\sigma g)-\sigma g\rangle 
\leq \sigma^2\E\big[\operatorname{div}\Pi_T(\sigma g)\big]. 
\label{eq:stein-upper-proj}
\end{equation}
Additionally, it holds that,
\begin{equation}\label{eq:sure-form}
\E\|\Pi_T(\sigma g)\|^2
= \E\|\sigma g-\Pi_T(\sigma g)\|^2 - d\sigma^2 + 2\sigma^2\,\E\big[\operatorname{div}\Pi_T(\sigma g)\big].
\end{equation}

The map $\Pi_T$ is $1$-Lipschitz, and hence almost everywhere differentiable. Denoting the Jacobian by $J\Pi_T(y)\in\R^{d\times d}$, it immediately follows that $\|J \Pi_T(y)\|_{\rm op} \leq 1$ at any point of differentiability $y \in \R^d$. 
Moreover, the matrix $J\Pi_T(y)$ is 
symmetric at all points of differentiability~\cite[Cor. 13.54]{RocWets98}, and by basic properties of 
closed convex sets, the projection $\Pi_T$ is a monotone operator~\cite[Cor. 12.20 and pp.~612]{RocWets98}, and thus 
$J\Pi_T$ is symmetric positive semidefinite.

From these observations, it follows that 
$0\le \operatorname{div}\Pi_T(y)=\operatorname{tr}(J\Pi_T(y))\le d$, which by \eqref{eq:stein-upper-proj} recovers the trivial bound $\E\|\Pi_T(\sigma g)\|^2\le d\sigma^2$.
More importantly, for some classes of bodies (e.g., polyhedra \cite{meyer2000degrees} or smooth bodies
\cite{kato2009degrees}) one can compute $\E[\operatorname{div}\Pi_T(\sigma g)]$ explicitly. Through
\eqref{eq:stein-upper-proj}, \eqref{eq:sure-form} and the fixed-point equivalences in
\Cref{prop:equivalentfixedpoints}, such formulas give geometric control of $\E\|\Pi_T(\sigma g)\|^2$ and
$r(\sigma)$, and therefore at least in some cases provide an alternative route to bounding localized widths that does not rely
on chaining arguments.
\end{remark}

\subsection{On the integrated terms and Chatterjee's analysis of least squares}
\label{sec:chaterjeecomparison}

In this section we discuss the connection between $r(\sigma)$ and the behavior of metric projections under Gaussian measure, in more detail. Specifically, given a closed, convex set $T \subset \R^d$ containing the origin, we study
$\E\|\Pi_T(\sigma g)\|^2$, where 
$g \sim \Normal{0}{I_d}$.
A result of Chatterjee \cite{Cha14} (see also \cite[Lemma 2.1]{prasadan2025some} for a more general statement we use here) implies that for an absolute constant $c > 0$, we have
\begin{equation}
\label{eq:projectionlowerbound}
r^2(\sigma) \le c\max\left\{\sigma^2, \E\|\Pi_T(\sigma g)\|^2\right\}, 
\quad \mbox{and} \quad 
\E\|\Pi_T(\sigma g)\|^2 \le c \, \max\{r^2(\sigma), \sigma^2\}.
\end{equation}
The following result shows that up to a small additive diameter term, which by \eqref{eq:vershyninrelation} never dominates the Gaussian width, both integral terms in \Cref{thm:decomp-of-gauss-width} are of the same order in the Gaussian case. Note that we always have a one-sided bound. That is, by \eqref{eq:upperboundonintegrals} for any $\sigma \ge 0$, we have 
\[
\int_{\sigma}^{\infty} \, 
\E \|\Pi_{T/\nu}(g)\|_2^2 \, \ud \nu \;
\le\;
\int_{\sigma}^{\infty}\frac{r(\nu)^2}{\nu^2}\, \ud\nu.
\]
\begin{proposition}
\label{pr:two-sidedgaussianwidth}
Let $T \subset \R^d$ be a closed, convex set, which contains the origin. For any $\sigma \ge 0$, it holds that
\[
\int_{\sigma}^{\infty}\frac{r(\nu)^2}{\nu^2}\, \ud\nu
\lesssim \int_{\sigma}^{\infty} \, 
\E \|\Pi_{T/\nu}(g)\|_2^2 \, \ud \nu + \rad(T).
\]
\end{proposition}

\begin{proof}
Since $0 \in T$, we have $\max\{r(\nu), \|\Pi_T(y)\|_2\}\le \rad(T)$, for all $y\in\R^d$.
The elementary inequality
$\min\{\max\{u,v\},w\}\le v+\min\{u,w\}$ for $u,v,w\ge 0$, and the bound~\eqref{eq:projectionlowerbound} yield
\begin{equation}\label{eq:int-B-by-A}
\int_{\sigma}^{\infty}\frac{r(\nu)^2}{\nu^2}\ud\nu 
\lesssim \int_{\sigma}^{\infty} \, 
\E \|\Pi_{T/\nu}(g)\|_2^2 \, \ud \nu +
\int_{\sigma}^{\infty}\min\Big\{1,\frac{\rad(T)^2}{\nu^2}\Big\}\ud\nu 
\lesssim 
\int_{\sigma}^{\infty} \, 
\E \|\Pi_{T/\nu}(g)\|_2^2 \, \ud \nu +
\rad(T),
\end{equation}
as required.
\end{proof}

\section{Wills functional and the distribution of intrinsic volumes}
\label{sec:wills}
One of the remarkable properties of the first term in both decompositions of \Cref{thm:decomp-of-gauss-width} is that it admits a bound closely connected to the Wills functional from convex geometry \cite{wills1973gitterpunktanzahl}. In our setting, this bound is never loose: it provably yields an estimate of the Gaussian width within a universal multiplicative factor of the true Gaussian width (see \Cref{cor:gausstowills} below). At the same time, passing to the Wills functional unlocks geometric tools that are not directly available for bounding $w(T)$, which we believe highlights one of the main strengths of the decomposition in \Cref{thm:decomp-of-gauss-width} and in \Cref{thm:decomp-via-projections}. For a detailed exposition and a series of results on the Wills functional, we refer to the recent work by Mourtada \cite{mourtada2025universal}.

As before, given a nonempty convex set $T$ in $\R^d$ let $\{V_j(T)\}_{j = 0}^d$ denote the sequence of its intrinsic volumes.
The Wills functional \cite{wills1973gitterpunktanzahl, hadwiger1975will} is the sum of intrinsic volumes 
\[
W(T) = \sum_{j = 0}^dV_j(T).
\]
The connection between the first terms of decompositions in \Cref{thm:decomp-of-gauss-width} and the Wills functional arises from inequality~\eqref{eq:boundontauone}.
Applying Jensen's inequality and Vitale's representation of the Wills functional~\cite{vitale1996wills}, it holds that
\begin{multline}
w\left(T \cap \ChatFixed(\sigma) B_2^d\right)-\frac{\ChatFixed(\sigma)^2}{2\sigma}
\le \sigma\,\E\left[\sup_{x \in T/\sigma}\left\{\langle x,g\rangle-\frac{\|x\|^2}{2}\right\}\right] 
\\\le \sigma \log \E\left[\sup_{x \in T/\sigma}\exp\left(\langle x,g\rangle-\frac{\|x\|^2}{2}\right)\right] 
= \sigma \log W\left(\frac{T}{\sigma\sqrt{2\pi}}\right).
\label{eq:willsbound}
\end{multline}
We emphasize that, for the purposes of controlling the Gaussian width---which is the context of \Cref{thm:decomp-of-gauss-width}---we have the following complementary result of McMullen \cite{mcmullen1991inequalities}:
\begin{equation}
\label{eq:mcmullen}
\sigma\log\left(W\left(\frac{T}{\sigma\sqrt{2\pi}}\right)\right) \le \sigma V_1\left(\frac{T}{\sigma\sqrt{2\pi}}\right) = w(T),
\end{equation}
where the last relation is due to \eqref{eq:intrvolumes}. Hence, passing through $\sigma\log\left(W\left(\frac{T}{\sigma\sqrt{2\pi}}\right)\right)$ does not lead to any loss, for the purposes of controlling $w(T)$.
\subsection{Bounding Gaussian width using Wills functional}
As a precise statement we have the following version of \Cref{thm:decomp-of-gauss-width}.
\begin{corollary}\label{cor:gausstowills} For a compact convex set $T \subset \R^\dimension$ containing the origin and any $\sigma > 0$ it holds that 
\[
 w(T) \asymp \sigma\log\left(W\left(\frac{T}{\sigma\sqrt{2\pi}}\right)\right) + \int_{\sigma}^{\infty}\E \|\Pi_{T/\nu}(g)\|_2^2\ud\nu,
\]
and
\[
w(T) \asymp \sigma\log\left(W\left(\frac{T}{\sigma\sqrt{2\pi}}\right)\right) + \int_{\sigma}^{\infty}\frac{r(\nu)^2}{\nu^2}\ud\nu.
\]
\end{corollary}
\begin{proof}
Both upper bounds follow from \Cref{thm:decomp-of-gauss-width} and \eqref{eq:willsbound}. The lower bound is a combination of \Cref{thm:decomp-of-gauss-width} and McMullen's inequality \eqref{eq:mcmullen}.
\end{proof}

The results of \Cref{cor:gausstowills} provide a natural interplay between two limiting regimes. For example, by \cite[Proposition 3.3]{mourtada2025universal} we have
\[
\lim_{\sigma \to \infty}\sigma\log\left(W\left(\frac{T}{\sigma\sqrt{2\pi}}\right)\right) = w(T), \quad \text{while} \quad \lim_{\sigma \to \infty}\int_{\sigma}^{\infty}\frac{r(\nu)^2}{\nu^2}\ud\nu = \lim_{\sigma \to \infty}\int_{\sigma}^{\infty}\E\|\Pi_{T/\nu}(g)\|^2\ud\nu = 0.
\]
When $\sigma \to 0$, the Gaussian width is fully governed by decompositions of \Cref{thm:decomp-of-gauss-width} and \Cref{thm:decomp-via-projections}. 

Another corollary of \eqref{eq:willsbound} and \Cref{thm:decomp-of-gauss-width} is the improvement on the recent lower bound in \cite[Theorem 1.2]{AHY21}, which states that for any convex body $T$, such that $T \subset rB_{2}^d$ it holds that
\begin{equation}
\label{eq:ahy21}
\log W\left(\frac{T}{\sqrt{2\pi}}\right) \ge w(T) - \frac{r^2}{2}.
\end{equation}
An immediate consequence of \Cref{thm:decomp-via-projections} and~\eqref{eq:willsbound} is the following lower bound.
\begin{corollary}
\label{cor:lowerwills}
For a compact convex set $T \subset \R^\dimension$ it holds that 
\[
\log W\left(\frac{T}{\sqrt{2\pi}}\right) 
\geq w(T) - \frac{1}{2}\int_{1}^{\infty}\E\|\Pi_{T/\nu}( g)\|^2\ud\nu.
\]
\end{corollary} 
\medskip
Note that when $T \subseteq rB_{2}^d$, we have \[
\frac{1}{2}\int_{1}^{\infty}\E\|\Pi_{T/\nu}( g)\|^2\ud\nu = \frac{1}{2}\int_{1}^{\infty}\frac{\E\|\Pi_{T}(\nu g)\|^2}{\nu^2}\ud\nu \le \frac{r^2}{2}\int_{1}^{\infty}\frac{1}{\nu^2}\ud\nu = \frac{r^2}{2},\]
so the lower bound of \Cref{cor:lowerwills} is never worse than the bound \eqref{eq:ahy21}. 

\begin{remark}
A further improvement of \Cref{cor:lowerwills} is possible via a combination with the recent result in \cite{fernandez2025convex}, which is itself an improvement of \cite[Theorem 1.2]{AHY21}. Namely, applying \cite[Theorem 7.1]{fernandez2025convex} with \(K = T/\sqrt{2\pi}\), and combining it with the proof of \Cref{cor:lowerwills}, we immediately obtain
\[
\log W\left(\frac{T}{\sqrt{2\pi}}\right) 
\geq w(T) + \frac{1}{2}\Var\left(\sup_{x \in T}\left\{\langle x,g\rangle-\frac{\|x\|^2}{2}\right\}\right)- \frac{1}{2}\int_{1}^{\infty}\E\|\Pi_{T/\nu}( g)\|^2\,\ud\nu.
\]
We note, however, that the additional variance term cannot contribute more than the integrated projection term, since otherwise this would contradict McMullen's inequality~\eqref{eq:mcmullen}. Alternatively, by the Gaussian Poincar\'e inequality and an elementary monotonicity argument,
\[
\Var\left(\sup_{x \in T}\left\{\langle x,g\rangle-\frac{\|x\|^2}{2}\right\}\right)
\le \E\|\Pi_T(g)\|^2
= \int_{1}^{\infty}\frac{\E\|\Pi_T(g)\|^2}{\nu^2}\,\ud\nu \le \int_{1}^{\infty}\E\|\Pi_{T/\nu}(g)\|^2\,\ud\nu.
\]
\end{remark}

\subsection{Proof of \Cref{thm:indexstarbound}}
Before proceeding with the details of the proof we need several facts. The first is that the Poisson log-concavity of the intrinsic volume sequence implies
\begin{equation}
\label{eq:viupperbound}
    V_{i}(T) \le \frac{(V_{1}(T))^i}{i!}.
\end{equation}
The key technical result we need in our proof is an inequality due to Mourtada~\cite{mourtada2025universal}. Namely, by \cite[Proposition~2.1]{mourtada2025universal}, we have
\begin{equation}
\label{eq:upperandlowerboundwills}
\sigma\log \left(\max_{1\le i\le d} V_i\left(\frac{T}{\sigma\sqrt{2\pi}}\right)\right)\le
\sigma\log\left(W\left(\frac{T}{\sigma\sqrt{2\pi}}\right)\right)
\ \le\ 8\sigma\log \left(\max_{1\le i\le d} V_i\left(\frac{T}{\sigma\sqrt{2\pi}}\right)\right),
\end{equation}
whenever $w(T) \ge 2\sigma$. Let us briefly comment on this inequality. The fact that the intrinsic volume sequence is unimodal is well-known. What \eqref{eq:upperandlowerboundwills} shows is that, up to universal multiplicative constants, $\log W$ is controlled by the logarithm of a single intrinsic volume. This result follows directly from Poisson-log-concavity of the sequence of intrinsic volumes.

Now, the proof proceeds in exactly the same way for both integral terms, so we choose the decomposition of \Cref{thm:decomp-of-gauss-width} without loss of generality. Assume first that $\PeakIndex{\sigma} \neq 0$. If $w(T)<2\sigma$, then the desired bound is immediate since $w(T)<2\sigma\le 45\sigma\max\{\PeakIndex{\sigma},1\}$ (and the integral term is nonnegative). Hence we may assume $w(T)\ge 2\sigma$, so that \eqref{eq:upperandlowerboundwills} applies. By \Cref{thm:decomp-of-gauss-width}, inequalities \eqref{eq:willsbound} and \eqref{eq:upperandlowerboundwills}, and since $\sigma \le w(T)$, we have
\begin{align}
w(T)
&\le 8\sigma\log \left(\max_{1\le i\le d} V_i\left(\frac{T}{\sigma\sqrt{2\pi}}\right)\right) +\frac{1}{2}\int_{\sigma}^{\infty}\frac{r(\nu)^2}{\nu^2}\ud\nu
\\
&= 8\sigma\log \left( V_{i^{\star}_\sigma}\left(\frac{T}{\sigma\sqrt{2\pi}}\right)\right) +\frac{1}{2}\int_{\sigma}^{\infty}\frac{r(\nu)^2}{\nu^2}\ud\nu
\\
&\le 8\sigma \log \left(\frac{1}{\sigma^{i^{\star}_\sigma}}\frac{(w(T))^{i^{\star}_\sigma}}{{i^{\star}_\sigma}!}\right) + \frac{1}{2}\int_{\sigma}^{\infty}\frac{r(\nu)^2}{\nu^2}\ud\nu,
\end{align}
where in the last line we used \eqref{eq:viupperbound} for $K=\frac{T}{\sigma\sqrt{2\pi}}$ and $V_1(K)=w(T)/\sigma$. Denote $x = w(T)/\sigma$ and $y = \frac{1}{2\sigma}\int_{\sigma}^{\infty}\frac{r(\nu)^2}{\nu^2}\ud\nu$. Then, we have
\[
x \le 8\big(i^{\star}_\sigma\log x - \log(i^{\star}_\sigma!)\big) + y.
\]
Using $\log s \le s-1$ with $s=\frac{x}{16i^{\star}_\sigma}$ and $\log(i^{\star}_\sigma!) \ge i^{\star}_\sigma\log i^{\star}_\sigma - i^{\star}_\sigma$, we get
\[
x \le 8\Big(\frac{x}{16} + i^{\star}_\sigma\log 16\Big) + y,
\]
hence
\[
x \le 16i^{\star}_\sigma\log(16) + 2y \le 45i^{\star}_\sigma + 2y.
\]
Thus, we have
\[
w(T) \le 45\sigma i^{\star}_\sigma + \int_{\sigma}^{\infty}\frac{r(\nu)^2}{\nu^2}\ud\nu.
\]
Finally, the case where $i^{\star}_\sigma = 0$ is impossible since in this case it holds that $V_1\big(\frac{T}{\sigma\sqrt{2\pi}}\big)\le V_0=1$, which implies $\sigma \ge w(T)$ contradicting the assumption $\sigma < w(T)$.

We now turn to the lower bound. By the above argument, we have $i^{\star}_\sigma \ge 1$. Recall that $K = \frac{T}{\sigma\sqrt{2\pi}}$. For intrinsic volumes the Poisson-log-concavity gives, for every $i\ge 1$,
\begin{equation}\label{eq:ratio-V1}
\frac{V_i(K)}{V_{i-1}(K)} \le \frac{V_1(K)}{i}.
\end{equation}
From \eqref{eq:ratio-V1} we get $V_i(K)/V_{i-1}(K)<1$ for all $i>V_1(K)$, hence the peak intrinsic index $i^{\star}_\sigma$ satisfies
\begin{equation}
\label{eq:upperboundonistar}
i^{\star}_\sigma \le \lfloor V_1(K)\rfloor, \quad \textrm{which implies} \quad \sigma i^{\star}_\sigma \le w(T),
\end{equation}
where we used $V_1(K)=\sqrt{2\pi}w(K)$ and $w(K)=w(T)/(\sigma\sqrt{2\pi})$.
Separately, \Cref{thm:decomp-of-gauss-width} yields $w(T)\ge \frac{1}{2}\int_{\sigma}^{\infty}\frac{r(\nu)^2}{\nu^2}\ud\nu$. Taking the halved sum of these two lower bounds proves the claim.

Finally, we justify the diameter specialization stated in \Cref{thm:indexstarbound}. Take $\sigma=\diam(T)/\sqrt{2\pi}$ and set $K=\frac{T}{\sigma\sqrt{2\pi}}=\frac{T}{\diam(T)}$. If $w(T)=\sigma$, then $V_1(K)=w(T)/\sigma=1=V_0(K)$, and \eqref{eq:viupperbound} gives
$V_i(K)\le 1/i!$ for all $i\ge 1$, so we have $i^\star=1$.
Hence, in what follows we may assume $w(T)>\sigma$ (so $\sigma\in(0,w(T))$).
Since $0\in T$, we have $\|t\|\le \diam(T)$ for all $t\in T$. In particular, this implies $r(\nu)\le \diam(T)$ for all $\nu\ge \sigma$, and thus
\[
\frac{1}{2}\int_{\sigma}^{\infty}\frac{r(\nu)^2}{\nu^2}\ud\nu
\le \frac{\diam^2(T)}{2}\int_{\sigma}^{\infty}\frac{1}{\nu^2}\ud\nu
= \frac{\diam^2(T)}{2\sigma}
= \frac{\sqrt{2\pi}}{2}\diam(T).
\]
Let $i^\star$ be as in \eqref{eq:istar}. Since $V_1(K)=w(T)/\sigma\ge 1$ by \eqref{eq:vershyninrelation}, such a maximizer exists with $i^\star\ge 1$. Moreover, the same argument leading to \eqref{eq:upperboundonistar} yields
\[
i^\star \le \lfloor V_1(K)\rfloor
= \left\lfloor \sqrt{2\pi}\,w(K)\right\rfloor
\le \left\lfloor \sqrt{2\pi d}\right\rfloor,
\]
where in the last inequality we used \eqref{eq:vershyninrelation} for $K$ (note that $\diam(K)=1$). The lower bound in \Cref{thm:indexstarbound} gives
\[
w(T)\ge \sigma i^\star = \frac{1}{\sqrt{2\pi}}\,i^\star\diam(T).
\]
For the upper bound, we use the estimate already proved above and the bound on the integral term to get
\[
w(T) \le 45\sigma i^\star + \int_{\sigma}^{\infty}\frac{r(\nu)^2}{\nu^2} \, \ud\nu
\le \frac{45}{\sqrt{2\pi}}\,i^\star\diam(T) + \sqrt{2\pi}\diam(T)
\le 21\,i^\star\diam(T),
\]
where in the last step we used $i^\star\ge 1$. This proves $w(T)\asymp i^\star\diam(T)$, completing the proof. \qed

\subsubsection{Proof of~\Cref{cor:T-lowner}}
\label{sec:T-lowner}

The statement that $T$ lies in Löwner position is equivalent to the polar body $T^\circ$ lying in John's position: the maximal volume ellipsoid included within $T^\circ$ is the ball $\inrad(T^\circ) B^d_2$. 
Hence, it holds that 
\begin{equation}
\label{ineq:sandwich-under-Lowner}
\rad(T) \, w(B^d_2) \stackrel{{\rm(i)}}{\geq} w(T) 
= \E \|g\|_{T^\circ} 
\stackrel{{\rm(ii)}}{\geq}\frac{1}{\inrad(T^\circ)} \E \|g\|_\infty = \rad(T) \, w(B^d_1).
\end{equation}
Above, inequality (i) follows from the trivial inclusion $T \subset \rad(T)\,B^d_2$, while inequality (ii) follows from the result of Schechtman-Schmuckenschläger~\cite[Proposition 4.11]{SchSch95}; alternatively, one can obtain the inequality with a worse constant from the contact points of the John ellipsoid.
By~\Cref{thm:indexstarbound}, we have $w(T) \asymp i^\star \rad(T)$, from which the claim follows from the inequalities~\eqref{ineq:sandwich-under-Lowner}: simply observe $w(B^d_1) \asymp \sqrt{\log(d)}$ and $w(B^d_2) \asymp \sqrt{d}$. \qed

\section{Statistical rates and information-theoretic tools}
\label{sec:info-and-stat}
In this section, we discuss how different information-theoretic tools can be used to bound the terms in our decomposition theorems. 

\subsection{Upper bounds involving covering numbers and statistical rates}

In this section, we develop bounds on the radii $\ChatFixed(\sigma)$ that involve statistical rates of estimation and the metric entropy of the underlying constraint set. 
We make use of the statistical rate for estimating $\thetastar \in T$ when observed through Gaussian noise. This is the minimax mean squared error (MSE), and is given by
\begin{equation}\label{defn:minimax-rate}
\big(\eps_\star(\sigma)\big)^2 
\defn 
\inf_{\hat \theta} \sup_{\theta \in T}
\E_{Y \sim \Normal{\theta}{\sigma^2 I_\dimension}}
\Big[\|\hat \theta(Y) - \theta\|_2^2\Big]. 
\end{equation}
Above, the infimum is taken over estimators (\ie measurable functions) $\hat \theta \colon \R^\dimension \to \R^\dimension$. Our main result shows that the radii and statistical rate are related via $r(\sigma) \lesssim \eps_\star(\sigma)$.

\begin{theorem} \label{thm:upper-bound-minimax}
There is a universal constant $C > 0$ such that for a compact, centrally symmetric, convex set $T \subset \R^\dimension$ it holds that 
\begin{equation}
\label{eqn:results-via-minimax-rates}
\max\left\{\ChatFixed(\sigma)^2, \E \|\Pi_{T}(\sigma g)\|_2^2\right\}  \lesssim  \big(\eps_\star(\sigma)\big)^2,
\end{equation}
for any $\sigma > 0$.
\end{theorem} 

\noindent The proof is presented below in \Cref{sec:proof-of-prop-upper-minimax}. Although from~\Cref{thm:upper-bound-minimax}, it appears as though the bound on the radius $r(\sigma)$ could be difficult to compute, it turns out that the variational problem underlying $\eps_\star(\sigma)$ can be simplified when  $T \subset \R^\dimension$ is convex, using metric properties of $T$.  Recall that the packing entropy is given by
\[
M(A, B) \defn 
\sup_{\text{finite}~A_0 \subset A} \Big\{\, 
|A_0| : x - y \not \in B \quad \mbox{for all distinct}~x, y \in A_0 \,\Big\}.
\]
We define the \emph{local} packing entropy of a convex set $T$ at scale $\eps > 0$ by 
\begin{equation}\label{def:local-packing-entropy-convex}
M^{\rm loc}_T(\eps) 
\defn 
\sup_{\theta \in T} M\big(T \cap (\theta + 2\eps B^d_2), \eps B^d_2\big).
\end{equation}
Then, the next result shows that $\eps_\star(\sigma)$ is determined by a suitable fixed point of the local entropy. 

\begin{proposition} [Metric entropy characterization of statistical rate]
\label{pr:metric-characterization-of-stat-rate}
For any bounded convex set $T \subset \R^\dimension$, and any $\sigma > 0$, the statistical rate satisfies 
\begin{equation}
\label{eqn:metric-characterization}
\eps_\star(\sigma) \asymp
\sup\bigg\{\, \eps > 0 : \log M^{\rm loc}_T(\eps)
\geq \frac{\eps^2}{\sigma^2}\,\bigg\}.
\end{equation}
\end{proposition} 
This result was first claimed, to the best of our knowledge, in~\cite[Theorem~2.11]{Ney23}, building on~\cite{birge1993rates,yang1999information}.\footnote{As later announced in~\cite{yi2025nonparametric}, the original proof in~\cite{Ney23} contained a mistake. The author announced a correction in a later paper~\cite[Theorem 3.10]{prasadan2025informationtheoreticlimitsrobust}, which yields a result for the Gaussian sequence model by taking $\eps = 0, N = 1$ in the notation of that work.} 
We provide a simplified alternative argument in~\Cref{app:GSM}: we show that the upper bound in~\cref{eqn:metric-characterization} is achieved by the LSE over an appropriately constructed global $\varepsilon$-net of $T$, in contrast to the more complicated iterative procedure in~\cite{Ney23}. 

As a first immediate corollary of \Cref{thm:upper-bound-minimax} and \Cref{pr:metric-characterization-of-stat-rate}, we have the following improvement of \Cref{prop:equivalentfixedpoints}.
\begin{corollary}
\label{cor:improvedfixedpointformula}
For a compact, centrally symmetric, convex set $T$ in $\R^d$, it holds that
\[
r(\sigma) \lesssim \eps_\star(\sigma) \lesssim \min\{r^\star(2\sigma), \operatorname{rad}(T)\}.
\]
\end{corollary}

\begin{proof}
The lower bound $r(\sigma)\lesssim \varepsilon_\star(\sigma)$ follows from \Cref{thm:upper-bound-minimax}.
Next, $\varepsilon_\star(\sigma)\le \rad(T)$ since the constant estimator $\hat\theta = 0$ has
$\sup_{\theta\in T}\E\|\hat\theta-\theta\|_2^2 = \sup_{\theta\in T}\|\theta\|_2^2=\rad(T)^2$. For the remaining bound, consider the projection $\hat\theta=\Pi_T(\theta+\sigma g)$. By definition of the minimax risk,
\[
\varepsilon_\star(\sigma)^2
\le \sup_{\theta\in T}\E\|\Pi_T(\theta+\sigma g)-\theta\|_2^2.
\]
The standard localized Gaussian-complexity analysis of least squares over star-shaped classes
(see, e.g., \cite[Ch.~13]{Wai19}, in particular the proof based on the critical radius)
yields
\[
\sup_{\theta\in T}\E\|\Pi_T(\theta+\sigma g)-\theta\|_2^2 \lesssim r^\star(2\sigma)^2,
\]
up to universal constants. 
Combining the displays proves $\varepsilon_\star(\sigma)\lesssim r^\star(2\sigma)$.
\end{proof}

We note a curious aspect of the gap between $r(\sigma)$ and $r^\star(2\sigma)$ highlighted by \Cref{cor:improvedfixedpointformula}.
Both $r(\sigma)$ and $r^\star(2\sigma)$ depend explicitly on the Gaussian width of $T$ at different scales and, in general, are not expected to be determined by metric entropy at a single scale.
In contrast, the intermediate quantity $\eps_\star(\sigma)$ that controls their discrepancy is purely metric: it admits a characterization in terms of local packing numbers of $T$ and does not depend on any additional structure.

A second immediate consequence of~\Cref{thm:upper-bound-minimax}~and~\Cref{pr:metric-characterization-of-stat-rate} is the following bound on the integrated terms which appeared in~\Cref{thm:decomp-of-gauss-width,thm:decomp-via-projections}.
\begin{corollary} \label{cor:bound-on-term-two}
For any compact, centrally symmetric convex set $T \subset \R^\dimension$ and any $\sigma > 0$, it holds that
\[
\max\left\{\int_\sigma^\infty \frac{r^2(\nu)}{\nu^2} \, \ud \nu, \int_\sigma^\infty \E \|\Pi_{T/\nu}(g)\|_2^2 \, \ud \nu\right\} \lesssim
\int_\sigma^\infty \frac{\eps^2_\star(\nu)}{\nu^2} \, \ud \nu. 
\]
\end{corollary} 

Combining~\Cref{cor:bound-on-term-two} with the metric entropy characterization in \Cref{pr:metric-characterization-of-stat-rate} yields inequalities solely in terms of local entropy of the underlying set $T$. To state the result, we recall the (truncated) Dudley entropy integrals, 
\[
\cJ^{\rm loc}_\delta(T) \defn \int_\delta^{\infty} \sqrt{\log M^{\rm loc}_T(\eps)} \,\ud \eps.
\]
In the case that $\delta = 0$, we simply write $\cJ^{\rm loc}(T)$. Our next result relates the right-hand side quantity in \Cref{cor:bound-on-term-two} with the Dudley integral.

\begin{proposition}
\label{prop:entropy-integral-comparison}
There exists a constant $c > 1$ such that the following inequality holds for every $\sigma > 0$ and any bounded convex set $T \subset \R^d$:
\begin{equation}
\label{eqn:relations-for-series}
\TruncIntegral^{\rm loc}_{c\eps_\star(\sigma)}(T) + 
\frac{\eps^2_\star(\sigma)}{\sigma}
\lesssim 
\int_\sigma^\infty \frac{\eps_\star(\nu)^2}{\nu^2} \, \ud \nu
\lesssim 
\TruncIntegral^{\rm loc}_{ \eps_\star(\sigma)/c}(T) + 
\frac{\eps^2_\star(\sigma)}{\sigma}
\end{equation}
\end{proposition} 

\noindent The proof is presented in \Cref{sec:proof-of-entropy-characterization}. Combining \Cref{prop:entropy-integral-comparison} with~\Cref{thm:decomp-of-gauss-width},~\Cref{prop:balltype-small-sigma}(i), and~\Cref{cor:bound-on-term-two}, and the fact that $\eps_\star(\sigma) \lesssim \sigma \sqrt{d}$ (see~\Cref{lem:equivalent-versions-of-the-minimax-rate}), we recover the classical Dudley entropy integral bound,
\[
w(T) \lesssim \limsup_{\sigma \to 0^+} 
\Big\{\sigma d + \int_\sigma^\infty \frac{\eps_\star(\nu)^2}{\nu^2}\ud \nu \Big\} 
\asymp 
\TruncIntegral^{\rm loc}(T) \leq 
\int_0^\infty \sqrt{\log M(T, \eps B^d_2)} \, \ud \eps.
\]

\begin{remark}[Sharpness of Dudley's entropy integral]
Combining \Cref{cor:improvedfixedpointformula} with \Cref{prop:entropy-integral-comparison}, and leveraging the equivalence between the local and global Dudley integrals (see \Cref{prop:dudleygloballocal}), it can be seen that
\[
w(T) \ll \int_0^\infty \sqrt{\log M(T, \eps B^d_2)} \, \ud \eps, \quad \mbox{if and only if}
\quad 
\int_0^\infty \frac{\E \|\Pi_{T}(\sigma g)\|_2^2}{\sigma^2} \, \ud \sigma 
\ll 
\int_0^\infty \frac{\eps_\star(\sigma)^2}{\sigma^2} \, \ud \sigma. 
\]
In other words, the Dudley entropy integral bound is loose whenever the (integrated) statistical minimax rates are much larger than the (integrated) performance of the LSE at the origin, or equivalently, the (integrated) expected norm squared of the metric projection.
\end{remark}

\subsection{Sudakov minoration via information-theoretic tools}
\label{sec:sudakov}

In this section, we revisit the classical Sudakov minoration \cite{Sudakov1976} in the context of our results and the Gaussian sequence model viewpoint. 
Our first step is a localized inequality that is, on its face, weaker than the standard global Sudakov bound. 
Nevertheless, our key observation is that when combined with the comparison results developed below (\Cref{prop:sudakovlocalglobal}), it is already sufficient to recover the usual Sudakov lower bound up to constants.

\begin{lemma}
\label{prop:minoration-local}
For any closed convex bounded set $T \subset \R^d$, the Gaussian width satisfies
\[
w(T) \gtrsim \sup_{\eps > 0} \eps \sqrt{\log M^{\rm loc}_T(\eps)}.
\]
\end{lemma}
\begin{proof}
Let $\Pi_T(Y) \defn \argmin_{\vartheta \in T}\|Y-\vartheta\|_2^2$ denote the Euclidean projection onto $T$. 
By the projection inequality, for any $\theta^\star\in T$ and any $Y\in\R^d$,
\[
\|\Pi_T(Y)-\theta^\star\|_2^2 \le \langle Y-\theta^\star,\Pi_T(Y)-\theta^\star\rangle.
\]
Taking $Y=\theta^\star+\sigma g$ with $g\sim\Normal{0}{I_d}$, taking expectations, and then taking $\sup_{\theta^\star\in T}$ yields
\[
\eps_\star(\sigma)^2 \le \sup_{\theta^\star\in T}\E\|\Pi_T(\theta^\star+\sigma g)-\theta^\star\|_2^2 \le \sigma w(T).
\]
Combining this with the metric characterization of the statistical rate in \Cref{pr:metric-characterization-of-stat-rate}, we obtain for a universal constant $c>0$,
\[
w(T) \ge c \sup_{\eps>0}\sup_{\substack{\sigma>0:\\ \eps^2\le \sigma^2\log M_T^{\rm loc}(\eps)}} \frac{\eps^2}{\sigma}
\ge c \sup_{\eps>0}\eps\sqrt{\log M_T^{\rm loc}(\eps)},
\]
where in the last step we take $\sigma^2=\eps^2/\log M_T^{\rm loc}(\eps)$ (when $\log M_T^{\rm loc}(\eps)>0$).
\end{proof}

\begin{remark}[Proof method for \Cref{prop:minoration-local}]
\label{rem:KL-properties}
We emphasize that our argument is information-theoretic. The only Gaussian-specific input is the explicit form of the Kullback-Leibler divergence in the Gaussian shift model, as it enters Fano's lemma, and the proof would go through under any comparable control of the divergence at the relevant scales. This contrasts with standard proofs of Sudakov minoration, which typically proceed either via Gaussian comparison inequalities such as Sudakov-Fernique, or via Gaussian shift arguments combined with the dual Sudakov inequality and entropy duality; see \cite{vershynin2018high,ledoux2013probability,vershynin_gfa_notes,Talagrand2021}.
\end{remark}

\begin{remark}
\label{rem:generalized-mmt}
In fact, following~\Cref{rem:KL-properties}, the the quadratic behavior of the KL, combined with~\Cref{thm:decomp-via-projections} and the rate-distortion integral of J.\ Liu~\cite{liu2025simple} leads to a simplified proof and generalized statement of the lower bound in the majorizing measures theorem. 
This is further developed in a short note that the current authors recently announced~\cite{PatZhi26}.
The paper of I.\ Zadik~\cite{Zad26} also uses a form of~\Cref{thm:decomp-via-projections} (for finite $T$) along with~\cite{liu2025simple}
to obtain the lower bound for the majorizing measures theorem, in the Gaussian setting.
\end{remark}

\begin{remark}[On the convexity assumption in  \Cref{prop:minoration-local}]
\label{rem:sudakov}
The same approach yields the lower bound
\[
\eps_\star(\sigma)\gtrsim \sup\Big\{\eps>0:\ \eps \le \sigma\sqrt{\log M_T^{\rm loc}(\eps)}\Big\}
\]
for any $T\subset\R^\dimension$. Convexity is only used to obtain the complementary upper bound $\eps_\star(\sigma)^2 \le \sigma w(T)$.
For a general closed set $T$, the projection estimator still satisfies
$
\|\Pi_T(Y)-\theta\|_2^2 \le 2\langle Y-\theta,\Pi_T(Y)-\theta\rangle,
$
and therefore $\eps_\star(\sigma)^2 \le 2\sigma w(T)$, leading to an analog of \Cref{prop:minoration-local} with slightly worse constants.
\end{remark}

Now, using information-theoretic arguments based on Anderson's lemma, we obtain a dual form of Sudakov minoration.
Here we show that the same comparison principle, combined with a Fano-type argument, yields the following dual Sudakov bound.
We note that the dual Sudakov inequality is classical and is often proved via Gaussian shift and volumetric arguments
(see \cite{ledoux2013probability, vershynin_gfa_notes, Talagrand2021}), whereas in our approach the Gaussian shift is combined with the information-theoretic method.

\begin{proposition}
\label{prop:global-minoration}
Let $T \subset \R^\dimension$ be a centrally symmetric convex body.
Then the Gaussian width satisfies
\[
w(T) \gtrsim \sup_{\eps > 0} \eps \sqrt{\log M(B_2^\dimension, \eps T^\circ)}.
\]
\end{proposition}
\noindent Together with the duality of entropy \cite{artstein2004duality}, \Cref{prop:global-minoration} implies the usual Sudakov minoration
\[
w(T) \gtrsim \sup_{\eps > 0} \eps \sqrt{\log M(T, \eps B_2^\dimension)}.
\]
For an elementary argument relating the primal and dual forms of Sudakov minoration, see \cite[Lemma~2.5]{vershynin_gfa_notes}.

\begin{proof}
Fix $\sigma>0$ and consider the Gaussian location model $Y\sim\Normal{\theta}{\sigma^2 I_\dimension}$.
Since $T$ is centrally symmetric, $h_T=\|\cdot\|_{T^\circ}$ is a norm.
By the minimax identity for Gaussian shifts based on Anderson's lemma
(see \cite[Theorem~28.7]{polyanskiy2025information} and \cite[Lemmas~28.9--28.10]{polyanskiy2025information}),
\begin{equation}\label{eq:anderson-minimax-width}
\inf_{\hat\theta}\sup_{\theta\in\R^\dimension}\E_\theta\|\hat\theta(Y)-\theta\|_{T^\circ}
=\E\|\sigma g\|_{T^\circ}
=\sigma\,\E\|g\|_{T^\circ}
=\sigma\,w(T),
\end{equation}
where $g\sim\Normal{0}{I_\dimension}$. Fix $\eps>0$ and set $N\defn M(B_2^\dimension,\eps T^\circ)$.
Choose corresponding $\{\theta_1,\dots,\theta_N\}\subset B_2^\dimension$ such that
$\|\theta_i-\theta_j\|_{T^\circ}>\eps$ for all $i\neq j$.
A standard application of Fano's inequality (see e.g., \cite[Proposition 15.12]{Wai19}) yields
\[
\inf_{\hat\theta}\max_{1\le i\le N}\E_{\theta_i}\|\hat\theta(Y)-\theta_i\|_{T^\circ}
\ge \frac{\eps}{2}\Big(1-\frac{1}{2\sigma^2\log N}-\frac{\log 2}{\log N}\Big).
\] If $N\ge 6$, take $\sigma^2=(\log N)^{-1}$ in the previous display and combine with
\eqref{eq:anderson-minimax-width} to obtain
\[
w(T)\ge \frac{\eps}{2}\sqrt{\log N}\Big(\frac12-\frac{\log 2}{\log N}\Big)
\ge \frac{1}{20}\,\eps\sqrt{\log N}.
\]
If $2\le N\le 5$, then $\eps\le \|\theta_1-\theta_2\|_{T^\circ}\le
2\rad(T)$, so $\rad(T)\ge \eps/2$.
Let $t_0\in T$ satisfy $\|t_0\|_2=\rad(T)$. Then
\[
w(T)=\E\sup_{t\in T} \, \langle g,t\rangle
\ge \E\max\{\langle g,t_0\rangle,0\}
= \frac{1}{\sqrt{2\pi}}\|t_0\|_2
= \frac{1}{\sqrt{2\pi}}\rad(T)
\ge \frac{1}{2\sqrt{2\pi}}\,\eps
\ge \frac{1}{20}\,\eps\sqrt{\log N},
\]
where we used $\sqrt{\log N}\le \sqrt{\log 5}$ when $N\le 5$. Combining the two cases, for every $\eps>0$ we have
\[
w(T)\gtrsim \eps\sqrt{\log M(B_2^\dimension,\eps T^\circ)}.
\]
Taking the supremum over $\eps>0$ gives the claim.
\end{proof}

Similarly, one can use Fano's inequality, combined with the properties of the Wills functional, to obtain the usual form of the Sudakov minoration.
Again, the Sudakov minoration is classical, but is typically proved via comparison arguments, and so we emphasize the interest is primarily in the combination of information-theoretic and geometric tools in our proof.

\begin{proposition}
\label{prop:global-minoration-v2}
For any convex body $T \subset \R^\dimension$, the Gaussian width satisfies
\[
w(T) \gtrsim \sup_{\eps > 0} \eps \sqrt{\log M(T, \eps B^d_2)}.
\]
\end{proposition}
\begin{proof}
Arguing as in the proof of~\Cref{prop:global-minoration}, by adjusting constants, it suffices to show that $w(T) \gtrsim \eps \sqrt{\log M(T, \eps B^d_2)}$ for $\eps > 0$ such that $M(T, \eps B^d_2) \geq 16$. 
Let us fix such $\eps > 0$ and define
\[
\sigma = \frac{\eps^2}{8 w(T)} 
\quad \mbox{and} \quad P_\theta = \Normal{\theta}{\sigma^2 I_d}.
\]
Fano's inequality~\cite[Ch.~15]{Wai19} gives the 
following lower bound on the statistical minimax rate,
\begin{equation}\label{ineq:Fano-lower-bound}
\eps_\star(\sigma)^2 = 
\inf_{\widehat{\theta}}
\sup_{\theta\in T}
\E_{Y \sim P_\theta} \|\widehat{\theta}(Y)-\theta\|_2^2 \geq \frac{\eps^2}{4}\left(
1-\frac{\inf_Q \max_{j}
\mathcal{KL}\left(P_{\theta_j} \big\|\,Q\right)+\log 2}{\log M(T, \eps B^d_2)}\right).
\end{equation}
Above, the maximum ranges over $\theta_j$ in a maximal packing set corresponding to the packing number $M(T, \varepsilon B^d_2)$, and the infimum ranges over all probability measures $Q$ on $\R^d$. 
Arguing as in~\Cref{prop:minoration-local}, it holds that $\eps_\star(\sigma)^2 \leq \sigma w(T) \leq \eps^2/8$. Combined with~\cref{ineq:Fano-lower-bound}, we obtain,
\begin{equation}
\label{ineq:for-the-best-Q}
\frac{\eps^2}{4} \log M(T, \eps B^d_2) 
\leq \eps^2 \Big(\half \log M(T, \eps B^d_2)  - \log 2\Big) \leq 
 8 \sigma w(T) \, 
 \inf_Q \max_{j}\mathcal{KL}
 \left(P_{\theta_j} \big\|\,Q\right).
\end{equation}
Now, denoting by $\dist(y, T) = \inf_{x \in T} \|x - y\|_2$, we take $Q$ with Lebesgue density 
\[
q(y) = \frac{\exp(-\dist(y, T)^2/(2\sigma^2))}{(\sigma \sqrt{2\pi})^{d} Z_T(\sigma)}, \quad \mbox{where} \quad 
Z_T(\sigma) \defn \frac{1}{(\sigma \sqrt{2\pi})^d}\int_{\R^d} \exp(-\dist(x, T)^2/(2\sigma^2)) \, \ud x.
\]
By construction $\sup_{y \in \R^d} \tfrac{\ud P_{\theta_i}}{\ud Q}(y) \leq Z_T(\sigma)$, and hence 
\begin{equation}
\label{ineq:upper-bound-on-KL-term}
\mathcal{KL}
 \left(P_{\theta_j} \big\|\,Q\right) 
 =\E_{Y \sim P_{\theta_j}}
 \log \frac{\ud P_{\theta_j}}{\ud Q}(Y)
 \leq
 \log Z_T(\sigma) = \log W\Big(\frac{T}{\sigma\sqrt{2\pi}}\Big) 
 \leq \frac{w(T)}{\sigma}.
\end{equation}
Above, the penultimate equality arises by Hadwiger's formula for the Wills functional~\cite{hadwiger1975will}, while the final inequality is due to McMullen~\cite{mcmullen1991inequalities}.
Combining~\cref{ineq:for-the-best-Q,ineq:upper-bound-on-KL-term}, we obtain
\[
w(T) \geq \frac{1}{4\sqrt{2}} \, 
\eps \sqrt{\log M(T, \eps B^d_2)},
\]
as required.
\end{proof}

\subsection{Equivalence of ``local'' and ``global'' parameters}
\label{sec:localandglobal}

In this section, we discuss the relationship between the ``local'' versions of the results obtained in the previous sections (for instance, the local entropy integral bound in \Cref{prop:entropy-integral-comparison} and the local minoration lower bound in \Cref{prop:minoration-local}) and their more standard ``global'' counterparts, such as the usual Sudakov minoration and the usual Dudley entropy integral.
Although even for the crosspolytope local and global covering numbers are known to differ at some scales, we show that the Sudakov minoration and Dudley entropy integrals remain of the same order, for any set $T$, when the global entropy is replaced by the local entropy.
A related one-sided observation appears in~\cite[Section~4.2]{van2018chainingtwo}, where a ``local Dudley'' inequality is proved in which the generic chaining functional is upper bounded by a Dudley-type sum built from entropy numbers of local neighborhoods.

Throughout this section, it will be convenient to introduce the following notation.
Given a bounded set $T \subset \R^\dimension$, the local and global packing entropies of $T$ at scale $\eps > 0$ are defined, respectively, as
\[
\LocEnt_T(\eps) \defn \sup_{\delta \geq \eps} \sup_{x \in T}
\log M\bigl(T \cap (x + 2\delta B_2^\dimension), \delta B_2^\dimension\bigr),
\qquad\text{and}\qquad
\GlobEnt_T(\eps) \defn \log M(T, \eps B_2^\dimension),
\]
for any $\eps > 0$.
When $T$ is convex, the supremum over $\delta \geq \eps$ in the definition of $\LocEnt_T(\eps)$ is attained at $\delta=\eps$, so this definition coincides with \eqref{def:local-packing-entropy-convex}.
The starting point of our analysis is the following basic observation, which essentially follows from Yang-Barron~\cite{yang1999information}.

\begin{lemma}\label{lem:dyadic-local-entropy-ineq}
Let $T \subset \R^\dimension$ be a bounded set, and let $\Delta \defn \diam(T)$. Then, for any integer $k \geq 1$, it holds that
\begin{equation}\label{ineq:dyadic-local-entropy-bound}
\LocEnt_T(\Delta 2^{-k})
\leq \GlobEnt_T(\Delta 2^{-k}) \leq
\sum_{j=1}^k \LocEnt_T(\Delta 2^{-j}).
\end{equation}
\end{lemma}
\begin{proof}
The first inequality in~\eqref{ineq:dyadic-local-entropy-bound} is immediate, so we focus on the second.
We recall the relationship between covering and packing numbers~\cite[Ch.~5]{Wai19}:
\begin{equation}\label{eqn:covering-to-packing}
N(A, \eps B_2^\dimension) \leq M(A, \eps B_2^\dimension) \leq N(A, \tfrac{\eps}{2} B_2^\dimension),
\qquad \mbox{for any } A \subset \R^\dimension \mbox{ and } \eps > 0.
\end{equation}
Fix $\delta>0$, and let $\cM$ be a maximal $\delta$-packing of $T$ with respect to $\ell_2^\dimension$.
Let $\{\theta^{(i)}\}_{i \in [N]} \subset T$ be a $2\delta$-net for $T$ with respect to $\ell_2^\dimension$ of minimal cardinality, so that
$N = N(T,2\delta B_2^\dimension)$.
Then
\begin{equation}
|\cM|
\leq \sum_{i=1}^N \big|\cM \cap (\theta^{(i)} + 2\delta B_2^\dimension)\big|
\leq N \max_{i \in [N]} \big|\cM \cap (\theta^{(i)} + 2\delta B_2^\dimension)\big| 
\leq N \sup_{\theta \in T} M\bigl(T \cap (\theta + 2\delta B_2^\dimension), \delta B_2^\dimension\bigr),
\end{equation}
where the last inequality uses that $\cM \cap (\theta^{(i)} + 2\delta B_2^\dimension)$ is $\delta$-separated, and hence its cardinality is at most the corresponding local packing number.
Taking the supremum over all $\delta$-packings $\cM$ yields
\[
M(T,\delta B_2^\dimension)
\leq N(T,2\delta B_2^\dimension)\cdot \sup_{\theta \in T} M\bigl(T \cap (\theta + 2\delta B_2^\dimension), \delta B_2^\dimension\bigr).
\]
Since $N(T,2\delta B_2^\dimension)\le M(T,2\delta B_2^\dimension)$ by~\eqref{eqn:covering-to-packing}, we obtain
\begin{equation}\label{ineq:global-to-local-entropy}
\GlobEnt_T(\delta)-\GlobEnt_T(2\delta)\le \LocEnt_T(\delta),
\qquad \mbox{for any } \delta>0.
\end{equation}
Applying~\eqref{ineq:global-to-local-entropy} with $\delta=\Delta 2^{-j}$ for $j\ge 1$ and summing over $j\in[k]$ gives
\[
\GlobEnt_T(\Delta 2^{-k})-\GlobEnt_T(\Delta)
\le \sum_{j=1}^k \LocEnt_T(\Delta 2^{-j}).
\]
Finally, we have $M(T,\Delta B_2^\dimension)=1$ and hence $\GlobEnt_T(\Delta)=0$, which completes the proof.
\end{proof}

We can now use \Cref{lem:dyadic-local-entropy-ineq} to show that the ``local'' and ``global'' versions of the Dudley entropy integral and Sudakov minorations are equivalent up to universal constant factors. To state the first result, we recall the local and global Dudley entropy integrals, which are respectively given by
\[
\TruncIntegral_\delta^{\rm loc}(T) \defn \int_\delta^{\diam(T)} \sqrt{\LocEnt_T(\eps)} \,\ud \eps,
\quad \mbox{and} \quad
\TruncIntegral_\delta(T) \defn \int_\delta^{\diam(T)} \sqrt{\GlobEnt_T(\eps)} \, \ud \eps,
\]
for $\delta \in [0,\diam(T)]$.

\begin{proposition}[Equivalence of local and global Dudley integrals]
\label{prop:dudleygloballocal}
For any bounded set $T \subset \R^\dimension$, it holds that
\[
\TruncIntegral_\delta^{\rm loc}(T) \leq \TruncIntegral_\delta(T)
\leq 4\; \TruncIntegral_{\delta/4}^{\rm loc}(T),
\]
for any $\delta \in [0,\diam(T)]$.
\end{proposition}
\begin{proof}
By rescaling, assume without loss of generality that $\diam(T)=1$.
The first inequality is immediate.
For the second, using that $\GlobEnt_T$ is nonincreasing, fix $\eps\in(0,1]$ and let $k=\ceil{\log_2(1/\eps)}$.
Then \Cref{lem:dyadic-local-entropy-ineq} and subadditivity of the square root give
\begin{equation}\label{ineq:eps-global-to-local-ineq}
\sqrt{\GlobEnt_T(\eps)}
\le \sqrt{\GlobEnt_T(2^{-k})}
\le \sum_{j=1}^k \sqrt{\LocEnt_T(2^{-j})}
\le \sum_{j=1}^\infty \ind{\eps\le 2^{1-j}}\sqrt{\LocEnt_T(2^{-j})}.
\end{equation}
Integrating~\eqref{ineq:eps-global-to-local-ineq} over $\eps\in[\delta,1]$ yields
\[
\TruncIntegral_\delta(T)
\le \sum_{j=1}^\infty (2^{1-j}-\delta)_+\,\sqrt{\LocEnt_T(2^{-j})}.
\]
For each $j\ge 1$, monotonicity of $\LocEnt_T$ implies
\[
2^{-(j+1)}\sqrt{\LocEnt_T(2^{-j})}
\le \int_{2^{-(j+1)}}^{2^{-j}} \sqrt{\LocEnt_T(\eps)}\,\ud\eps,
\]
and hence $(2^{1-j}-\delta)_+\sqrt{\LocEnt_T(2^{-j})}\le 4\int_{2^{-(j+1)}}^{2^{-j}}\sqrt{\LocEnt_T(\eps)}\,\ud\eps$.
Summing over $j$ for which $2^{1-j}>\delta$ gives
\[
\TruncIntegral_\delta(T)
\le 4 \int_{2^{-(N_\delta+1)}}^{1}\sqrt{\LocEnt_T(\eps)}\,\ud\eps
\le 4 \int_{\delta/4}^{1}\sqrt{\LocEnt_T(\eps)}\,\ud\eps
=4\,\TruncIntegral_{\delta/4}^{\rm loc}(T),
\]
where $N_\delta\defn \floor{\log_2(2/\delta)}$.
\end{proof}

We now introduce two versions of the minoration,
\[
\Minoration^{\rm loc}(T) \defn \sup_{\eps>0}\; \eps \sqrt{\LocEnt_T(\eps)},
\qquad\text{and}\qquad
\Minoration(T) \defn \sup_{\eps>0}\; \eps \sqrt{\GlobEnt_T(\eps)}.
\]

\begin{proposition}[Equivalence of local and global Sudakov minoration]
\label{prop:sudakovlocalglobal}
For any bounded set $T \subset \R^\dimension$, we have
\[
\Minoration^{\rm loc}(T) \le \Minoration(T) \le 4\,\Minoration^{\rm loc}(T).
\]
\end{proposition}
In particular, this result combined with \Cref{prop:minoration-local} immediately implies the usual Sudakov minoration inequality for any bounded (not necessarily convex due to \Cref{rem:sudakov}) closed set $T \subset \R^\dimension$:
\[
w(T) \gtrsim \sup_{\eps > 0} \eps \sqrt{\log M(T, \eps B_2^\dimension)}.
\]
\begin{proof}
By rescaling, assume without loss that $\diam(T)=1$.
The first inequality is immediate.
For the second, using~\eqref{ineq:eps-global-to-local-ineq}, for $\eps\in(0,1]$ and $k=\ceil{\log_2(1/\eps)}$,
\[
\eps \sqrt{\GlobEnt_T(\eps)}
\le \eps \sum_{j=1}^{k} \sqrt{\LocEnt_T(2^{-j})}
\le \eps \sum_{j=1}^{k} 2^j \cdot \underbrace{2^{-j}\sqrt{\LocEnt_T(2^{-j})}}_{\le \Minoration^{\rm loc}(T)}
\le \eps(2^{k+1}-1)\Minoration^{\rm loc}(T)
\le 4\Minoration^{\rm loc}(T).
\]
Taking the supremum over $\eps>0$ yields the claim.
\end{proof}

\subsection{Comparisons with generic chaining}
\label{sec:generic-chaining}
One of the most natural questions is to relate Talagrand's $\gamma_2$ functional to the decompositions implied by \Cref{thm:decomp-of-gauss-width}. Recall that by the Fernique-Talagrand majorizing measures theorem \cite[Theorem 2.10.1]{Talagrand2021} we have $w(T) \asymp \gamma_2(T)$,
where we use the formulation in terms of admissible {partitions}.
Concretely, an admissible sequence of partitions of $T$ is a nested family
$\mathcal{A}=(\mathcal{A}_n)_{n\ge 0}$ such that $\mathcal{A}_0=\{T\}$, $\mathcal{A}_{n+1}$ refines $\mathcal{A}_n$ for every $n$, and $|\mathcal{A}_n|\le 2^{2^n}$ for all $n\ge 0$.
For $x\in T$, let $A_n(x)\in \mathcal{A}_n$ denote the unique cell containing $x$. 
Then (\cite[Definition 2.7.3]{Talagrand2021}),
\[
\gamma_2(T)
= \inf_{\mathcal{A}} \; \sup_{x \in T}\sum_{n=0}^\infty 2^{n/2}\,\diam\bigl(A_n(x)\bigr),
\]
where the infimum is over all admissible sequences of partitions $\mathcal{A}$ of $T$. We now introduce truncated generic-chaining quantities in the language of admissible partitions.
Fix an admissible sequence of partitions $\mathcal{A}=(\mathcal{A}_n)_{n\ge 0}$ and an integer $p\ge 1$, and define the corresponding \emph{first $p$ levels} and \emph{tail} contributions by
\[
\gamma_{2,<p}(T;\mathcal{A})
= \sup_{x\in T}\sum_{n=0}^{p-1} 2^{n/2}\diam\bigl(A_n(x)\bigr),
\quad \textrm{and}\quad
\gamma_{2, \ge p}(T;\mathcal{A})
= \sup_{x\in T}\sum_{n=p}^{\infty} 2^{n/2}\diam\bigl(A_n(x)\bigr).
\]
We set $\gamma_{2,<0}(T;\mathcal{A}) = 0$. Finally, we define the corresponding {optimized} truncated functionals
\[
\gamma_{2,<p}(T) = \inf_{\mathcal{A}} \gamma_{2,<p}(T;\mathcal{A}),
\qquad
\gamma_{2, \ge p}(T) = \inf_{\mathcal{A}} \gamma_{2, \ge p}(T;\mathcal{A}),
\]
where the infima are over admissible sequences of partitions of $T$. We remark that corresponding truncated functionals appear naturally in the literature involving the analysis of generic chaining functionals \cite{dirksen2015tail,mendelson2016upper, mourtada2025universal}. It appears that the format of our result and corresponding derivations simplify significantly if we focus on a decomposition of \Cref{cor:gausstowills}, namely
\begin{equation}
\label{eq:gaussianwidthviawills}
w(T) \asymp \sigma\log\left(W\left(\frac{T}{\sigma\sqrt{2\pi}}\right)\right) + \int_{\sigma}^{\infty}\frac{r(\nu)^2}{\nu^2}\ud\nu.
\end{equation}
Since in the context of generic chaining we can only capture the Gaussian width up to multiplicative constant factors, the bound \eqref{eq:gaussianwidthviawills} is meaningful, even though its decomposition is not exact as in \Cref{thm:decomp-of-gauss-width}.
\begin{proposition}
    Given a compact convex set $T \subset \R^d$ containing the origin and $\sigma > 0$, set 
    \[
p_{\sigma} = \argmin_{p \ge 0}\left\{\sigma(2^p-1)+\gamma_{2, \ge p}(T)\right\}.
    \]
    Then, it holds that
    \[
    w(T) \asymp \sigma\log\left(W\left(\frac{T}{\sigma\sqrt{2\pi}}\right)\right) + \gamma_{2, <p_{\sigma}}(T),
    \]
    and
    \[
    \sigma\log\left(W\left(\frac{T}{\sigma\sqrt{2\pi}}\right)\right) \asymp \gamma_{2, \ge p_{\sigma}}(T) + \sigma(2^{p_\sigma}-1).
    \]
\end{proposition}
\begin{remark}
The proposition does not identify the integral term $\int_{\sigma}^{\infty}\frac{r(\nu)^2}{\nu^2}\ud\nu$ with $\gamma_{2,<p_{\sigma}}(T)$. Rather, it shows that, up to universal constants, the Wills term captures the large-scale (tail) contribution of generic chaining, while $\gamma_{2,<p_\sigma}(T)$ can be viewed as a small-scale complement in the sense of \eqref{eq:gaussianwidthviawills}.
\end{remark}

\begin{proof}
    First, by Corollary 4.1 in \cite{mourtada2025universal} with the only observation that $\sigma\gamma_{2, \ge p}(T/\sigma) = \gamma_{2, \ge p}(T)$ we have
    \begin{equation}
    \label{eq:willstochaining}
    \sigma\log\left(W\left(\frac{T}{\sigma\sqrt{2\pi}}\right)\right) \asymp \gamma_{2, \ge p_{\sigma}}(T) + \sigma(2^{p_\sigma}-1).
    \end{equation}
    Next, by the majorizing measures theorem we have for some absolute constant $c > 0$,
\[
w(T)
\le c \inf_{\mathcal{A}}
\sup_{x \in T} \sum_{n=0}^{\infty}
2^{n/2}\diam\bigl(A_n(x)\bigr).
\]
Now we argue that
    \[
    w(T) \le c\bigl(\gamma_{2, < p_{\sigma}}(T) + \gamma_{2, \ge p_{\sigma}}(T)\bigr).
    \]
    This step is almost obvious, but requires some technicalities to merge together two admissible sequences corresponding to $\gamma_{2, < p_{\sigma}}(T)$ and $\gamma_{2, \ge p_{\sigma}}(T)$ respectively.
Set $p=p_\sigma$. If $p=0$, this is immediate from \eqref{eq:majorizingmeasure}. Let $\mathcal A^-=(\mathcal A^-_n)_{n\ge 0}$ and $\mathcal A^+=(\mathcal A^+_n)_{n\ge 0}$ be arbitrary admissible sequences of partitions of $T$.
Define a new sequence $\mathcal A=(\mathcal A_n)_{n\ge 0}$ by
\[
\mathcal{A}_n = 
\begin{cases} 
\mathcal A^-_n, & 0\le n\le p-1,\\
\mathcal A^-_{p-1}, & n = p \\ 
\{A\cap B:A\in \mathcal A^-_{p-1},B\in \mathcal A^+_{n-1}\}, & n\ge p+1.
\end{cases}
\]
Then, $\mathcal A$ is admissible. Indeed, nesting is clear, and for $n\ge p+1$,
$|\mathcal A_n|\le |\mathcal A^-_{p-1}||\mathcal A^+_{n-1}|
\le 2^{2^{p-1}}2^{2^{n-1}}\le 2^{2^n}$.
Moreover, for $n\ge p+1$ we have $\diam(A_n(x))\le \diam(A^+_{n-1}(x))$.
Therefore, for every $x\in T$,
\begin{align}
\sum_{n=0}^\infty 2^{n/2}\diam\bigl(A_n(x)\bigr)
&= \sum_{n=0}^{p-1} 2^{n/2}\diam\bigl(A^-_n(x)\bigr)
+2^{p/2}\diam\bigl(A^-_{p-1}(x)\bigr)
+\sum_{n=p+1}^\infty 2^{n/2}\diam\bigl(A_n(x)\bigr)\\
&\le \gamma_{2,<p}(T;\mathcal A^-)
+\sqrt{2}\cdot2^{(p-1)/2}\diam\bigl(A^-_{p-1}(x)\bigr)
+\sqrt{2}\sum_{m=p}^\infty 2^{m/2}\diam\bigl(A^+_m(x)\bigr)\\
&\le (1+\sqrt{2})\gamma_{2,<p}(T;\mathcal A^-)
+\sqrt{2}\gamma_{2,\ge p}(T;\mathcal A^+). 
\end{align}
Passing to the supremum over $x \in T$, we obtain
\[
\sup_{x\in T}\sum_{n=0}^\infty 2^{n/2}\diam\bigl(A_n(x)\bigr)
\le (1+\sqrt{2})\gamma_{2,<p}(T;\mathcal A^-)
+\sqrt{2}\gamma_{2,\ge p}(T;\mathcal A^+).
\]
We emphasize that $\mathcal A$ is obtained by gluing together the two generic admissible sequences $\mathcal A^-$ and $\mathcal A^+$. Plugging this $\mathcal A$ into the infimum over admissible sequences and then taking the infimum over $\mathcal A^-$ and $\mathcal A^+$ (separately) gives
\[
\inf_{\mathcal A}\sup_{x\in T}\sum_{n=0}^\infty 2^{n/2}\diam\bigl(A_n(x)\bigr)
\le (1+\sqrt{2})\gamma_{2,<p}(T)+\sqrt{2}\gamma_{2,\ge p}(T),
\]
and hence $w(T)\le c\bigl(\gamma_{2,<p_\sigma}(T)+\gamma_{2,\ge p_\sigma}(T)\bigr)$ after adjusting absolute constants. By \eqref{eq:willstochaining} this simplifies for some $c_1 > 0$,
\[
w(T) \le c_1\left(\sigma\log\left(W\left(\frac{T}{\sigma\sqrt{2\pi}}\right)\right) + \gamma_{2, < p_{\sigma}}(T)\right).
\]
For the lower bound, from McMullen's inequality~\eqref{eq:mcmullen}, we have $\sigma\log\left(W\left(\frac{T}{\sigma\sqrt{2\pi}}\right)\right)\le c_2 \, w(T)$.  The lower bound in \eqref{eq:majorizingmeasure} yields $\gamma_{2, <p_{\sigma}}(T)\le \gamma_2(T)\le c_3 \, w(T)$. Together, these give
\[
\sigma\log\left(W\left(\frac{T}{\sigma\sqrt{2\pi}}\right)\right)+\gamma_{2, <p_{\sigma}}(T)\le c_4 \, w(T),
\]
and the proof is complete.
 \end{proof}

\subsection{Upper bounds based on the Donsker-Varadhan variational identity}

One of the notable techniques for bounding the first term in \cref{eqn:projection-decomposition-sig-pos} in the Gaussian case relies on variational methods based on the Donsker--Varadhan duality formula.
This identity appears in universal coding and sequential prediction under logarithmic loss, where it underlies normalized maximum-likelihood constructions and minimax regret \cite{shtar1987universal,polyanskiy2025information}; see also the regret viewpoint in \cite{mourtada2025universal}.
In statistics, this approach has been developed extensively by Catoni and co-authors, mainly in the context of robust estimation \cite{catoni2007pac,audibert2011robust,catoni2017dimension}.
Recently, it was shown in \cite{zhivotovskiy2024dimension} that this method yields a sharp bound on the Gaussian width of ellipsoids, a case where the Dudley entropy integral is known to be suboptimal.

What remained unclear was whether the same variational method could be extended to other convex bodies to bound the Gaussian width.
We demonstrate that, at least for canonical examples such as ellipsoids and crosspolytopes, it does provide not only the optimal bounds on the Gaussian width but a sharp bound on the smaller first term in our decompositions.
In what follows, we present a self-contained exposition of this approach.

\begin{proposition}
\label{prop:simplepb}
Let $\mu$ be any distribution on $\R^d$ and let $g\sim \Normal{0}{I_d}$. Then
\[
\E\left[\sup_{\rho \ll \mu}\left(\E_{x \sim \rho}\left[\langle x, g\rangle - \|x\|^2/2\right] - \mathcal{KL}(\rho\,\|\, \mu)\right)\right] \le 0.
\]
\end{proposition}
\begin{proof}
Using Jensen's inequality, the Donsker--Varadhan variational formula \cite[Corollary~4.14]{boucheron2013concentration}, and Fubini's theorem, we have
\begin{align}
&\E\sup_{\rho \ll \mu}\Big(\E_{x \sim \rho}[\langle x,g\rangle - \|x\|^2/2]-\mathcal{KL}(\rho\,\|\,\mu)\Big) \\
&\le\log\left(\E\sup_{\rho \ll \mu}\exp\Big(\E_{x \sim \rho}[\langle x,g\rangle - \|x\|^2/2]-\mathcal{KL}(\rho\,\|\,\mu)\Big)\right) \\
&= \log\left(\E \E_{x\sim \mu}\exp(\langle x,g\rangle - \|x\|^2/2)\right) \\
&= \log \E_{x\sim \mu}\E \exp(\langle x,g\rangle - \|x\|^2/2)
= 0.
\end{align}
The claim follows.
\end{proof}

\begin{remark}
The result of \Cref{prop:simplepb} can also be established through connections with sequential probability assignment in the Gaussian setting.
A key observation is that $\log W(\cdot)$ admits an interpretation in terms of minimax regret via Shtarkov's integral representation \cite{shtar1987universal}; see \cite{mourtada2025universal} for a detailed account in the present geometric setting.
Importantly, \Cref{thm:localpb} shows that in our context the bound of \Cref{prop:simplepb} can be systematically improved.
\end{remark}

The idea of this approach is the following: once we choose the distribution $\mu$ (depending on $T$), the Donsker--Varadhan formula implies that
\[
\sup_{\rho \ll \mu}\left(\E_{x \sim \rho}\left[\langle x, g\rangle - \|x\|^2/2\right] - \mathcal{KL}(\rho\,\|\, \mu)\right)
\]
is attained at the Gibbs posterior $\widetilde{\rho}$ given by
\[
\widetilde{\rho}(\ud x)\ \propto\ \exp(\langle x, g\rangle - \|x\|^2/2)\,\mu(\ud x).
\]
Our aim will be to choose $\mu$ so that
\begin{equation}
\label{eq:pbaim}
\sup_{x \in T}\left(\langle x,g\rangle - \|x\|^2/2\right)
\le \E_{x \sim \widetilde{\rho}}\big[\langle x,g\rangle - \|x\|^2/2\big] + \Delta,
\end{equation}
where $\Delta$ is a nonnegative complexity term. This yields the desired bound on the first term in both decompositions. To motivate our next result, we first discuss the application of this method to ellipsoids in $\R^d$,
\[
\cE_a = \Big\{\, x \in \R^d: \sum_{i = 1}^d\frac{x_i^2}{a_i^2} \le 1\,\Big\}.
\]
We note that a technically more interesting computation for the crosspolytope is given in \Cref{sec:pacbayesellone}. 

\begin{corollary}[Logarithm of the Wills functional of the ellipsoid via the variational approach]
\label{cor:ellipsoid}
Fix $\lambda > 0$. 
For the ellipsoid $\cE_a \subset \R^d$, and with the choice
$
\mu = \Normal{0}{ \sum_{i=1}^d \frac{a_i^{2}}{\lambda}e_i e_i^\top},
$ it holds that
\[
\E\left[\sup_{x \in \cE_a}\left(\langle x, g\rangle -\|x\|^2/2\right)\right]
\le
\log\left(\E\exp\left[\sup_{x \in \cE_a}\left(\langle x, g\rangle -\|x\|^2/2\right)\right]\right)
\le
\frac{\lambda}{2} + \frac{1}{2}\sum_{i = 1}^d\log\left(1 + \frac{a_i^2}{\lambda}\right).
\]
\end{corollary}

Note that \Cref{cor:ellipsoid} bounds a larger quantity: via \eqref{eq:willsbound}, it yields an upper bound on $\log W(\cE_a/\sqrt{2\pi})$ (up to the factor $1/\sqrt{2\pi}$), rather than directly bounding the first term in \cref{eqn:projection-decomposition-sig-pos}. Moreover, by the lower bound in \cite[Proposition~7.1]{mourtada2025universal}, this estimate is optimal up to multiplicative constants for the logarithm of the Wills functional of the ellipsoid.

\begin{remark}[Width of the ellipsoid]
\label{rem:width-of-the-ellipse}
Observe that by a rescaling argument, the bound in~\Cref{cor:ellipsoid} is already enough to recover the Gaussian width of the ellipsoid. Indeed, in the context of \Cref{thm:decomp-of-gauss-width} we consider
\[
w\big(\cE_a \cap \ChatFixed(\sigma) B^d_2\big) - \frac{\ChatFixed(\sigma)^2}{2\sigma}
= \sigma \left(w\left(\frac{\cE_a}{\sigma} \cap \frac{\ChatFixed(\sigma)}{\sigma} B^d_2\right) - \frac{\ChatFixed(\sigma)^2}{2\sigma^2}\right)
\le \sigma\,\E\sup_{x \in \cE_{a/\sigma}}\left(\langle x, g\rangle - \frac{\|x\|^2}{2}\right).
\]
Applying
\Cref{cor:ellipsoid} to $\cE_{a/\sigma}$ then yields
\[
\lim_{\sigma \to \infty}\sigma\,\E\sup_{x \in \cE_{a/\sigma}}\left(\langle x, g\rangle - \frac{\|x\|^2}{2}\right)
\le \lim_{\sigma \to \infty}\inf_{\lambda > 0}\left(\frac{\lambda\sigma}{2} + \frac{\sigma}{2}\sum_{i = 1}^d\log\left(1 + \frac{a_i^2}{\lambda\sigma^2}\right)\right)
= \sqrt{\sum_{i = 1}^d a_i^2},
\]
which is a sharp upper bound on the Gaussian width of the ellipsoid.
\end{remark}

Although sufficient to bound the Gaussian width of the ellipsoid, the estimate in \Cref{prop:simplepb} is suboptimal. We now introduce another tool that at least for the ellipsoid yields a sharp bound on the first term in \cref{eqn:projection-decomposition-sig-pos}.
Rather than directly invoking chaining arguments, we again employ the variational approach, but now with localized priors, introduced in the statistical framework by Catoni \cite{catoni2007pac} (see also \cite{mourtada2023local}).

\begin{theorem}
\label{thm:localpb}
Let $\mu$ be a distribution on $\R^d$ and let $g\sim \Normal{0}{I_d}$. Define the localized ($g$-dependent) distribution
\[
\mu_{\operatorname{loc}}(\ud x)=
\frac{\exp(-\|x-g\|^2/8)\,\mu(\ud x)}
{\int \exp(-\|x'-g\|^2/8)\,\mu(\ud x')}.
\]
Then
\[
\E\Bigl[\sup_{\rho\ll \mu}\Bigl(\E_{x'\sim \rho}\bigl[\langle x',g\rangle-\|x'\|^2/2\bigr]-2\,\mathcal{KL}(\rho\,\|\,\mu_{\operatorname{loc}})\Bigr)\Bigr] \le 0.
\]
\end{theorem}

Unlike \Cref{prop:simplepb}, whose proof only uses the moment generating function of \(g\) and extends to sub-Gaussian \(g\), the proof of \Cref{thm:localpb}, deferred to \Cref{sec:proofoflocalpb}, critically exploits Gaussianity through an Ornstein-Uhlenbeck semigroup argument.
We now apply \Cref{thm:localpb} in a style resembling the proof of \Cref{cor:ellipsoid}.

\begin{corollary}[Local width of the ellipsoid via the variational approach]
\label{cor:localellipsoid}
Fix $\lambda > 0$. For the ellipsoid $\cE_a$, and with the choice of measure
$
\mu = \Normal{0}{\sum_{i=1}^d \frac{a_i^{2}}{\lambda}\, e_i e_i^\top},
$
it holds that
\[
\E\left[\sup_{x \in \cE_a}\left(\langle x, g\rangle -\|x\|^2/2\right)\right]
\le 2\lambda + \frac{1}{2}\sum_{i = 1}^d\frac{4a_i^2}{a_i^2 + 4\lambda}.
\]
\end{corollary}
Again, optimizing over $\lambda>0$ in~\Cref{cor:localellipsoid} yields an optimal bound directly for the quantity
$\E \, \sup_{x \in \cE_a}(\langle x, g\rangle -\|x\|^2/2)$. Moreover, when combined with \Cref{prop:equivalentfixedpoints} it also sharply characterizes the fixed points as calculated in \cite{mendelson2003performance} and \cite[Proposition~3.3]{koltchinskii2011oracle}.

\subsection{Remaining proofs of \Cref{sec:info-and-stat}}
\subsubsection{Proof of ~\Cref{thm:upper-bound-minimax}}
\label{sec:proof-of-prop-upper-minimax}

To begin with, we make use of a statistical interpretation of the radius $r(\sigma)$, which originally appeared in the work of Chatterjee~\cite{Cha14}.
It relates the radius to the risk of the least squares estimate (LSE), or equivalently, the Euclidean projection onto $T$, which is given by
\[
\Pi_T(y) \defn \argmin_{\vartheta \in T} 
\|y - \vartheta\|_2^2.
\]

To control $r(\sigma)$ we first need the following upper bound on the worst-case variance of the LSE, which is proved in~\Cref{sec:proof-of-variance-bound} via a minor modification of the ideas in the paper~\cite{KurEtal23}.

\begin{lemma}[Worst-case variance of the LSE]
\label{lem:variance-bound-for-erm}
For any nonempty, closed convex set $T \subset \R^\dimension$, it holds that 
\begin{equation}\label{ineq:worst-case-variance-bound}
\sup_{\theta \in T} \E_{Y \sim \Normal{\theta}{\sigma^2 I_\dimension}}\Big[
\big\|\Pi_T(Y) - \E[\Pi_T(Y)]\big\|_2^2\Big]
\lesssim  \eps_\star^2(\sigma),
\end{equation}
for any $\sigma > 0$.
\end{lemma} 

Now, to bound the variational radii $r(\sigma)$, we know from Chatterjee's work~\cite{Cha14} that there is $c_1 > 1$ sufficiently large that if $r(\sigma) \geq c_1 \sigma$, then $r(\sigma)$ controls the error of the LSE at $\theta = 0$:
\begin{equation}    \label{eq:rsigmatoprojection}
    r(\sigma)^2 \leq c_2 \E_{g \sim \Normal{0}{\sigma^2 I_\dimension}} \Big[\big\|\Pi_T(g) - \E[\Pi_T(g)]\big\|_2^2\Big] \leq 
    c_3 \eps_\star^2(\sigma).
    \end{equation}
    Above, the final inequality follows from \Cref{lem:variance-bound-for-erm}; $c_2, c_3 > 0$ are universal constants. 
    On the other hand, if  $r(\sigma) \leq c_1 \sigma$, then since $T$ contains the origin, we have $r(\sigma) \leq \rad(T)$, and it follows that $r(\sigma) \leq c_1 (\twomin{\sigma}{\rad(T)}) \leq c_4\eps_\star(\sigma)$; see \Cref{lem:trivial-lower-bound} in \Cref{app:GSM}. Hence, $r(\sigma) \leq C \eps_\star(\sigma)$ with $C \defn \twomax{\sqrt{c_3}}{c_4}$, and the first inequality in~\cref{eqn:results-via-minimax-rates} follows. For the second inequality, we simply use~\Cref{lem:variance-bound-for-erm}, and note that $\E \Pi_T(Y) = 0$ when $Y \sim \Normal{0}{\sigma^2 I_d}$ by central symmetry. The claim follows. \qed

\subsubsection{Proof of \Cref{lem:variance-bound-for-erm}}
\label{sec:proof-of-variance-bound}

The argument closely follows~\cite[Theorem 1]{KurEtal23}. Let $g \sim \Normal{0}{I_\dimension}$, and define
\[
\hat \mu(g) \defn \Pi_T(\theta + \sigma g), 
\quad 
h_{\eta}(g) \defn 
\|\hat \mu(g) - \eta\|_2, 
\quad 
\mbox{and} \quad 
\mu \defn \E \hat \mu(g).
\]
The variance is $\delta^2 \defn \E \|\hat \mu(g) - \mu\|_2^2$. 
Our goal is to establish:
\begin{equation}\label{ineq:desired-bound-gil}
\delta^2 \leq 128991 \, \eps^2_\star(\sigma).
\end{equation}
Let $\cM_\delta$ denote 
a maximal $(\delta/\sqrt{8})$-packing of $T \cap B(\mu, 2\delta)$. 
First, suppose that 
$\delta^2 \leq 4(4 \log(2) + 1) \sigma^2$. In this case, \Cref{lem:trivial-lower-bound} yields
\begin{subequations}
\begin{equation}\label{eqn:case-delta-small}
\delta^2 \leq 4(4 \log(2) + 1)  \min\{\sigma^2, \diam(T)^2\} \leq 242 \, \eps_\star^2(\sigma).
\end{equation}
Now suppose that $\delta^2 > 4(4 \log(2) + 1) \sigma^2$. 
For the sake of contradiction, suppose that 
$\sigma^2 \log |\cM_\delta| < \tfrac{\delta^2}{16}$. 
Note that for any $\eta \in \cM_\delta$, the map $h_\eta$ is $\sigma$-Lipschitz, so 
\[
\P\Big\{ \E h_\eta(g) - h_\eta(g) >   \sigma \sqrt{2 \log(2 |\GenPackSet{\delta}|)} \Big\} \leq \frac{1}{2 |\GenPackSet{\delta}|}.
\]
Markov's inequality yields
\[
\P\Big\{ \hat \mu(g) \in T \cap (\mu + 2\delta B^d_2)\Big\} \geq 1 - \P\Big\{\|\hat \mu(g) - \mu\|_2^2 > 4 \delta^2\Big\} \geq \frac{3}{4}.
\]
Therefore, for some $\etastar \in \cM_\delta$, we have 
that $\hat \mu(g) \in \etastar + \tfrac{\delta}{\sqrt{8}} B^d_2$ with probability at least $\tfrac{3}{4|\GenPackSet{\delta}|}$.
A union bound then gives us that 
\[
\P\Big\{ \E h_{\etastar}(g) -h_{\etastar}(g) \leq  \sigma \sqrt{2 \log(2 |\GenPackSet{\delta}|)},~ 
\mbox{and}~h_{\etastar}(g) 
\leq \frac{\delta}{\sqrt{8}} \Big\} 
\geq 
\frac{1}{4 |\GenPackSet{\delta}|} > 0.
\]
Hence we have 
\[
\Big(\E h_{\etastar}(g) \Big)^2 
\leq \Big( \sigma \sqrt{2 \log(2 |\GenPackSet{\delta}|)} + \frac{\delta}{\sqrt{8}}\Big)^2 
\leq \frac{\delta^2}{4} + 
(4 \log 2) \sigma^2 
+ 4 \sigma^2 \log |\GenPackSet{\delta}|,
\] 
where we used the scalar inequality $(x+y)^2 \leq 2(x^2 + y^2)$. We can now derive a contradiction:
\[
\delta^2 = \inf_{\eta \in \R^\dimension} 
\E \big[h_{\eta}(g)^2\big] \leq
(\E h_\etastar(g))^2 + \Var(h_\etastar) \leq 
\frac{\delta^2}{4} + 
(4\log 2 +1)\sigma^2 
+ 
4 \sigma^2 \log |\GenPackSet{\delta}| \leq \frac{3\delta^2}{4} < \delta^2,
\]
where we used the Gaussian Poincaré inequality~\cite[Proposition 4.1.1]{BakGenLed14} to conclude 
$\Var(h_\etastar) \leq \sigma^2$. 
Consequently, we have 
\[
\frac{\delta^2}{16 \sigma^2} \leq 
\log |\GenPackSet{\delta}| = 
\log M\Big(T \cap (\mu + 2\delta B^d_2), \frac{\delta}{\sqrt{8}} B^d_2) 
\leq 
\sup_{\mu \in T}
\log M\Big(T \cap (\mu + 2\delta B^d_2), \frac{\delta}{\sqrt{8}} B^d_2\Big).
\]
Hence, by \Cref{lem:equivalent-versions-of-the-minimax-rate} with $c_1 = 16, c_2 = \sqrt{8}, c_3 = 2$, we have 
\begin{equation}\label{eqn:case-delta-large}
\delta^2 \leq 128991 \, \eps_\star^2(\sigma).
\end{equation}
\end{subequations}
Combining the cases~\eqref{eqn:case-delta-small} and~\eqref{eqn:case-delta-large} completes the proof and establishes the desired inequality~\eqref{ineq:desired-bound-gil}.

\subsubsection{Proof of \Cref{prop:entropy-integral-comparison}}
\label{sec:proof-of-entropy-characterization}

\paragraph{Proof of equation~\eqref{eqn:relations-for-series}:} We use the shorthand notation
\begin{equation}
\begin{gathered}
h(\eps) \defn \log M^{\rm loc}_T(\eps), 
\quad 
f(\eps) \defn \frac{\eps}{\sqrt{h(\eps)}} \ind{h(\eps) \neq 0},\\
\eps(\sigma) \defn \sup\{\eps > 0, h(\eps) \neq 0 : f(\eps) \leq \sigma \}, 
\quad \mbox{and,} \quad 
\SeriesFixedPoint(\sigma) \defn 
\int_\sigma^\infty \frac{\eps(\nu)^2}{\nu^2} \, \ud \nu.
\end{gathered}
\end{equation}
Throughout we assume that for some $\eps > 0$, $h(\eps) > 0$; otherwise, the claim is trivial. For any $\tau > 0$ and $\delta >0$, we define the following truncated integrals, 
\[
I_{\tau}^{\geq}(\delta) 
\defn 
\int_{\tau}^\infty 
\ind{\eps(\nu) \geq \delta} 
\, \frac{1}{\nu^2} \, \ud \nu, 
\quad \mbox{and} \quad 
I_\tau^{>}(\delta) 
\defn 
\int_{\tau}^\infty 
\ind{\eps(\nu) > \delta} 
\, \frac{1}{\nu^2} \, \ud \nu. 
\]
These integrals satisfy the following simple relations.
\begin{lemma} 
\label{lem:truncated-integral-bounds}
For every $\tau > 0$ and every $\delta > 0$, we have 
\[
I^{\geq}_\tau(\delta) 
\geq \min\Big\{
\frac{1}{\tau}, \frac{\sqrt{h(\delta)}}{\delta}\Big\} 
\geq 
I^{>}_\tau(\delta).
\]
\end{lemma} 
The elementary identity 
\(\eps^2 = \int_0^\eps 2\delta \, \ud \delta\) and Fubini's theorem give
\begin{equation}\label{eqn:fubini-identity}
\int_{\tau}^\infty \frac{\eps^2(\nu)}{\nu^2} \, \ud \nu = \int_0^{\diam(T)} 2\delta I_\tau^{\geq}(\delta) \, \ud \delta = 
\int_0^{\diam(T)} 2\delta I_\tau^{>}(\delta) \, \ud \delta,\quad \mbox{for all}~\tau > 0.
\end{equation}
Combining the identities~\eqref{eqn:fubini-identity} with \Cref{lem:truncated-integral-bounds} gives for each $\tau > 0$, 
\[
\cS(\tau) = \half \int_{\tau}^\infty \frac{\eps^2(\nu)}{\nu^2} \, \ud \nu = 
\int_0^{\diam(T)} 
\min\Big\{\frac{\delta }{\tau}, \sqrt{h(\delta)}\Big\} \, \ud \delta 
= 
 \cJ_{\eps(\tau)}^{\rm loc}(T) + 
 \frac{\eps^2(\tau)}{2\tau}.
\]
By~\Cref{pr:metric-characterization-of-stat-rate}, 
there is a constant $c > 0$ such that $c^{-1} \, \eps(\sigma) \leq \eps_\star(\sigma) \leq c \eps(\sigma)$ for each $\sigma > 0$.

\paragraph{Proof of \Cref{lem:truncated-integral-bounds}:}

If $h(\delta) \neq 0$ and $f(\delta) > \nu$, then 
$\eps(\nu) \leq \delta$. Therefore, taking the contrapositive, we have
\[
\{\nu \mid \eps(\nu) > \delta\}
\subset 
\{\nu \mid  \nu > \delta/\sqrt{h(\delta)}\}.
\]
On the other hand, if $h(\delta) = 0$, then since $h$ is nonincreasing, we have $\eps(\nu) \leq \delta$ for all $\nu > 0$, and thus 
\[
\{\nu \mid  \nu > \delta/\sqrt{h(\delta)}\}
= \emptyset.
\]
Therefore
\begin{subequations}
\begin{equation}
\label{ineq:trunc-integral-strict}
I^{>}_\tau(\delta) 
\leq \ind{h(\delta) \neq 0} 
\int_{\nu > \max\{\tau, \delta/\sqrt{h(\delta)}\}}
\frac{1}{\nu^2} \, \ud \nu = 
\min\Big\{\frac{1}{\tau}, \frac{\sqrt{h(\delta)}}{\delta}\Big\}.
\end{equation}
Similarly, if $h(\delta) \neq 0$, and $f(\delta) \leq \nu$, then $\eps(\nu) \geq \delta$. 
Therefore, 
\[
\{\nu \mid \eps(\nu) \geq \delta\}
\supset 
\{\nu \mid  \nu \geq \delta/\sqrt{h(\delta)}\}.
\]
This implies 
\begin{equation}
\label{ineq:trunc-integral-geq}
I^{\geq}_\tau(\delta) 
\geq \ind{h(\delta) \neq 0} 
\int_{\nu \geq \max\{\tau, \delta/\sqrt{h(\delta)}\}}
\frac{1}{\nu^2} \,\ud \nu= 
\min\Big\{\frac{1}{\tau}, \frac{\sqrt{h(\delta)}}{\delta}\Big\}.
\end{equation}
\end{subequations}
Combining inequalities~\eqref{ineq:trunc-integral-strict} and~\eqref{ineq:trunc-integral-geq} yields the claim. \qed

\subsubsection{Proof of \Cref{thm:localpb}}
\label{sec:proofoflocalpb}
Let $\beta \defn 1/8$. First, we verify the identity
\[
\mathcal{KL}(\rho\,\|\,\mu_{\operatorname{loc}})
=\mathcal{KL}(\rho\,\|\,\mu)
+\beta \E_{x\sim \rho}\bigl(-2\langle g,x\rangle+\|x\|^2\bigr)
+\log\E_{x\sim \mu}\exp\bigl(-\beta(-2\langle g,x\rangle+\|x\|^2)\bigr).
\]
Denote
\[
\Delta' \defn \log\E_{x\sim \mu}\exp\bigl(-\beta(-2\langle g,x\rangle+\|x\|^2)\bigr).
\]
Note that $\Delta'$ does not depend on $\rho$. Therefore,
\begin{align}
&\E\left[\sup_{\rho \ll \mu}\left(\E_{x\sim \rho}\left[\langle x,g\rangle-\|x\|^2/2\right]-2\mathcal{KL}(\rho\,\|\,\mu_{\operatorname{loc}})\right)\right] \\
&\quad = \E\left[\sup_{\rho \ll \mu}\left(\E_{x\sim \rho}\left[\langle x,g\rangle-\|x\|^2/2\right]-2\mathcal{KL}(\rho\,\|\,\mu)
- 2\beta\E_{x\sim \rho}\left[-2\langle g, x \rangle + \|x\|^2\right]\right)\right] - 2\E\Delta' \\
&\quad = 2\E\left[\sup_{\rho \ll \mu}\left(\frac{1 + 4\beta}{2}\E_{x\sim \rho}\left[\langle x,g\rangle-\|x\|^2/2\right]-\mathcal{KL}(\rho\,\|\,\mu) \right)\right] - 2\E\Delta' \\
&\quad =
2\E\left[\log\left(\E_{x\sim \mu}\exp\Bigl(\tfrac{1 + 4\beta}{2}\bigl(\langle g,x\rangle-\|x\|^2/2\bigr)\Bigr)\right)\right] - 2\E\Delta',
\end{align}
where in the last line we used the Donsker-Varadhan variational formula.
Substituting $\beta=1/8$ and the definition of $\Delta'$ gives
\begin{equation}\label{eq:localpb-ratio}
\E\Bigl[\sup_{\rho\ll \mu}\Bigl(\E_{x'\sim \rho}\bigl[\langle x',g\rangle-\|x'\|^2/2\bigr]-2\,\mathcal{KL}(\rho\,\|\,\mu_{\operatorname{loc}})\Bigr)\Bigr]
=
2\E\log\left(\frac{\E_{x\sim \mu}\exp\left(\tfrac{3}{4} \, \langle x, g\rangle - \tfrac{3}{8}\|x\|^2\right)}{\E_{x\sim \mu}\exp\left(\tfrac{1}{4} \, \langle x, g\rangle - \tfrac{1}{8}\|x\|^2\right)}\right).
\end{equation}

By leveraging Gaussianity, we now show that the right-hand side is less than $0$. 
Recall the Ornstein-Uhlenbeck semigroup consists of the operators $\{P_t\}_{t \geq 0}$, defined for integrable functions $f$ by
\[
(P_t f)(u)\defn \E_{Z}\, f\bigl(e^{-t}u+\sqrt{1-e^{-2t}}\,Z\bigr),\; \textrm{where}\;
Z\sim\Normal{0}{I_d}.
\]
We define
\[
F(u)\defn \E_{x\sim\mu}\exp\Bigl(\tfrac34\langle x,u\rangle-\tfrac38\|x\|^2\Bigr),
\qquad
G(u)\defn \E_{x\sim\mu}\exp\Bigl(\tfrac14\langle x,u\rangle-\tfrac18\|x\|^2\Bigr).
\]
Using the explicit form of the  Gaussian moment generating function, it is straightforward to check that $G = P_{\log 3} F$. 
The invariance of the Gaussian under the Ornstein-Uhlenbeck semigroup yields
\[
\E\log G(g)
=\E\log\bigl(P_{\log(3)}F(g)\bigr)
\ge \E\bigl[P_{\log(3)}(\log F)(g)\bigr]
= \E\log F(g),
\]
and thus~\eqref{eq:localpb-ratio} is nonpositive, as needed.
Above, the inequality arose due to the concavity of the logarithm. \qed

\subsubsection{Proof of \Cref{cor:ellipsoid}}

Fix $\lambda > 0$, and choose
$
\mu = \Normal{0}{\sum_{i=1}^d \frac{a_i^{2}}{\lambda}\, e_i e_i^\top }.
$
Set
$
\alpha_i = \frac{a_i^2}{a_i^2+\lambda},
$
for $i \in [d]$.
Conditioning on $g$, choose
$
\widetilde{\rho}
=
\Normal{
\sum_{i=1}^d \alpha_i g_i e_i}{
\sum_{i=1}^d \alpha_i\, e_i e_i^\top}
$. Under $\widetilde{\rho}$, we have
\begin{align}
\E_{x' \sim \widetilde{\rho}}
\left[\langle x', g\rangle - \frac{\|x'\|^2}{2}\right]
&=
\sum_{i=1}^d
\left(
\alpha_i g_i^2
-
\frac{1}{2}\E_{x' \sim \widetilde{\rho}}[(x_i')^2]
\right)
\\
&=
\sum_{i=1}^d
\left(
\alpha_i g_i^2
-
\frac{1}{2}\big(\alpha_i + \alpha_i^2 g_i^2\big)
\right).
\label{eq:ellipsoid-posterior-energy}
\end{align}
Next, using the standard formula for the relative entropy between Gaussians, we obtain
\begin{align}
\mathcal{KL}(\widetilde{\rho}\,\|\,\mu)
&=
\frac{1}{2}\sum_{i=1}^d
\left[
\frac{\lambda}{a_i^{2}+\lambda}
+ \frac{\lambda a_i^{2}}{(a_i^{2}+\lambda)^{2}}\, g_i^{2}
- 1
+ \log\left(1+\frac{a_i^{2}}{\lambda}\right)
\right]
\\
&=
\frac{1}{2}\sum_{i=1}^d
\left[
-\alpha_i + \alpha_i(1-\alpha_i)g_i^2
+ \log\left(1+\frac{a_i^{2}}{\lambda}\right)
\right].
\label{eq:ellipsoid-kl}
\end{align}
Subtracting \eqref{eq:ellipsoid-kl} from \eqref{eq:ellipsoid-posterior-energy}, we find
\begin{equation}
\label{eq:ellipsoid-gibbs-value}
\E_{x' \sim \widetilde{\rho}}
\left[\langle x', g\rangle - \frac{\|x'\|^2}{2}\right]
-
\mathcal{KL}(\widetilde{\rho}\,\|\,\mu)
=
\frac{1}{2}\sum_{i=1}^d \alpha_i g_i^2
-
\frac{1}{2}\sum_{i=1}^d
\log\left(1+\frac{a_i^{2}}{\lambda}\right).
\end{equation}

We now compare the supremum over $\cE_a$ with the corresponding penalized quadratic problem on $\R^d$. Since $x \in \cE_a$ implies $\sum_{i=1}^d \frac{x_i^2}{a_i^2} \le 1$, we have
\begin{align}
\sup_{x \in \cE_a}\left(\langle x, g\rangle - \frac{\|x\|^2}{2} - \frac{\lambda}{2}\right)
&\le
\sup_{x \in \cE_a}\left(\langle x, g\rangle - \frac{\|x\|^2}{2} - \frac{\lambda}{2}\sum_{i=1}^d\frac{x_i^2}{a_i^2}\right)
\\
&\le
\sup_{x \in \R^d}\left(\langle x, g\rangle - \frac{\|x\|^2}{2} - \frac{\lambda}{2}\sum_{i=1}^d\frac{x_i^2}{a_i^2}\right)
\\
&=
\sum_{i=1}^d
\sup_{x_i \in \R}
\left\{
x_i g_i - \frac{1}{2}\left(1+\frac{\lambda}{a_i^2}\right)x_i^2
\right\}
\\
&=
\frac{1}{2}\sum_{i=1}^d \alpha_i g_i^2.
\label{eq:ellipsoid-penalized-sup}
\end{align}
Combining \eqref{eq:ellipsoid-gibbs-value} and \eqref{eq:ellipsoid-penalized-sup}, we obtain the pointwise inequality
\[
\sup_{x \in \cE_a}\left(\langle x, g\rangle - \frac{\|x\|^2}{2}\right)
\le
\frac{\lambda}{2}
+
\frac{1}{2}\sum_{i=1}^d \log\left(1+\frac{a_i^2}{\lambda}\right)
+
\E_{x' \sim \widetilde{\rho}}
\left[\langle x', g\rangle - \frac{\|x'\|^2}{2}\right]
-
\mathcal{KL}(\widetilde{\rho}\,\|\,\mu).
\]
Exponentiating both sides and taking expectations with respect to $g$, and using
that, by the proof of \Cref{prop:simplepb}, $\E\exp\left(\E_{x' \sim \widetilde{\rho}}
\left[\langle x', g\rangle - \frac{\|x'\|^2}{2}\right]
-
\mathcal{KL}(\widetilde{\rho}\,\|\,\mu)\right) \le 1$, yields the second inequality in the statement.
The first inequality follows from Jensen's inequality. \qed

\subsubsection{Proof of \Cref{cor:localellipsoid}}

First, a direct computation shows that
\[
\mu_{\operatorname{loc}} = \Normal{\sum_{i=1}^{d}
        \frac{a_i^{2}g_i}{a_i^{2}+4\lambda}e_i}{\sum_{i=1}^{d}
        \frac{4a_i^{2}}{a_i^{2}+4\lambda}\,e_i e_i^{\top}}.
\]
Using that the expectation of $\mu_{\operatorname{loc}}$ is exactly the global maximizer of $\langle x, g\rangle - \frac{\|x\|^2}{2} - 2\lambda\sum\limits_{i = 1}^d\frac{x_i^2}{a_i^2}$, we have
\begin{align}
\E\sup\limits_{x \in \cE_a} \left(\langle x, g\rangle - \frac{\|x\|^2}{2} - 2\lambda\right) &\le \E\sup\limits_{x \in \cE_a}\left(\langle x, g\rangle - \frac{\|x\|^2}{2} - 2\lambda\sum\limits_{i = 1}^d\frac{x_i^2}{a_i^2}\right) 
\\
&\le \E\sup\limits_{x \in \R^d}\left(\langle x, g\rangle - \frac{\|x\|^2}{2} - 2\lambda\sum\limits_{i = 1}^d\frac{x_i^2}{a_i^2}\right)
\\
&\le\E\E_{x^\prime \sim \mu_{\operatorname{loc}}}\left[\langle x^\prime, g\rangle - \frac{\|x^\prime\|^2}{2}\right] + \frac{1}{2}\sum\limits_{i = 1}^d\frac{4a_i^2}{a_i^2 + 4\lambda}
\le  \frac{1}{2}\sum\limits_{i = 1}^d\frac{4a_i^2}{a_i^2 + 4\lambda},
\end{align}
where the last line is due to \Cref{thm:localpb} with $\rho = \mu_{\operatorname{loc}}$. The claim follows. \qed

\section{Example: Gaussian width of the crosspolytope in $\R^d$}
\label{sec:examples}
In this section we focus on a single illustrative example: the Gaussian width of the crosspolytope
\[
B_1^d=\{x\in\R^d:\ \|x\|_1\le 1\}.
\]
Of course, one can bound $w(B_1^d)$ by the elementary identity
\begin{equation}
\label{eq:shortcut}
w(B_1^d)=\E\sup_{x\in B_1^d} \, \langle x,g\rangle
=\E\max_{i\in[d]}|g_i|
\le \sqrt{2\log(2d)}.
\end{equation}
This example is instructive because it highlights the limitations of entropy-based chaining bounds: Dudley's integral incurs its largest possible gap for sets in $\R^d$ and yields an upper bound of order $\log^{3/2}(d)$ for $B_1^d$.
Sharp bounds can be obtained by constructing an optimal admissible sequence (see, e.g., \cite[Section~4.4]{nelson2016chaining}), and related work develops more explicit geometric relaxations of generic chaining for convex bodies \cite{van2018chaining} recovering the correct bound on $w(B_1^d)$.

Below we present three different methods to bound $w(B_1^d)$ for the crosspolytope that avoid both the coordinatewise maximum argument and explicit generic chaining constructions (and their relaxations). The purpose is not to compete with the shortcut \eqref{eq:shortcut}, but to illustrate how our decomposition viewpoint leads to sharp bounds through several complementary mechanisms.

\subsection{A bound based on the variational approach}
\label{sec:pacbayesellone}
 Since our goal is to avoid the explicit admissible sequence constructions and their relaxations, we first use \Cref{prop:simplepb} with a specifically chosen prior, almost identical to one introduced recently in the information theory literature \cite{miyaguchi2019adaptive} in the context of proving asymptotic regret bounds with respect to logarithmic loss, namely
\begin{equation}
\label{eq:sparseprior}
\mu(\ud x) \propto \exp(-\lambda\|x\|_1)\bigotimes_{i = 1}^d\left(\delta_{0}(\ud x_i) + \frac{\exp(\lambda^2/2)}{\lambda^2}\ind{|x_i| \ge \lambda}\ud x_i\right),
\end{equation}
where $\delta_0$ is the Dirac measure at zero. As will be clear from our computations below, the logic of this prior lies in mimicking the typical solution of the $\ell_1$-regularized minimization problem: the Dirac mass is responsible for possible sparse solutions, namely a coordinate being equal to zero. A straightforward computation shows that the normalization constant $Z_0$ for this distribution is equal to
\[
Z_0 = \left(1+\frac{2}{\lambda^3}\exp(-\lambda^{2}/2)\right)^{d}.
\]
We are ready to state the result of this section.

\begin{corollary}
\label{cor:elloneball}
Let $d\ge 2$. Then, if in \Cref{prop:simplepb} we choose the prior $\mu$ according to \eqref{eq:sparseprior} with $\lambda = \sqrt{2\log(d)}$, we obtain
\[
\E\left[\sup\limits_{x \in B_1^d}\left(\langle x, g\rangle -\|x\|^2/2\right)\right]
\le \sqrt{2\log(d)} + \frac{3\log(2\log(d))}{2\sqrt{\pi}\sqrt{\log(d)}} + \frac{c}{{\log^{3/2}(d)}},
\]
where $c = \frac{1}{\sqrt{2}} + \frac{1}{2\sqrt{\pi}}$.
\end{corollary}

Note that the leading term $\sqrt{2\log(d)}$ is sharp for the Gaussian width of $B_1^d$. Moreover, our \Cref{thm:decomp-of-gauss-width} with $\sigma = 1$ immediately implies the correct bound on $w(B_1^d)$ with the correct leading term $\sqrt{2\log(d)}$. Indeed, in the second term of the decomposition we may use that for the crosspolytope $r(\sigma) \le 1$ and therefore, for the same $c = \frac{1}{\sqrt{2}} + \frac{1}{2\sqrt{\pi}}$,
\[
w(T) \le \sqrt{2\log(d)} + \frac{3\log(2\log(d))}{2\sqrt{\pi}\sqrt{\log(d)}} + \frac{1}{2} + \frac{c}{{\log^{3/2}(d)}}.
\]

\begin{proof}
First, by \Cref{prop:simplepb} and the Donsker--Varadhan variational formula we have
\begin{equation}
\label{eq:dualitywritten}
0 \ge \E\sup\limits_{\rho \ll \mu}\left(\E_{x^\prime \sim \rho}\left[\langle x^\prime, g\rangle - \|x^\prime\|^2/2\right] - \mathcal{KL}(\rho\|\mu)\right)
= \E\log\left(\E_{x^\prime \sim \mu}\left[\exp(\langle x^\prime, g\rangle - \|x^\prime\|^2/2)\right]\right).
\end{equation}
We can write $\log\left(\E_{x^\prime \sim \mu}\left[\exp(\langle x^\prime, g\rangle - \|x^\prime\|^2/2)\right]\right)$ as
\begin{align}
&-\log(Z_0) + \log\left(\int_{\R^d}\exp\left(\langle x, g\rangle - \frac{\|x\|^2}{2} - \lambda\|x\|_1\right)\bigotimes_{i = 1}^d\left(\delta_{0}(\ud x_i) + \frac{\exp(\lambda^2/2)}{\lambda^2}\ind{|x_i| \ge \lambda}\,\ud x_i\right)\right)
\\
&= -\log(Z_0) + \sum\limits_{i = 1}^d\log\left(\int_{\R}\exp\left(x g_i - x^2/2 - \lambda|x|\right)\left(\delta_{0}(\ud x) + \frac{\exp(\lambda^2/2)}{\lambda^2}\ind{|x| \ge \lambda}\,\ud x\right)\right).
\end{align}
Observe that
\[
\sup\limits_{x \in \R}(x g_i - x^2/2 - \lambda|x|)
= \begin{cases}
  0, & |g_i| \le \lambda,\\[6pt]
  \dfrac12\bigl(|g_i|-\lambda\bigr)^{2}, & |g_i| > \lambda.
\end{cases}
\]
First, consider the case where $|g_i| \le \lambda$. In this case, using the monotonicity of the logarithm, we have for the individual summand
\begin{align}
&\log\left(\int_{\R}\exp\left(x g_i - x^2/2 - \lambda|x|\right)\left(\delta_{0}(\ud x) + \frac{\exp(\lambda^2/2)}{\lambda^2}\ind{|x| \ge \lambda}\,\ud x\right)\right)
\\
&\ge \log\left(\int_{\R}\exp\left(x g_i - x^2/2 - \lambda|x|\right)\delta_{0}(\ud x)\right)
= \sup\limits_{x \in \R}(x g_i - x^2/2 - \lambda|x|)
= 0.
\end{align}
Otherwise, if $|g_i| \ge \lambda$, by the exact computation of the integral we have
\begin{align}
&\log\left(\int_{\R}\exp\left(x g_i - x^2/2 - \lambda|x|\right)\left(\delta_{0}(\ud x) + \frac{\exp(\lambda^2/2)}{\lambda^2}\ind{|x| \ge \lambda}\,\ud x\right)\right)
\\
&\ge \log\left(\int_{\R}\exp\left(x g_i - x^2/2 - \lambda|x|\right)\frac{\exp(\lambda^2/2)}{\lambda^2}\ind{|x| \ge \lambda}\,\ud x\right)
\\
&= \log\left(\frac{\exp(\lambda^2/2)}{\lambda^2}\left(\int_{-\infty}^{-\lambda}\exp\left(x g_i - x^2/2 + \lambda x\right)\ud x + \int_{\lambda}^{+\infty}\exp\left(x g_i - x^2/2 - \lambda x\right)\ud x\right)\right)
\\
&=\log\left(\frac{\exp(\lambda^2/2)}{\lambda^2}\left(\sqrt{2\pi}\,\Bigl[
   \exp(\tfrac{(g_i+\lambda)^{2}}{2})\,
   \Phi(-g_i-2\lambda)
   +
   \exp(\tfrac{(g_i-\lambda)^{2}}{2})\,
   \bigl(1-\Phi(2\lambda-g_i)\bigr)
\Bigr]\right)\right)
\\
&\ge \frac{1}{2}\bigl(|g_i|-\lambda\bigr)^{2} + \frac{\lambda^{2}}{2}
                  -2\log\lambda
                  +\frac12\log(2\pi)
                  +\log(1-\Phi(\lambda))
\\
&\ge \sup\limits_{x \in \R}(x g_i - x^2/2 - \lambda|x|) - \frac{1}{\lambda^2} - 3\log(\lambda),
\end{align}
where in the last line we used $\log(1 - \Phi(\lambda)) \ge -\frac{\lambda^2}{2} - \log(\lambda) - \frac{1}{2}\log(2\pi) - \frac{1}{\lambda^2}$. Therefore, combining the above inequalities we have
\begin{align}
&-\log(Z_0) + \sum\limits_{i = 1}^d\log\left(\int_{\R}\exp\left(x g_i - x^2/2 - \lambda|x|\right)\left(\delta_{0}(\ud x) + \frac{\exp(\lambda^2/2)}{\lambda^2}\ind{|x| \ge \lambda}\,\ud x\right)\right) \nonumber
\\
&\ge-\log(Z_0) + \sup\limits_{x \in \R^d}\left(\langle x, g\rangle - \frac{\|x\|^2}{2} - \lambda\|x\|_1\right) - \sum\limits_{i = 1}^d\ind{|g_i| \ge \lambda}\left(3\log(\lambda) + \frac{1}{\lambda^2}\right).
\label{eq:zoequation}
\end{align}
Finally, we have
\begin{align}
\sup\limits_{x \in B_1^d} \left(\langle x, g\rangle - \frac{\|x\|^2}{2} - \lambda\right)
&\le \sup\limits_{x \in B_1^d}\left(\langle x, g\rangle - \frac{\|x\|^2}{2} - \lambda\|x\|_1\right)
\\
&\le \sup\limits_{x \in \R^d}\left(\langle x, g\rangle - \frac{\|x\|^2}{2} - \lambda\|x\|_1\right).
\end{align}
Combining this with \eqref{eq:dualitywritten} and \eqref{eq:zoequation} we obtain
\[
\E\sup\limits_{x \in B_1^d} \left(\langle x, g\rangle - \frac{\|x\|^2}{2}\right)
\le \lambda + \log(Z_0) + d \, \Pr\{|g_1| \ge\lambda\} \, \Bigl(\frac{1}{\lambda^2} + 3\log(\lambda)\Bigr).
\]
We choose $\lambda = \sqrt{2\log(d)}$ (note that $\lambda \ge 1$ for $d \ge 2$) and obtain
\begin{align}
\E\sup\limits_{x \in B_1^d} \left(\langle x, g\rangle - \frac{\|x\|^2}{2}\right)
\le \sqrt{2\log(d)} + \frac{1}{\sqrt{2}\;{\log^{3/2}(d)}} + \frac{1}{2\sqrt{\pi}\log^{3/2}(d)} + \frac{3\log(2\log(d))}{2\sqrt{\pi}\sqrt{\log(d)}}.
\end{align}
The claim follows.
\end{proof}

\subsection{A bound based on the profile of intrinsic volumes}
Recall from \Cref{thm:indexstarbound} that, since $\diam(B_1^d)=2$, we have
\[
w(B_1^d) \asymp i^\star\,\diam(B_1^d),
\]
where
\[
i^\star \in \argmax_{i \in \{1,\ldots,d\}}\left\{ V_i\left(B_1^d/2\right)\right\}.
\]
Thus, to upper bound $w(B_1^d)$ it suffices to control the location of $i^\star$. By unimodality of the intrinsic volume sequence, this reduces to understanding the ratios
\[
\frac{V_{i}(B_1^d/2)}{V_{i+1}(B_1^d/2)} = 2\,\frac{V_{i}(B_1^d)}{V_{i+1}(B_1^d)}.
\]
The next two lemmas provide lower bounds on these ratios that are sharp enough to yield the optimal order for $w(B_1^d)$.
We remark that our argument uses explicit formulas for intrinsic volume ratios of $B_1^d$ for $i\ge 1$. While this is a strong input, once the zeroth intrinsic volume is excluded it is still scale-free, and it is not immediate from the ratios alone how to recover the Gaussian width without the structural connection provided by \Cref{thm:indexstarbound}.

\begin{lemma}
\label{lem:volumeratio}
For $1 \le i \le d-2$ set $m = d - i - 1$ and denote
\[
Z_0=\int_{0}^{\infty} \exp(-(i+1)x^2)\big(\operatorname{erf}(x)\big)^{m}\ud x,
\qquad
\ud\nu(x)=\frac{1}{Z_0}\exp(-(i+1)x^2)\big(\operatorname{erf}(x)\big)^{m}\ud x,
\]
where $\operatorname{erf}(x) = \frac{2}{\sqrt{\pi}}\int_{0}^x\exp(-y^2)\ud y$.
Let $Y$ be a non-negative random variable distributed according to $\nu$. Then
\[
\frac{V_{i}(B_1^d)}{V_{i+1}(B_1^d)} = \frac{i+1}{2\sqrt{\pi}\,\E[Y]}.
\]
\end{lemma}

\begin{proof}
First, the identity in \cite[Theorem~2.1]{betke1993intrinsic} gives for $1 \le i \le d - 2$,
\[
\frac{V_{i}(B_1^d)}{V_{i+1}(B_1^d)}  = \frac{(i+1)^2}{2m}\frac{\int_{0}^\infty\exp(-(i+1)x^2)(\operatorname{erf}(x))^{m}\ud x}{\int_{0}^\infty\exp(-(i+2)x^2)(\operatorname{erf}(x))^{m - 1}\ud x}.
\]
This is equivalent to
\begin{equation}
\frac{V_i(B_1^d)}{V_{i+1}(B_1^d)}
=\frac{(i+1)^2}{2m}
\frac{1}{\E\big[\exp(-Y^2)/\operatorname{erf}(Y)\big]}.
\label{eq:ratio-mid}
\end{equation}
Set
$
g(x)=\exp(-(i+1)x^2)\big(\operatorname{erf}(x)\big)^{m}.
$
Since $m\ge1$, we have $g(0)=0$ and $g(x)\to0$ as $x\to\infty$. Observe that
\[
g'(x)
=-2(i+1)x\exp(-(i+1)x^2)\big(\operatorname{erf}(x)\big)^{m}
+\frac{2m}{\sqrt{\pi}}\exp(-(i+2)x^2)\big(\operatorname{erf}(x)\big)^{m-1}.
\]
We divide the identity $\int_{0}^{\infty}g'(x)\ud x=0$ by $Z_0$ and obtain
\[
-2(i+1)\,\E[Y]
+\frac{2m}{\sqrt{\pi}}\,
\E\left[\frac{\exp(-Y^2)}{\operatorname{erf}(Y)}\right]=0.
\]
Plugging this into \eqref{eq:ratio-mid} completes the proof.
\end{proof}

To complete the derivation of the upper bound on $w(B_1^d)$, note that if for some $j$ we have
$\frac{j+1}{\sqrt{\pi}\,\E[Y]} > 1$, then by unimodality we may use $j$ as an upper bound on the maximizer $i^\star$ above.
Thus it suffices to lower bound $\frac{j+1}{\sqrt{\pi}\,\E[Y]}$, which reduces to upper bounding $\E[Y]$.
The following lemma provides the required estimate.

\begin{lemma}
\label{lem:ybound}
For the random variable $Y$ defined in \Cref{lem:volumeratio} we have
\[
\E Y \le \sqrt{\log\left(\frac{d}{i+1}\right)} + \frac{1}{\sqrt{\pi(i+1)}}.
\]
\end{lemma}

\begin{proof}
As above, let $g(x)=\exp(-(i+1)x^2)\big(\operatorname{erf}(x)\big)^{m}$, so that $\ud\nu(x) = \frac{g(x)}{Z_0}\ud x$.
Let $t$ be a maximizer of $g$. Using $g^\prime(t) = 0$, we have
\begin{equation}\label{eq:mode-eq}
-2(i+1)t+\frac{2m}{\sqrt{\pi}}\frac{\exp(-t^2)}{\operatorname{erf}(t)}=0,
\quad\textrm{which implies}\quad
\exp(t^2)\,\operatorname{erf}(t)=\frac{m}{\sqrt{\pi}(i+1)\,t}.
\end{equation}
Using the bound $1-\operatorname{erf}(x)\le \dfrac{\exp(-x^2)}{\sqrt{\pi}\,x}$ for $x>0$,
we have
$
\exp(t^2)\,\operatorname{erf}(t)\ge\exp(t^2)-\frac{1}{\sqrt{\pi}\,t}.
$
Combining this with \eqref{eq:mode-eq}, we obtain
$
t\exp(t^2) \le \frac{1}{\sqrt{\pi}}\Big(1+\frac{m}{i+1}\Big).
$
If $t\ge 1$, then $\exp(t^2)\le t\exp(t^2)$ and hence $t^2\le \log\bigl(1+\frac{m}{i+1}\bigr)$.
If $t<1$ and $\log\bigl(1+\frac{m}{i+1}\bigr)<1$, then $\frac{m}{i+1}\le e-1$ and \eqref{eq:mode-eq} together with
$\operatorname{erf}(t)\ge \frac{2}{\sqrt{\pi}}t\,\exp(-t^2)$ yields $t^2\le \frac{m}{2(i+1)}\le \log\bigl(1+\frac{m}{i+1}\bigr)$.
Thus,
\begin{equation}
\label{eq:boundont}
t\le\sqrt{\log\left(1+\frac{m}{i+1}\right)}.
\end{equation}

Given this bound on $t$ it remains to bound $\E(Y - t)_{+}$, where $(x)_+ = \max\{x, 0\}$. Let $h(x)=\log g(x)=-(i+1)x^2 + m\log(\operatorname{erf}x)$. Since for any $x > 0$,
\[
h''(x)= -2(i+1) + m\big(\log(\operatorname{erf}x)\big)'' \le -2(i+1),
\]
and $h'(t)=0$, we have for all $x\ge t$,
\begin{equation}
\label{eq:ratioofintegrals}
h(x)\le h(t) - (i+1)(x-t)^2,\quad \textrm{which implies} \quad
\frac{g(x)}{g(t)}\le \exp\big(-(i+1)(x-t)^2\big).
\end{equation}
Finally, using this inequality we have
\[
\E\big[(Y-t)_+\big]
=\frac{\int_t^\infty (x-t)\,g(x)\ud x}{\int_0^\infty g(x)\ud x}
\le
\frac{\int_t^\infty (x-t)\,g(x)\ud x}{\int_t^\infty g(x)\ud x}
=\frac{\int_0^\infty u\,g(t+u)\ud u}{\int_0^\infty g(t+u)\ud u}.
\]
We now bound the last ratio. We introduce the random variable $U$, with density proportional to $\exp(-(i+1)u^2)$ on $[0,\infty)$, so that
\[
\frac{\int_0^\infty u\,g(t+u)\ud u}{\int_0^\infty g(t+u)\ud u}
= \frac{\E\big[U\,r(U)\big]}{\E\big[r(U)\big]},
\]
with $r(u) = \frac{g(t + u)}{\exp(-(i + 1)u^2)}$. Note that the function $r$ is non-increasing. Indeed, recalling that
$h'(x)=-2(i+1)x+\frac{2m}{\sqrt{\pi}}\frac{\exp(-x^2)}{\operatorname{erf}(x)}$ and that $t$ maximizes $g$, we have $h'(t)=0$, that is,
\begin{equation}\label{eq:mode}
\frac{2m}{\sqrt{\pi}}\frac{\exp(-t^2)}{\operatorname{erf}(t)}=2(i+1)t.
\end{equation}
Define $\phi(x)=\frac{2}{\sqrt{\pi}}\frac{\exp(-x^2)}{\operatorname{erf}(x)}$ for $x>0$.
A direct derivative calculation shows that $\phi$ is strictly decreasing. Now,
$\log(r(u))=h(t+u)+(i+1)u^2$, and hence
\[
\frac{\ud}{\ud u}\log(r(u))
=h'(t+u)+2(i+1)u
=-2(i+1)t+m\,\phi(t+u).
\]
Because $\phi$ is decreasing and $t+u\ge t$, we have $\phi(t+u)\le \phi(t)$, and thus, by \eqref{eq:mode},
\[
\frac{\ud}{\ud u}\log(r(u))\le -2(i+1)t+m\,\phi(t)=0,
\]
so $r$ is non-increasing. Therefore, by Chebyshev's association inequality \cite[Theorem~2.14]{boucheron2013concentration} we have
\[
\frac{\E\big[U\,r(U)\big]}{\E\big[r(U)\big]} \le \E U
=\frac{\int_0^\infty u\exp({-(i+1)u^2})\ud u}{\int_0^\infty \exp({-(i+1)u^2})\ud u}
=\frac{1}{\sqrt{\pi(i+1)}}.
\]
This combined with \eqref{eq:boundont} proves the claim.
\end{proof}

We are ready to provide the desired upper bound on $w(B_{1}^d)$. By \Cref{lem:volumeratio} and \Cref{lem:ybound}, for any $1 \le i \le d -2$,
\[
\frac{V_i(B_1^d/2)}{V_{i+1}(B_1^d/2)}
=\frac{i + 1}{\sqrt{\pi}\E[Y]}
\ge \frac{i + 1}{\sqrt{\pi}\left(\sqrt{\log\left(\frac{d}{i+1}\right)} + \frac{1}{\sqrt{\pi(i+1)}}\right)}.
\]
Since $\log\left(\frac{d}{i+1}\right)\le \log d$, the right-hand side is at least one as soon as $i \gtrsim \sqrt{\log d}$.
By unimodality this implies $i^\star \lesssim \sqrt{\log d}$, and therefore $w(B_1^d)\lesssim \sqrt{\log d}$, as desired.

\subsection{A bound based on the analysis of metric projections}

In this section, we demonstrate how to bound the Gaussian width using
\Cref{thm:decomp-via-projections} and bounds on the metric projection 
onto $B^d_1$. 
Let us introduce the function 
\[
R(\sigma, d) 
\defn 
\begin{cases}
    \sigma^2 d & \mbox{if}~\sigma \leq \tfrac{1}{d}, \\ 
    \sigma [\log(\e d \sigma)]^{-1/2} & \mbox{if}~\sigma \in (\tfrac{1}{d}, \sqrt{\log(\e d)}), \\ 
    1 & \mbox{if}~\sigma \geq \sqrt{\log(\e d)}.
\end{cases}
\]

\begin{proposition} 
\label{prop:l1-metric-projection-upper}
The projection onto the $\ell_1$ ball satisfies the following bounds
\[
\E_{\NoiseVec \sim \Normal{0}{I_\dimension}}\Big[ \|\Pi_{B^d_1}(\sigma \NoiseVec)\|_2^2\Big] 
\lesssim R(\sigma, d), 
\]
for every $\sigma > 0$ and any $d \geq 1$.
\end{proposition} 
\noindent The proof of this result is presented in \Cref{app:proof-of-l1-metric-projection-upper}.
 Indeed, a short calculation using $\rad(B^d_1) = 1$, \Cref{thm:decomp-via-projections} and \Cref{prop:l1-metric-projection-upper} yields that
\begin{equation}
w(B^d_1) \lesssim 1 + \int_0^\infty \frac{R(\sigma, d)}{\sigma^2} \, \ud \sigma  \lesssim \sqrt{\log(\e d)},
\label{eqn:upper-on-the-width-from-projections}
\end{equation}
which is the correct order for the Gaussian width of the crosspolytope.

\subsubsection{Proof of \Cref{prop:l1-metric-projection-upper}}
\label{app:proof-of-l1-metric-projection-upper}

To prove the claim, we need different arguments for different ``ranges'' of the noise level. 
The most technical aspect is the following result for ``moderate'' noise levels. 
\begin{lemma} [Bound for moderate $\sigma$]
\label{lem:medium-sigma-bound}
If $d \geq 1$ and $\sigma \geq 1/d$, then 
\[
\E \|\BallProj(\sigma \NoiseVec)\|_2^2 \leq 
 205584 \frac{\sigma}{\sqrt{\log(\e d \sigma)}}.
\]
\end{lemma} 
\noindent The proof of \Cref{lem:medium-sigma-bound} is given in the next section. \Cref{lem:medium-sigma-bound}  completes the proof of the claim: for $\sigma \leq 1/d$, the nonexpansiveness of projections yields
\[
\E \|\BallProj(\sigma \NoiseVec)\|_2^2 \leq \sigma^2 \E \|\NoiseVec\|_2^2 = \sigma^2 d.
\]
For $\sigma \geq \sqrt{\log(\e d)}$, the claim follows from the inclusion $B^d_1 \subset B^d_2$. 

\subsubsection{Proof of \Cref{lem:medium-sigma-bound}}

Let $Z \sim \Normal{0}{1}$. We use the notation 
\[\Psi(z) \defn \P\{Z \geq z\}, \quad 
\Phi(z) \defn \P\{Z \leq z\}, \quad \mbox{and} \quad 
\phi(z) = \frac{1}{\sqrt{2\pi}} \e^{-z^2/2}.
\]
Of course, $\Psi(z) = 1 - \Phi(z)$ and $\Phi(-z) = \Psi(z)$. We also repeatedly use bounds on Mills' ratio, 
\[
M(z) \defn \frac{\Psi(z)}{\phi(z)}.
\]
We use the shorthand $\alpha \defn \sigma d$. 
We define the functions
\[
S_1(\lambda) \defn 
\E [(|Z| - \lambda)_+] 
\quad \mbox{and,} \quad 
S_2(\lambda) \defn 
\E [(|Z| - \lambda)_+^2].
\]
We also require their random counterparts, 
\[
\hat S_1(\lambda) \defn 
\frac{1}{d} \sum_{i=1}^d (|\NoiseVec_i| - \lambda)_+ 
\quad \mbox{and,} \quad 
\hat S_2(\lambda) \defn 
\frac{1}{d} \sum_{i=1}^d 
(|\NoiseVec_i| - \lambda)_+^2. 
\]
We define the two fixed points, 
\begin{equation}\label{eqn:fixed-points}
\lambdastar 
\defn 
\inf\big\{\lambda \geq 0 : S_1(\lambda) \leq \tfrac{1}{\alpha}\,\big\}, 
\quad \mbox{and,} \quad 
\hat\lambda 
\defn 
\inf\big\{\lambda \geq 0 : \hat S_1(\lambda) \leq \tfrac{1}{\alpha}\,\big\}.
\end{equation}

\subsubsection{Auxiliary results}

The following claim follows easily from convex duality; we omit the proof, but note that it uses the same argument as in Lemma 1 in~\cite{AolJorPatUli25}.  
\begin{lemma} [Characterization of the projection onto the $\ell_1$ ball]
\label{lem:characterization-of-projection}
For any $d \geq 1$, we have for any $\NoiseVec \in \R^\dimension, \sigma > 0$ that
\[
\Big[\BallProj(\sigma \NoiseVec)\Big]_i = 
\sigma \sign(\NoiseVec_i) \big(|\NoiseVec_i| - \hat \lambda\big)_+
\]
where $\hat\lambda$ is defined as in display~\eqref{eqn:fixed-points}.
\end{lemma} 

\begin{lemma} [Bounds on the Mills' ratio]
\label{lem:Mills-ratio-bounds}
The Mills' ratio 
satisfies the following bounds: 
\[
\frac{x^2}{1+x^2} \leq \frac{x^2 + \tfrac{1}{5x^2}}{1+x^2} \leq x M(x) \leq \frac{x^2 + 2}{x^2 + 3} \leq 1,
\]
where the upper bounds hold for $x \geq 0$, and the lower bounds hold for $x \geq 1$.
\end{lemma} 
\begin{proof}
By Lemma 5 in~\cite{GasUtz14}, we have, in their notation, for all $x \geq 0$,
\[
xM(x) \leq \frac{x Q_3(x)}{P_3(x)} = 
\frac{x(x^2+2)}{x^3 + 3x} = \frac{x^2 + 2}{x^2 + 3}.
\]
Additionally, for $x \geq 1$, we have 
\[
x M(x) \geq \frac{xQ_4(x)}{P_4(x)} = 
\frac{x^4 + 5x^2}{x^4 + 6x^2 + 3} = 
1 - \frac{1}{1 + x^2} \Big(1 - \frac{2x^2}{x^4 + 6x^2 + 3}\Big) \geq 
1 - \frac{1}{1 + x^2} \Big(1 - \frac{1}{5x^2}\Big),
\]
where we used $x \geq 1$ for the final inequality.
\end{proof}

\subsubsection{Proof outline of \Cref{lem:medium-sigma-bound}}

By definition of $\hat \lambda, \lambdastar$ (see eqn.~\eqref{eqn:fixed-points}) we have 
$\|\BallProj(\sigma \NoiseVec)\|_2^2 = \sigma^2 d \hat S_2(\hat\lambda)$.
We would expect that
\begin{equation}\label{eqn:emp-to-pop}
\hat \lambda \approx \lambdastar, 
\quad \mbox{and, hence,} \quad 
\E \|\BallProj(\sigma \NoiseVec)\|_2^2 \approx 
\sigma^2 d \hat S_2(\hat \lambda) 
\approx \sigma^2 d S_2(\lambdastar).
\end{equation}
by definition of $\alpha$ and by \Cref{lem:bounds-on-lambda-star}. 
The first part of the argument, presented below in \Cref{sec:population-quantities}, shows that $S_2(\lambdastar) \asymp [\sigma d \sqrt{\log (\e d \sigma)}]^{-1}$. The second part of the argument, presented below in \Cref{sec:emp-to-pop}, makes the approximations above~\eqref{eqn:emp-to-pop} rigorous and will yield the claim.
We combine the first two parts of the argument and complete the proof in \Cref{sec:complete-the-proof}.

\subsubsection{Bounds on population quantities}
\label{sec:population-quantities}

\begin{lemma} 
\label{lem:first-and-second-moment-decomps}
Let $E(\lambda) \defn 1 - \lambda M(\lambda)$. 
Then, we have 
\[
S_1(\lambda) =
2 \phi(\lambda) E(\lambda)
\quad\mbox{and} \quad 
S_2(\lambda) = 
2 \frac{\phi(\lambda)}{\lambda}
[1-(\lambda^2 + 1)E(\lambda)].
\]
\end{lemma} 
\begin{proof}
First note that 
\[
S_1(\lambda) = 
2 \int_\lambda^\infty (z - \lambda) \phi(z) \, \ud z = 2[\phi(\lambda) - \lambda \Psi(\lambda)]
= 
2\phi(\lambda)[1 - \lambda M(\lambda)] 
= 2 \phi(\lambda) E(\lambda)
\]
Similarly, we have 
\begin{multline}\label{eqn:second-moment}
S_2(\lambda) = 
2 \int_\lambda^\infty (z^2 - 2\lambda z + \lambda^2) \, \phi(z) \, \ud z = 2[\lambda \phi(\lambda) -2\lambda \phi(\lambda) +(\lambda^2 + 1)\Psi(\lambda)] \\= 
2[(\lambda^2 + 1)\Psi(\lambda) - \lambda \phi(\lambda)] = 2 
\frac{\phi(\lambda)}{\lambda}
[(\lambda^2 + 1)\lambda M(\lambda) - \lambda^2]
\end{multline}
where we used $(z^2 - 1) \phi(z) = (-z\phi(z))'$. To conclude, we simply recall that $\lambda M(\lambda) = 1 - E(\lambda)$.
\end{proof}

\begin{lemma} [Bounds on $S_1$ and $S_2$]
\label{lem:rational-bounds-moments}
For $\lambda \geq 1$, we have 
\[
\frac{1}{2} \, 
\frac{\varphi(\lambda)}{\lambda^2} 
\leq S_1(\lambda) \leq 
2 \, \frac{\varphi(\lambda)}{\lambda^2}, 
\quad \mbox{and,} \quad 
\frac{2}{5} \frac{\phi(\lambda)}{\lambda^3} 
\leq S_2(\lambda) \leq 4 
\frac{\phi(\lambda)}{\lambda^3}.
\]
\end{lemma} 
\begin{proof}
    From \Cref{lem:Mills-ratio-bounds}, it follows that, for $\lambda \geq 1$, we have for $E(\lambda) \defn 1 - \lambda M(\lambda)$ that
    \[
    \frac{1}{4} \frac{1}{\lambda^2}
    \leq 
    \frac{1}{3 + \lambda^2} 
    \leq E(\lambda) \leq 
    \frac{1 - 1/(5\lambda^2)}{\lambda^2 + 1}
    \leq 
    \frac{1}{1+\lambda^2} \leq \frac{1}{\lambda^2}.
    \]
    The bounds on $S_1(\lambda)$ now follow immediately from \Cref{lem:first-and-second-moment-decomps}. Using the bounds in \Cref{lem:Mills-ratio-bounds} again, we have
    \[
    S_2(\lambda) = 2 \frac{\phi(\lambda)}{\lambda} [1 - (\lambda^2+1)E(\lambda)] \leq \frac{\phi(\lambda)}{\lambda} 
    \frac{4}{\lambda^2 + 3} \leq 4 \frac{\phi(\lambda)}{\lambda^3}.
    \]
    And, on the other hand, we have 
    \[
    S_2(\lambda) = 2 \frac{\phi(\lambda)}{\lambda} [1 - (\lambda^2+1)E(\lambda)] \geq \frac{2}{5} \frac{\phi(\lambda)}{\lambda^3}. \qedhere 
    \]
\end{proof}

\begin{lemma} \label{lem:bounds-on-lambda-star}
If $\alpha \geq \tfrac{1}{S_1(1)}$, then 
\[
\sqrt{\frac{1}{7} \log(\e \alpha)}
\leq \lambdastar \leq 
\sqrt{2 \log(\e \alpha)}, 
\quad \mbox{and,} \quad 
 \frac{1}{2 \sqrt{2}} \frac{1}{\alpha \sqrt{\log(\e \alpha)}}
\leq S_2(\lambdastar) \leq 
 12\sqrt{7} \frac{1}{\alpha \sqrt{\log(\e \alpha)}}.
\]
\end{lemma} 
\begin{proof}
    Note that $S_1$ is continuous and strictly decreasing, and $S_1(0) = \sqrt{2/\pi}$. Hence, if $\alpha \geq 1/S_1(1)$, then $\lambdastar \geq 1$ and $S_1(\lambdastar) = 1/\alpha$.
    Hence, by the bounds in \Cref{lem:rational-bounds-moments} we have 
    \begin{equation}\label{eqn:ineq-on-lambda}
    \frac{2}{5\sqrt{2\pi}} \frac{\e^{-(\lambdastar)^2/2}}{(\lambdastar)^2} \leq \frac{1}{\alpha} \leq \frac{2}{\sqrt{2\pi}} \frac{\e^{-(\lambdastar)^2/2}}{(\lambdastar)^2}
    \end{equation}
    Equivalently, 
    \[
    \log(\e \alpha) - \frac{5}{2}\sqrt{2\pi} (\lambdastar)^2 
    \leq 
    \log(\e\alpha)  - \log ((5/2)\e\sqrt{2\pi} (\lambdastar)^2) \leq \frac{(\lambdastar)^2}{2} \leq \log(\e\alpha)  - \log \frac{\e\sqrt{2\pi} (\lambdastar)^2}{2} \leq \log(\e \alpha).
    \]
    Rearranging establishes the desired bounds on $\lambdastar$; we also used $\tfrac{2}{1 + 5 \sqrt{2\pi}} > \tfrac{1}{7}$. Now, using the bounds on $S_2$ from \Cref{lem:rational-bounds-moments} as well as the inequalities~\eqref{eqn:ineq-on-lambda}, we find 
    \[
    \frac{1}{2\sqrt{2}} \frac{1}{\alpha \sqrt{\log(\e \alpha)}}
    \leq 
    \frac{1}{2\alpha \lambdastar}
    \leq 
    \frac{1}{2\sqrt{2\pi}} \frac{\e^{-(\lambdastar)^2/2}}{(\lambdastar)^3}
    \leq S_2(\lambdastar) \leq 
    \frac{6}{\sqrt{2\pi}} \frac{\e^{-(\lambdastar)^2/2}}{(\lambdastar)^3}
    \leq \frac{12}{\alpha \lambdastar} 
    \leq \frac{12\sqrt{7}}{\alpha \sqrt{\log(\e \alpha)}},
    \]
    as claimed.
\end{proof}

\subsubsection{Bounds on empirical quantities}
\label{sec:emp-to-pop}

\begin{lemma} \label{lem:tail-bound-for-first-moment}
For any $u > 0$ and $\lambda > 0$, we have 
\[
\P\Big\{\hat S_1(\lambda) - S_1(\lambda) \leq -u\Big\} 
\leq \exp\Big\{-u^2 \frac{d \lambda^3}{32 \phi(\lambda)}\Big\}
\]
\end{lemma} 
\begin{proof}
Denote $X_\lambda \defn (|\NoiseVec_1| - \lambda)_+$. 
Fix $\gamma > 0$. From $\log(1+u) \leq u$, we have for any $\gamma > 0$,
\begin{multline}\label{eqn:mgf-bound-for-each-summand}
\log \E \e^{\gamma (X_\lambda - \E X_\lambda)}
\leq 
\E(\e^{\gamma X_\lambda} - 1 - \gamma X_\lambda)\\
\stackrel{{\rm(i)}}{=} 2\gamma
\int_0^\infty 
(\e^{\gamma z} - 1) \Psi(\lambda+z) \, \ud z 
\stackrel{{\rm(ii)}}{\leq} 
2\gamma^2 
\int_0^\infty 
z\e^{\gamma z} \Psi(\lambda+z) \, \ud z
\stackrel{{\rm(iii)}}{\leq} 2 \gamma^2 \int_0^\infty 
z \e^{\gamma z} \frac{\phi(\lambda + z)}{z + \lambda} \, \ud z \\
\stackrel{{\rm(iv)}}{\leq} 
2 \gamma^2 \frac{\phi(\lambda)}{\lambda} \int_0^\infty z \e^{-(\lambda - \gamma)z} \, \ud z  \stackrel{{\rm(v)}}{=} 
2 \frac{\phi(\lambda)}{\lambda}
\frac{\gamma^2}{(\lambda -\gamma)^2}.
\end{multline}
Above, inequality (i) used the identity 
$\E f(Z) = \int_0^\infty f'(z) \P\{Z \geq z\}\, \ud z$ for a differentiable function $f$ with $f(0) = 0$; here we took $f(z) = \e^{\gamma z} - 1 - \gamma z$. In (ii) we used the inequality $\e^x - 1 \leq x \e^x$.
Inequality (iii) applied the Mills' ratio bounds from \Cref{lem:Mills-ratio-bounds}; note that the inequality $\Psi(\lambda) \leq \tfrac{\phi(\lambda)}{\lambda}$ holds for any $\lambda > 0$. Inequality (iv) used $z + \lambda \geq \lambda, e^{-z^2/2} \leq 1$, and (v) is an explicit computation. 

Recognizing $\hat S_1(\lambda)$ as an i.i.d. sum of $X_\lambda$, the log moment generating function satisfies
\[
\log \E\e^{\tau (\hat S_1(\lambda) - S_1(\lambda))}
= 
d \log \E \exp(\tfrac{\tau}{d} (X_\lambda - \E X_\lambda))
\leq 
 8 \tau^2  \frac{1}{d} \frac{\phi(\lambda)}{\lambda^3}, 
 \quad \mbox{for all}~\tau < \frac{\lambda d}{2}.
\]
Above, we applied inequality~\eqref{eqn:mgf-bound-for-each-summand} with $\gamma = \tfrac{\tau}{d}$, noting that by the assumption $\tau < \frac{\lambda d}{2}$, we have $\tfrac{1}{(\lambda - \gamma)^2} \leq \frac{8}{\lambda^2}$. 
The claimed inequality on the lower tail now follows from the standard Chernoff inequality. 
\end{proof}

We now bound the probability that $\hat S_1(\lambdastar - \delta) \leq 1/\alpha$ for small $\delta$. We use the auxiliary function 
\[
H(\lambdastar, \delta) = 
\frac{\phi(\lambdastar - \delta)} {\phi(\lambdastar)}\frac{(\lambdastar)^3}{(\lambdastar - \delta)^3}.
\]
\begin{lemma} \label{lem:delta-tail-bound}
Assume $\lambdastar \geq 1$, fix $\delta \in (0, \lambdastar)$, and assume $\alpha \geq \tfrac{1}{S_1(1)}$. We have 
\[
\P\Big\{\hat S_1(\lambdastar - \delta) \leq 
\tfrac{1}{\alpha}\Big\} \leq 
\exp\Big\{-\frac{1}{384\sqrt{14}} \frac{(\lambdastar)^3}{\sigma} \frac{\delta^2}{H(\lambdastar, \delta)}\Big\}.
\]
\end{lemma} 
\begin{proof}
We have $S_1'(\lambda) = -2 \Psi(\lambda)$. Therefore, for any $\delta > 0, \lambda \geq 1$ we have 
\[
S_1(\lambdastar - \delta) - 
S_1(\lambdastar) = 2 \int_{\lambdastar - \delta}^{\lambdastar} \Psi(x) \, \ud x  \geq 2 \delta \Psi(\lambdastar) \geq \delta \frac{\phi(\lambdastar)}{\lambdastar},
\]
where we applied \Cref{lem:Mills-ratio-bounds} for the final inequality.
From the tail bounds in \Cref{lem:tail-bound-for-first-moment}, applied with $\lambda = \lambdastar-\delta$ and $u = \delta \tfrac{\phi(\lambdastar)}{\lambdastar}$, 
it holds that 
\[
\P\Big\{\hat S_1(\lambdastar - \delta) \leq \frac{1}{\alpha}\Big\} \le 
\P\Big\{\hat S_1(\lambdastar - \delta) -  S_1(\lambdastar- \delta) \leq -\delta \frac{\phi(\lambdastar)}{\lambdastar}\Big\}
\leq 
\exp\Big\{- \frac{\lambdastar\delta^2 }{32\sigma H(\lambdastar, \delta)} 
\phi(\lambdastar) \sigma d\Big\}.
\]
Now note that we have 
\[
\sigma d \phi(\lambdastar) 
\stackrel{{\rm(i)}}{\geq} \frac{\sigma d}{6} (\lambdastar)^3 S_2(\lambdastar) 
\stackrel{{\rm(ii)}}{\geq} 
\frac{(\lambdastar)^3}{12\sqrt{2} \sqrt{\log(\e d \sigma)}} 
\stackrel{{\rm(iii)}}{\geq} 
\frac{(\lambdastar)^2}{12\sqrt{14}}.
\]
Combining the previous two displays yields the claim. Above, inequality (i) follows from the inequality for $S_2$ in \Cref{lem:rational-bounds-moments} and inequalities (ii) and (iii) follow from \Cref{lem:bounds-on-lambda-star}, where we recall $\alpha = \sigma d$.
\end{proof}

\subsubsection{Putting the pieces together}
\label{sec:complete-the-proof}

To conclude, we first establish a bound under the assumption that $\alpha \geq 1/S_1(1)$. 
\begin{lemma} \label{lem:bound-in-the-large-alpha-case}
Suppose $\alpha \geq 1/S_1(1)$. Then we have  
\[
\E \|\BallProj(\sigma \NoiseVec)\|_2^2 
\leq 
 205584 \frac{\sigma}{\sqrt{\log(\e d \sigma)}},
\]
for all $\sigma \geq 0, d \geq 1$.
\end{lemma} 
\begin{proof}
Let us set 
\begin{equation}\label{eqn:final-choices}
\delta = \frac{1}{2\lambdastar} \quad 
\mbox{and} \quad 
\cE = \Big\{\hat S_1(\lambdastar - \delta) > \frac{1}{\alpha}\Big\}.
\end{equation}
Note that on $\cE$, we have from the monotonicity of $\hat S_1, \hat S_2$ and \Cref{lem:characterization-of-projection} that $\hat \lambda \geq \lambdastar - \delta$, and consequently $\hat S_2(\hat \lambda) \leq \hat S_2(\lambdastar - \delta)$. 
Thus, we have, using on the complement that $\|\BallProj(\sigma \NoiseVec)\|_2 \leq 1$ deterministically, that
\begin{equation}\label{eqn:final-decomposition}
    \E \|\BallProj(\sigma \NoiseVec)\|_2^2 = 
    \E[\sigma^2 d \hat S_2(\hat \lambda) \ind{\cE}] + \P(\cE^c) \leq 
    \sigma^2 d S_2(\lambdastar - \delta) + \P(\cE^c) \defn T_1 + T_2.
\end{equation}
We now control each term. First, we note that 
\begin{subequations}\label{eqn:term-bounds}
\begin{equation}\label{eqn:bound-on-term-one}
T_1 = \sigma^2 d S_2(\lambdastar) \cdot \frac{S_2(\lambdastar - \delta)}{S_2(\lambdastar)}
\stackrel{{\rm (i)}}\leq 12 \sqrt{7} \frac{\sigma}{\sqrt{\log(\e d \sigma)}} 
\cdot \frac{S_2(\lambdastar - \delta)}{S_2(\lambdastar)} 
\stackrel{{\rm(ii)}}{\leq} 
144 \sqrt{7} 
\frac{\sigma}{\sqrt{\log(\e d \sigma)}} 
H(\lambdastar, \delta).
\end{equation}
Above inequality (i) follows from \Cref{lem:bounds-on-lambda-star}, and inequality (ii) follows from \Cref{lem:rational-bounds-moments}.
To control the second term, we apply \Cref{lem:delta-tail-bound} with our choice of $\delta$, given in display~\eqref{eqn:final-choices}. Using $\e^{-x}\leq \tfrac{1}{x}$, we have 
\begin{equation}\label{eqn:bound-on-term-2}
T_2 \leq 1536\sqrt{14}  \frac{\sigma}{\lambdastar} H(\lambdastar, \delta) \leq
10752 \sqrt{2} \frac{\sigma}{\sqrt{\log(\e d \sigma)}} H(\lambdastar, \delta).
\end{equation}
where the final inequality makes use of \Cref{lem:bounds-on-lambda-star}. 
Finally, since $\lambdastar \geq 1$ under our assumption that $\alpha\geq 1/S_1(1)$, we have 
\begin{equation}\label{ineq:bound-on-H}
H(\lambdastar, \delta) = \sqrt{\e} \frac{1}{(1-\tfrac{1}{2(\lambdastar)^2})^3} 
\e^{-\delta^2/2} \leq 8 \sqrt{\e}.
\end{equation}
\end{subequations}
Combining the decomposition~\eqref{eqn:final-decomposition} with the term-wise bounds~\eqref{eqn:term-bounds} yields
\[
 \E \|\BallProj(\sigma \NoiseVec)\|_2^2
 \leq (144 \sqrt{7} + 10752\sqrt{2}) 8 \sqrt{\e} \frac{\sigma}{\sqrt{\log(\e d \sigma)}}
 < 
 205584 \frac{\sigma}{\sqrt{\log(\e d \sigma)}},
\]
as claimed. 
\end{proof}

To complete the proof of \Cref{lem:medium-sigma-bound}, we need to control the projection when $\sigma d \in [1, \frac{1}{S_1(1)}]$.
However, we can simply use the naïve bound, 
\begin{equation}\label{ineq:upper-in-the-alpha-small-case}
\E \|\BallProj(\sigma \NoiseVec)\|_2^2 
\leq \sigma^2 \E \|\NoiseVec\|_2^2 =
\sigma^2 d
\leq 
\frac{\sigma}{\sqrt{\log(\e d \sigma)}}\,
\sup_{\alpha \in [1, \tfrac{1}{S_1(1)}]}
\alpha \sqrt{\log(\e \alpha)}
\le
11 \frac{\sigma}{\sqrt{\log(\e d \sigma)}},
\end{equation}
where we used the fact that $x^2 \log(\e x)$ is increasing on $[1, \infty)$ and 
\[
S_1(1) = 2 (\phi(1) - \Psi(1)), \quad \mbox{and,} \quad 
\frac{1}{S_1(1)} \sqrt{1 + \log(1/S_1(1))} 
< 11. 
\]
Combining \Cref{lem:bound-in-the-large-alpha-case} with inequality~\eqref{ineq:upper-in-the-alpha-small-case} yields the claim. \qed

\subsubsection*{Acknowledgements}

Reese Pathak gratefully acknowledges support from the National Science Foundation, under grant DMS-2503579.

{\footnotesize
\bibliographystyle{abbrv}
\bibliography{references}

@misc{Zad26,
      title={A Bayesian Proof and Interpretation of Talagrand's Majorizing Measure Theorem}, 
      author={Ilias Zadik},
      year={2026},
      eprint={2605.30321},
      archivePrefix={arXiv},
      primaryClass={math.PR},
      url={https://arxiv.org/abs/2605.30321}, 
}

@misc{PatZhi26,
      title={A remark on the majorizing measures theorem for general processes},
      author={Reese Pathak and Nikita Zhivotovskiy},
      year={2026},
      eprint={2606.03973},
      archivePrefix={arXiv},
      primaryClass={math.PR},
      url={https://arxiv.org/abs/2606.03973},
}

@incollection {SchSch95,
    AUTHOR = {Schechtman, G. and Schmuckenschl\"{a}ger, M.},
     TITLE = {A concentration inequality for harmonic measures on the
              sphere},
 BOOKTITLE = {Geometric Aspects of Functional Analysis ({I}srael,
              1992--1994)},
    SERIES = {Oper. Theory Adv. Appl.},
    VOLUME = {77},
     PAGES = {255--273},
 PUBLISHER = {Birkh\"{a}user, Basel},
      YEAR = {1995},
      ISBN = {3-7643-5207-8},
   MRCLASS = {60J65 (31B99)},
  MRNUMBER = {1353465},
MRREVIEWER = {Jacques\ Vauthier},
}

@misc{AolJorPatUli25,
      title={Revisiting mean estimation over $\ell_p$ balls: Is the {MLE} optimal?}, 
      author={Liviu Aolaritei and Michael I. Jordan and Reese Pathak and Annie Ulichney},
      year={2025},
      eprint={2506.10354},
      archivePrefix={arXiv},
      primaryClass={math.ST},
      url={https://arxiv.org/abs/2506.10354}, 
}

@article {GasUtz14,
    AUTHOR = {Gasull, Armengol and Utzet, Frederic},
     TITLE = {Approximating {M}ills ratio},
   JOURNAL = {J. Math. Anal. Appl.},
  FJOURNAL = {Journal of Mathematical Analysis and Applications},
    VOLUME = {420},
      YEAR = {2014},
    NUMBER = {2},
     PAGES = {1832--1853},
      ISSN = {0022-247X},
   MRCLASS = {60E05},
  MRNUMBER = {3240110},
MRREVIEWER = {C. Satheesh Kumar},
       DOI = {10.1016/j.jmaa.2014.05.034},
       URL = {https://doi.org/10.1016/j.jmaa.2014.05.034},
}

@article {Ney23,
    AUTHOR = {Neykov, Matey},
     TITLE = {On the minimax rate of the {G}aussian sequence model under
              bounded convex constraints},
   JOURNAL = {IEEE Trans. Inform. Theory},
  FJOURNAL = {Institute of Electrical and Electronics Engineers.
              Transactions on Information Theory},
    VOLUME = {69},
      YEAR = {2023},
    NUMBER = {2},
     PAGES = {1244--1260},
      ISSN = {0018-9448},
   MRCLASS = {62C20 (62F12 62H12)},
  MRNUMBER = {4564653},
MRREVIEWER = {Alicja Jokiel-Rokita},
       DOI = {10.1109/tit.2022.3213141},
       URL = {https://doi.org/10.1109/tit.2022.3213141},
}

@article{mendelson2003performance,
  title={On the performance of kernel classes},
  author={Mendelson, Shahar},
  journal={Journal of Machine Learning Research},
  volume={4},
  number={Oct},
  pages={759--771},
  year={2003}
}

@book{boucheron2013concentration,
  title={Concentration Inequalities: A Nonasymptotic Theory of Independence},
  author={Boucheron, St{\'e}phane and Lugosi, G{\'a}bor and Massart, Pascal},
  year={2013},
  publisher={Oxford University Press}
}

@article {BreHub79,
    AUTHOR = {Bretagnolle, J. and Huber, C.},
     TITLE = {Estimation des densit\'{e}s: risque minimax},
   JOURNAL = {Z. Wahrsch. Verw. Gebiete},
  FJOURNAL = {Zeitschrift f\"{u}r Wahrscheinlichkeitstheorie und Verwandte
              Gebiete},
    VOLUME = {47},
      YEAR = {1979},
    NUMBER = {2},
     PAGES = {119--137},
      ISSN = {0044-3719},
   MRCLASS = {62G05},
  MRNUMBER = {523165},
MRREVIEWER = {Luc P. Devroye},
       DOI = {10.1007/BF00535278},
       URL = {https://doi.org/10.1007/BF00535278},
}

@article{mendelson2016upper,
  title={Upper bounds on product and multiplier empirical processes},
  author={Mendelson, Shahar},
  journal={Stochastic Processes and their Applications},
  volume={126},
  number={12},
  pages={3652--3680},
  year={2016},
  publisher={Elsevier}
}

@article{dirksen2015tail,
  title={Tail bounds via generic chaining},
  author={Dirksen, Sjoerd},
  journal={Electronic Journal of Probability},
  volume={20},
  pages={1--29},
  year={2015},
  publisher={The Institute of Mathematical Statistics and the Bernoulli Society},
  doi={10.1214/EJP.v20-3760}
}

@book {Wai19,
    AUTHOR = {Wainwright, M. J.},
     TITLE = {High-dimensional statistics},
    SERIES = {Cambridge Series in Statistical and Probabilistic Mathematics},
    VOLUME = {48},
      NOTE = {A non-asymptotic viewpoint},
 PUBLISHER = {Cambridge University Press, Cambridge},
      YEAR = {2019},
     PAGES = {xvii+552},
      ISBN = {978-1-108-49802-9},
   MRCLASS = {62-01 (60B20 60E15 60Fxx 62Gxx 62Hxx 62Jxx)},
  MRNUMBER = {3967104},
MRREVIEWER = {Pierre Alquier},
       DOI = {10.1017/9781108627771},
       URL = {https://doi.org/10.1017/9781108627771},
}

@article{zhivotovskiy2024dimension,
  title={Dimension-free bounds for sums of independent matrices and simple tensors via the variational principle},
  author={Zhivotovskiy, Nikita},
  journal={Electronic Journal of Probability},
  volume={29},
  pages={1--28},
  year={2024},
  publisher={The Institute of Mathematical Statistics and the Bernoulli Society}
}

@article{AravindaMarsigliettiMelbourne2022ULC,
  author  = {Aravinda, Heshan and Marsiglietti, Arnaud and Melbourne, James},
  title   = {Concentration inequalities for ultra log-concave distributions},
  journal = {Studia Mathematica},
  year    = {2022},
  volume  = {265},
  number  = {1},
  pages   = {111--120}
}

@article{barron1999risk,
  title={Risk bounds for model selection via penalization},
  author={Barron, Andrew and Birg{\'e}, Lucien and Massart, Pascal},
  journal={Probability Theory and Related Fields},
  volume={113},
  number={3},
  pages={301--413},
  year={1999},
  publisher={Springer}
}

@book{koltchinskii2011oracle,
	author = {Koltchinskii, Vladimir},
	publisher = {Springer Science \& Business Media},
	title = {Oracle Inequalities in Empirical Risk Minimization and Sparse Recovery Problems: Ecole d'Et{\'e} de Probabilit{\'e}s de Saint-Flour XXXVIII-2008},
	volume = {2033},
	year = {2011}}

@inproceedings{mourtada2023local,
  title={Local risk bounds for statistical aggregation},
  author={Mourtada, Jaouad and Va{\v{s}}kevi{\v{c}}ius, Tomas and Zhivotovskiy, Nikita},
  booktitle={The Thirty Sixth Annual Conference on Learning Theory},
  pages={5697--5698},
  year={2023}
}

@article{audibert2011robust,
	author = {Audibert, Jean-Yves and Catoni, Olivier},
	journal = {The Annals of Statistics},
	number = {5},
	pages = {2766--2794},
	title = {Robust linear least squares regression},
	volume = {39},
	year = {2011}}

@article{catoni2017dimension,
	author = {Catoni, Olivier and Giulini, Ilaria},
	journal = {arXiv preprint arXiv:1712.02747},
	title = {Dimension-free {PAC-Bayesian} bounds for matrices, vectors, and linear least squares regression},
	year = {2017}}

@book{catoni2007pac,
  title={PAC-Bayesian Supervised Classification: The Thermodynamics of Statistical Learning},
  author={Catoni, Olivier},
  year={2007},
  publisher={Institute of Mathematical Statistics},
  series={Lecture Notes - Monograph Series},
  volume={56},
  isbn={0-940600-72-2}
}

@article{alexander1987rates,
  title={Rates of growth and sample moduli for weighted empirical processes indexed by sets},
  author={Alexander, Kenneth S},
  journal={Probability Theory and Related Fields},
  volume={75},
  number={3},
  pages={379--423},
  year={1987},
  publisher={Springer}
}

@article{amelunxen2014living,
  title={Living on the edge: Phase transitions in convex programs with random data},
  author={Amelunxen, Dennis and Lotz, Martin and McCoy, Michael B and Tropp, Joel A},
  journal={Information and Inference: A Journal of the IMA},
  volume={3},
  number={3},
  pages={224--298},
  year={2014},
  publisher={Oxford University Press}
}

@inproceedings{lotz2020concentration,
  title={Concentration of the intrinsic volumes of a convex body},
  author={Lotz, Martin and McCoy, Michael B and Nourdin, Ivan and Peccati, Giovanni and Tropp, Joel A},
  booktitle={Geometric Aspects of Functional Analysis: Israel Seminar (GAFA) 2017-2019 Volume II},
  pages={139--167},
  year={2020},
  organization={Springer}
}

@article{han2023noisy,
  title={Noisy linear inverse problems under convex constraints: Exact risk asymptotics in high dimensions},
  author={Han, Qiyang},
  journal={The Annals of Statistics},
  volume={51},
  number={4},
  pages={1611--1638},
  year={2023},
  publisher={Institute of Mathematical Statistics}
}

@article{Sudakov1976,
  author={Sudakov, V. N.},
  title={Geometric problems of the theory of infinite-dimensional probability distributions},
  journal={Trudy Mat. Inst. Steklov},
  volume={141},
  pages={3--191},
  year={1976},
}

@article{massart2000some,
  title={Some applications of concentration inequalities to statistics},
  author={Massart, Pascal},
  journal={Annales de la Facult{\'e} des sciences de Toulouse: Math{\'e}matiques},
  volume={9},
  number={S2},
  pages={245--303},
  year={2000},
  publisher={Universit{\'e} Paul Sabatier}
}

@book{vandegeer2000empirical,
  title={Empirical Processes in M-Estimation},
  author={van de Geer, Sara A},
  year={2000},
  publisher={Cambridge University Press},
  series={Cambridge Series in Statistical and Probabilistic Mathematics},
  volume={6},
  isbn={9780521650021}
}

@article{birge1993rates,
  title={Rates of convergence for minimum contrast estimators},
  author={Birg{\'e}, Lucien and Massart, Pascal},
  journal={Probability Theory and Related Fields},
  volume={97},
  number={1-2},
  pages={113--150},
  year={1993},
  publisher={Springer}
}

@article{liu2025simple,
  title={Simple and Sharp Generalization Bounds via Lifting},
  author={Liu, Jingbo},
  journal={arXiv preprint arXiv:2508.18682},
  year={2025}
}

@article{giannopoulos2018inequalities,
  title={Inequalities for the surface area of projections of convex bodies},
  author={Giannopoulos, Apostolos and Koldobsky, Alexander and Valettas, Petros},
  journal={Canadian Journal of Mathematics},
  volume={70},
  number={4},
  pages={804--823},
  year={2018},
  publisher={Cambridge University Press}
}

@article{liu2022minoration,
  title={Minoration via mixed volumes and {C}over’s problem for general channels},
  author={Liu, Jingbo},
  journal={Probability Theory and Related Fields},
  volume={183},
  number={1},
  pages={315--357},
  year={2022},
  publisher={Springer}
}

@book{Talagrand2021,
  author    = {Talagrand, Michel},
  title     = {Upper and Lower Bounds for Stochastic Processes: Decomposition Theorems},
  year      = {2021},
  publisher = {Springer},
  volume    = {60}
}

@article{AHY21,
  author  = {David Alonso-Guti{\'e}rrez and Mar{\'i}a A. Hern{\'a}ndez Cifre and Jes{\'u}s Yepes Nicol{\'a}s},
  title   = {Further inequalities for the (generalized) {W}ills functional},
  journal = {Communications in Contemporary Mathematics},
  volume  = {23},
  number  = {3},
  pages   = {2050011},
  year    = {2021},
  doi     = {10.1142/S021919972050011X}
}

@article{bellec2018sharp,
  title={Sharp oracle inequalities for least squares estimators in shape restricted regression},
  author={Bellec, Pierre C},
  journal={The Annals of Statistics},
  volume={46},
  number={2},
  pages={745--780},
  year={2018},
  publisher={JSTOR}
}

@article{shtar1987universal,
  title={Universal sequential coding of single messages},
  author={Shtar'kov, Yurii Mikhailovich},
  journal={Problemy Peredachi Informatsii},
  volume={23},
  number={3},
  pages={3--17},
  year={1987},
  publisher={Russian Academy of Sciences, Branch of Informatics, Computer Equipment and~…}
}

@article{mcmullen1991inequalities,
  title={Inequalities between intrinsic volumes},
  author={McMullen, Peter},
  journal={Monatshefte f{\"u}r Mathematik},
  volume={111},
  number={1},
  pages={47--53},
  year={1991},
  publisher={Springer}
}

@article{vitale1996wills,
  title={The {W}ills functional and {G}aussian processes},
  author={Vitale, Richard A},
  journal={The Annals of Probability},
  volume={24},
  number={4},
  pages={2172--2178},
  year={1996},
  publisher={Institute of Mathematical Statistics}
}

@article{betke1993intrinsic,
  title={Intrinsic volumes and lattice points of crosspolytopes},
  author={Betke, Ulrich and Henk, Martin},
  journal={Monatshefte f{\"u}r Mathematik},
  volume={115},
  number={1},
  pages={27--33},
  year={1993},
  publisher={Springer}
}

@article{artstein2004duality,
  title={Duality of metric entropy},
  author={Artstein, Shirit and Milman, Vitali and Szarek, Stanis{\l}aw J},
  journal={Annals of Mathematics},
  pages={1313--1328},
  year={2004},
  publisher={JSTOR}
}

@misc{vershynin_gfa_notes,
  author    = {Vershynin, Roman},
  title     = {Lectures in Geometric Functional Analysis},
  year      = {2011},
  howpublished = {Lecture notes},
  url       = {https://www.math.uci.edu/~rvershyn/papers/GFA-book.pdf}
}

@misc{prasadan2025informationtheoreticlimitsrobust,
      title={Information theoretic limits of robust sub-{G}aussian mean estimation under star-shaped constraints}, 
      author={Akshay Prasadan and Matey Neykov},
      year={2025},
      eprint={2412.03832},
      archivePrefix={arXiv},
      primaryClass={math.ST},
      url={https://arxiv.org/abs/2412.03832}, 
}

@article{hadwiger1975will,
  title={Das {W}ill'sche {F}unktional},
  author={Hadwiger, Hugo},
  journal={Monatshefte f{\"u}r Mathematik},
  volume={79},
  number={3},
  pages={213--221},
  year={1975},
  publisher={Springer}
}

@article{wills1973gitterpunktanzahl,
  title={Zur {G}itterpunktanzahl {K}onvexer {M}engen.},
  author={Wills, J{\"o}rg M},
  journal={Elemente der Mathematik},
  volume={28},
  pages={57--63},
  year={1973}
}

@article{mourtada2025universal,
  title={Universal coding, intrinsic volumes, and metric complexity},
  author={Mourtada, Jaouad},
  journal={Journal of the European Mathematical Society},
  year={2025}
}

@inproceedings{miyaguchi2019adaptive,
  title={Adaptive minimax regret against smooth logarithmic losses over high-dimensional $\ell_1$-balls via envelope complexity},
  author={Miyaguchi, Kohei and Yamanishi, Kenji},
  booktitle={The 22nd International Conference on Artificial Intelligence and Statistics},
  pages={3440--3448},
  year={2019},
  organization={PMLR}
}

@article{van2018chainingtwo,
  title={Chaining, interpolation and convexity {II}: The contraction principle},
  author={{v}an Handel, Ramon},
  journal={The Annals of Probability},
  volume={46},
  number={3},
  pages={1764--1805},
  year={2018},
  publisher={JSTOR}
}

@article{prasadan2025some,
  title={Some facts about the optimality of the {LSE} in the {G}aussian sequence model with convex constraint},
  author={Prasadan, Akshay and Neykov, Matey},
  journal={IEEE Transactions on Information Theory},
  year={2025},
  publisher={IEEE}
}

@article{van2018chaining,
  title={Chaining, interpolation, and convexity},
  author={{v}an Handel, Ramon},
  journal={Journal of the European Mathematical Society},
  volume={20},
  number={10},
  pages={2413--2435},
  year={2018}
}

@article{nelson2016chaining,
  author       = {Nelson, Jelani},
  title        = {Chaining {I}ntroduction With Some {C}omputer {S}cience {A}pplications},
  journal      = {Bulletin of the European Association for Theoretical Computer Science (EATCS)},
  year         = {2016},
  note         = {Bulletin of EATCS, No.\ 120}
}

@article{fernandez2025convex,
  title={On Convex Functions of {G}aussian Variables},
  author={Fern{\'a}ndez-Unzueta, Maite and Melbourne, James and Palafox-Castillo, Gerardo},
  journal={arXiv preprint arXiv:2510.06676},
  year={2025}
}

@book{ledoux2013probability,
	author = {Ledoux, Michel and Talagrand, Michel},
	publisher = {Springer Science \& Business Media},
	title = {Probability in Banach Spaces: {I}soperimetry and Processes},
	year = {2013}}

@book{polyanskiy2025information,
  title     = {Information Theory: From Coding to Learning},
  author    = {Polyanskiy, Yury and Wu, Yihong},
  year      = {2025},
  publisher = {Cambridge University Press},
  address   = {Cambridge, United Kingdom and New York, NY, USA},
  isbn      = {9781108832908},
  doi       = {10.1017/9781108966351}
}

@article{bshouty2009using,
  title={Using the doubling dimension to analyze the generalization of learning algorithms},
  author={Bshouty, Nader H and Li, Yi and Long, Philip M},
  journal={Journal of Computer and System Sciences},
  volume={75},
  number={6},
  pages={323--335},
  year={2009},
  publisher={Elsevier}
}

@incollection{bednorz2023sudakov,
  title={Sudakov Minoration for Products of Radial-Type Log-Concave Measures},
  author={Bednorz, Witold},
  booktitle={High Dimensional Probability IX: The Ethereal Volume},
  pages={257--295},
  year={2023},
  publisher={Springer}
}

@article{stein1981estimation,
  title={Estimation of the mean of a multivariate normal distribution},
  author={Stein, Charles M},
  journal={The Annals of Statistics},
  pages={1135--1151},
  year={1981},
  publisher={JSTOR}
}

@article{meyer2000degrees,
  title={On the degrees of freedom in shape-restricted regression},
  author={Meyer, Mary and Woodroofe, Michael},
  journal={The Annals of Statistics},
  volume={28},
  number={4},
  pages={1083--1104},
  year={2000},
  publisher={Institute of Mathematical Statistics}
}

@article{kato2009degrees,
  title={On the degrees of freedom in shrinkage estimation},
  author={Kato, Kengo},
  journal={Journal of Multivariate Analysis},
  volume={100},
  number={7},
  pages={1338--1352},
  year={2009},
  publisher={Elsevier}
}

@article{Latala2014Sudakov,
  author    = {Rafa{\l} Lata{\l}a},
  title     = {Sudakov-type minoration for log-concave vectors},
  journal   = {Studia Mathematica},
  volume    = {223},
  number    = {3},
  pages     = {251--274},
  year      = {2014}
}

@inproceedings{KurEtal23,
 author = {Kur, G. and Putterman, E. and Rakhlin, A.},
 booktitle = {Advances in Neural Information Processing Systems},
 editor = {A. Oh and T. Naumann and A. Globerson and K. Saenko and M. Hardt and S. Levine},
 pages = {37527--37539},
 publisher = {Curran Associates, Inc.},
 title = {On the Variance, Admissibility, and Stability of Empirical Risk Minimization},
 url = {https://proceedings.neurips.cc/paper_files/paper/2023/file/7644353d580a9e027e0069d6480d971b-Paper-Conference.pdf},
 volume = {36},
 year = {2023}
}

@article{zhivotovskiy2018localization,
  title={Localization of {VC} classes: {B}eyond local {R}ademacher complexities},
  author={Zhivotovskiy, Nikita and Hanneke, Steve},
  journal={Theoretical Computer Science},
  volume={742},
  pages={27--49},
  year={2018},
  publisher={Elsevier}
}

@article{yang1999information,
  title={Information-theoretic determination of minimax rates of convergence},
  author={Yang, Yuhong and Barron, Andrew},
  journal={Annals of Statistics},
  pages={1564--1599},
  year={1999},
  publisher={JSTOR}
}

@book{vershynin2018high,
  title={High-Dimensional Probability: An Introduction with Applications in Data Science},
  author={Vershynin, Roman},
  volume={47},
  year={2018},
  publisher={Cambridge University Press}
}

@article{Mendelson2017LocalGlobal,
  author    = {Shahar Mendelson},
  title     = {“{L}ocal” vs. “global” parameters — breaking the {G}aussian complexity barrier},
  journal   = {Annals of Statistics},
  year      = {2017},
  volume    = {45},
  number    = {5},
  pages     = {2265--2299},
  doi       = {10.1214/16-AOS1510},
  url       = {https://projecteuclid.org/euclid.aos/1493730318},
}

@article{mmn2019,
  title={Generalized dual {S}udakov minoration via dimension-reduction--a program},
  author={Mendelson, S. and Milman, E. and Netzer, N.},
  journal={Studia Mathematica},
  volume={244},
  number={2},
  pages={159--202},
  year={2019}
}

@book {BakGenLed14,
    AUTHOR = {Bakry, D. and Gentil, I. and Ledoux, M.},
     TITLE = {Analysis and geometry of {M}arkov diffusion operators},
    SERIES = {Grundlehren der mathematischen Wissenschaften [Fundamental
              Principles of Mathematical Sciences]},
    VOLUME = {348},
 PUBLISHER = {Springer, Cham},
      YEAR = {2014},
     PAGES = {xx+552},
      ISBN = {978-3-319-00226-2; 978-3-319-00227-9},
   MRCLASS = {60J25 (58J65 60J35 60J60)},
  MRNUMBER = {3155209},
MRREVIEWER = {Ming Liao},
       DOI = {10.1007/978-3-319-00227-9},
       URL = {https://doi.org/10.1007/978-3-319-00227-9},
}

@article{yi2025nonparametric,
  title={Nonparametric Exponential Family Regression Under Star-Shaped Constraints},
  author={Yi, Guanghong and Neykov, Matey},
  journal={arXiv preprint arXiv:2503.10794},
  year={2025}
}

@book {RocWets98,
    AUTHOR = {Rockafellar, R. Tyrrell and Wets, Roger J.-B.},
     TITLE = {Variational Analysis},
    SERIES = {Grundlehren der Mathematischen Wissenschaften [Fundamental
              Principles of Mathematical Sciences]},
 PUBLISHER = {Springer-Verlag, Berlin},
      YEAR = {1998},
     PAGES = {xiv+733},
      ISBN = {3-540-62772-3},
   MRCLASS = {49-02 (46N10 47N10 49J52 49K40 90C30)},
  MRNUMBER = {1491362},
MRREVIEWER = {Francis\ H.\ Clarke},
       DOI = {10.1007/978-3-642-02431-3},
       URL = {https://doi.org/10.1007/978-3-642-02431-3},
}

@article {Cha14,
    AUTHOR = {Chatterjee, S.},
     TITLE = {A new perspective on least squares under convex constraint},
   JOURNAL = {Ann. Statist.},
  FJOURNAL = {The Annals of Statistics},
    VOLUME = {42},
      YEAR = {2014},
    NUMBER = {6},
     PAGES = {2340--2381},
      ISSN = {0090-5364},
   MRCLASS = {62F10 (62F12 62F30 62G08)},
  MRNUMBER = {3269982},
MRREVIEWER = {Zakhar Kabluchko},
       DOI = {10.1214/14-AOS1254},
       URL = {https://doi.org/10.1214/14-AOS1254},
}

@article{gine2006concentration,
  author       = {Gin{\'e}, Evarist and Koltchinskii, Vladimir},
  title        = {Concentration inequalities and asymptotic results for ratio type empirical processes},
  journal      = {The Annals of Probability},
  volume       = {34},
  number       = {3},
  pages        = {1143--1216},
  year         = {2006}
}
}

\appendix 
\section{Proof of \Cref{pr:metric-characterization-of-stat-rate}}
\label{app:GSM}

 Throughout this section we make use of the following notation as a shorthand for the $\eps$-local packing entropy of the set $A \subset \R^\dimension$,
\begin{equation}\label{defn:local-packing-entropy-general}
M^{\rm loc}_A(\eps) 
\defn 
\sup_{\delta \geq \eps} 
\sup_{x \in A}
M\big(A \cap (x + 2\delta B^d_2), \delta B^d_2\big).
\end{equation}
In the case that $A \subset \R^\dimension$ is convex, note that the monotonicity of the local packing entropy implies the outer supremum is achieved at $\delta = \eps$, and hence definitions~\eqref{def:local-packing-entropy-convex} and~\eqref{defn:local-packing-entropy-general} coincide in this case. 
Finally, if $A$ is also centrally symmetric, the supremum over $x \in A$ is achieved at $x = 0$.
Let us also introduce the fixed point
\[
\MetricChar(\sigma)  
\defn
\sup\bigg\{\, \eps > 0 \mid \log M^{\rm loc}_T(\eps) \geq \frac{\eps^2}{\sigma^2} \,\bigg\}.
\]
In order to prove the claim, we will show that there exist universal constants $c_\ell, c_u > 0$ such that 
\begin{equation}\label{ineq:desired-bounds}
c_\ell \, \MetricChar(\sigma)
\stackrel{{\rm (i)}}{\leq}
\eps_\star(\sigma) 
\stackrel{{\rm (ii)}}{\leq}
c_u \, \MetricChar(\sigma).
\end{equation}

\paragraph{Proof of lower bound, inequality~\eqref{ineq:desired-bounds}(i):}
Fix an $\eps > 0$ such that 
\[
\frac{\eps^2}{\sigma^2} \leq \log M^{\rm loc}_T(\eps).
\]
Then, \Cref{lem:equivalent-versions-of-the-minimax-rate} applied with $c_1 = 1, c_2 = 1, c_3 = 2$ yields a constant $c_5 = 1/c_4(c_1, c_2, c_3) \approx 0.0315$ such that 
\[
c_5 \, \eps \leq \eps_\star(\sigma).
\]
Taking the supremum over all such $\eps > 0$ yields inequality~\eqref{ineq:desired-bounds}(i) with $c_\ell = c_5$.  \qed

\paragraph{Proof of upper bound, inequality~\eqref{ineq:desired-bounds}(ii):}

The next result shows that by computing the least squares estimate \emph{over a suitably chosen net of the underlying body $T$}, we can obtain an essentially matching upper estimate.
% \footnote{As discussed in~\cite[Section~4]{Ney23}, a potential benefit of iterative schemes, such as those proposed in~\cite{Ney23, prasadan2025informationtheoreticlimitsrobust, yi2025nonparametric}, is that they may adapt to the local entropy at the given $\thetastar$, as opposed to the worst-case target in $T$.}

\begin{proposition}
\label{prop:erm-upper}
Let $T \subset \R^\dimension$. Let $\eps > 0$ be such that the inequality 
\[
\log M^{\rm loc}_T(\eps) \leq \frac{\eps^2}{\sigma^2},
\]
holds. 
Let $\cM_\eps$ denote a maximal $\eps$-(global) packing of $T$ with respect to $\ell^d_2$. The estimator
\[
\hat \theta(Y) \defn \argmin_{\eta \in \cM_\eps} \|\eta - Y\|_2^2
\]
enjoys the following risk bound
\[
\sup_{\thetastar \in T} 
\E_{Y \sim \Normal{\thetastar}{\sigma^2 I_\dimension}} 
\|\hat \theta(Y) - \thetastar\|_2^2 
\lesssim \eps^2. 
\]
\end{proposition}

Note that \Cref{prop:erm-upper} does \emph{not} require $T$ to be convex; however if $T$ is not convex, then the supremum over radii $\delta \geq \eps$ in the definition of the local entropy $M^{\rm loc}_T(\eps)$ may not be achieved with $\delta = \eps$. We also explain how this result implies inequality~\eqref{ineq:desired-bounds}(ii). 
Fix any $\eps' > \MetricChar(\sigma) > 0$. By definition, the inequality $\log M^{\rm loc}_T(\eps') \leq (\eps')^2/\sigma^2$ must hold,
and hence we may apply \Cref{prop:erm-upper} with $\eps = \eps'$. We obtain the guarantee 
\begin{equation}\label{ineq:minimax-upper}
\eps_\star^2(\sigma) 
\leq 
\sup_{\thetastar \in T} 
\E_{Y \sim \Normal{\thetastar}{\sigma^2 I_d}}
\Big[\|\hat \theta_{\eps'}(Y) - \thetastar\|_2^2\Big]
\leq C \; (\eps')^2, 
\end{equation}
where $C > 0$ is the constant implicit in \Cref{prop:erm-upper} and $\hat \theta_{\eps'}$ denotes the projection onto a 
$(\eps')$-global packing of $T$ (with respect to $\ell^d_2$). Since this holds for every $\eps' > \overline{\eps}(\sigma)$, passing to the infimum in display~\eqref{ineq:minimax-upper}
yields inequality~\eqref{ineq:desired-bounds}(ii) with $c_u = \sqrt{C}$. 

\paragraph{Proof of \Cref{prop:erm-upper}:}
We first claim that we may assume that $M^{\rm loc}_T(\eps) \geq 2$. Indeed, suppose that $\diam(T) > \eps$. 
In this case, there exist $x, y \in T$ such that $\|x - y\|_2 > \eps$. Define $\delta = \max\{\eps, \|x-y\|_2/2\}$. Clearly, $x \in y + 2\delta B^d_2$, and hence the set $\{x,y\}$ is a $\delta$-packing of $T \cap (y + 2\delta B^d_2)$, and hence $M^{\rm loc}_T(\eps) \geq 2$. Taking the contrapositive, we see that if $M^{\rm loc}_T(\eps) \leq 1$, then $\eps \geq \diam(T)$, and hence we immediately have 
\[
\sup_{\thetastar \in T} \E_{Y \sim \Normal{\thetastar}{\sigma^2 I_d}} \|\hat \theta(Y) - \thetastar\|_2^2 \leq 4 \diam(T)^2 \leq 4 \eps^2.
\]
Consequently, for the remainder for the argument we assume $M^{\rm loc}_T(\eps) \geq 2$.

Now, let $\thetastar_\eps = \argmin_{\vartheta \in \cM_\eps} \|\vartheta - \thetastar\|_2^2$; we use the shorthand notation $\hat \theta \equiv \hat \theta(Y)$. 
The triangle inequality gives us 
\begin{equation}\label{ineq:triangle-to-net}
\|\hat \theta - \thetastar\|_2^2 \leq 2 \|\hat \theta - \thetastar_\eps\|_2^2 + 2\eps^2.
\end{equation} 
From the basic inequality $\|\hat \theta - Y\|_2^2 \leq \|\thetastar_\eps - Y\|_2^2$ we conclude from Young's inequality
\begin{align}
\|\hat \theta - \thetastar_\eps\|_2^2 &\leq 2  \langle \xi, \hat \theta 
- \thetastar_\eps \rangle +  2 \langle  \thetastar - \thetastar_\eps, \hat \theta - \thetastar_\eps \rangle \\
&\leq 
2  \langle \xi, \hat \theta 
- \thetastar_\eps \rangle + \frac{1}{2} \|\hat \theta - \thetastar_\eps\|_2^2 
+ 2 \eps^2
\end{align}
Rearranging, we get 
\begin{equation}\label{ineq:error-on-net}
\|\hat \theta - \thetastar_\eps\|_2^2  \leq 8 \eps^2 +  8 \Big( \langle \xi, \hat \theta - \thetastar_\eps\rangle - \frac{1}{8} \|\hat \theta- \thetastar_\eps\|_2^2 \Big). 
\end{equation}
Define the sets, for an integer $\ell \geq 0$,
\[
\cD \defn \cM_\eps - \thetastar_\eps, 
\quad 
\cD_0 = \cD \cap 2^{\ell + 1}\eps B^d_2 
\quad \mbox{and} \quad 
\cD_k \defn 
(\cD \cap  2^{k + \ell + 1}\eps B^d_2) \setminus 
    (\cD \cap  2^{k + \ell}\eps B^d_2), \quad \mbox{for}~k\geq 1.
\]
Combining inequalities~\eqref{ineq:triangle-to-net}
and~\eqref{ineq:error-on-net} yields
\begin{align}
\E \|\hat \theta - \thetastar\|_2^2 
&\leq 18 \eps^2 + 
16 \E\bigg[ \sup_{\Delta \in \cD} \, \Big\{ \langle \xi, \Delta\rangle 
- \frac{1}{8}\|\Delta\|_2^2\Big\}\bigg] \\
&\stackrel{{\rm(i)}}{\leq}
18 \eps^2 + 512 \sigma^2 
\log \Big(|\cD_0| \e^{\tfrac{4^{\ell + 1} \eps^2}{2048\sigma^2}} + \sum_{k = 1}^\infty 
     |\cD_k| \e^{-\tfrac{4^{k + \ell + 1} \eps^2}{2048 \sigma^2}}\Big) \\ 
&\stackrel{{\rm (ii)}}{\leq} 
18 \eps^2 + 512 \sigma^2 
\log \Big( \e^{ \tfrac{\eps^2}{\sigma^2} (\tfrac{4^{\ell + 1}}{2048} + \ell + 1)}  + \sum_{j = 2 + \ell}^\infty 
     \e^{-\tfrac{\eps^2}{\sigma^2}[\tfrac{4^{j}}{2048} - j]}\Big).
\end{align}
Above, inequality~(i) follows via a standard peeling argument (e.g., \Cref{lem:peeling}, with $T = \cD, r = \tfrac{1}{8}, \delta = 2^{\ell+1} \eps$), and inequality (ii) follows from the containment $\thetastar_\eps + \cD_k\subset T \cap B(\thetastar_\eps, 2^{k+\ell +1} \eps)$, along with the fact that $\cD_k$ is an $\eps$-separated packing set in $\ell^d_2$ and hence a standard packing argument (\Cref{lem:packing}) yields
\[
|\cD_k| \leq 
 M\Big(T \cap B(\thetastar_\eps, 2^{k + \ell + 1}\eps), \eps B^d_2\Big) 
\leq 
M^{\rm loc}_T(\eps)^{k + \ell + 1} 
\leq \exp\Big((k + \ell + 1)\frac{\eps^2}{\sigma^2}\Big).
\]
Set $\ell = 10$, and note $4^j/2048 - j \geq 2^j$ for $j \geq 12$. We obtain 
\[
\E \|\hat \theta - \thetastar\|_2^2  
\leq 18 \eps^2 + 
512 \sigma^2 
\log \Big( \e^{\tfrac{2^{12} \eps^2}{\sigma^2}}  + \sum_{j = 0}^\infty 
     \e^{-\tfrac{2^{12} \eps^2}{\sigma^2} 2^j}\Big) 
\leq C \, \eps^2.
\]
Above, we may take $C = 2,097,682$. The final inequality used the fact that $\eps^2/\sigma^2 \geq \log 2$, as follows from $\eps > 0$ and $\eps^2/\sigma^2 \geq \log M^{\rm loc}_T(\eps)$. \qed

\subsection{Supporting lemmas for \Cref{pr:metric-characterization-of-stat-rate}}

We recall that $\xi$ is $\sigma$-sub-Gaussian if it holds that 
\[
\log \E \exp(\lambda \langle \xi, \theta \rangle) 
\leq \frac{\lambda^2 \sigma^2 \|\theta\|_2^2}{2} \quad \mbox{for}~\lambda \in \R, \theta \in \R^\dimension.
\]

\begin{lemma} [Standard peeling]
\label{lem:peeling}
Fix $T \subset \R^\dimension$ and $r > 0$. Consider the process 
$X_t \defn \langle \xi, t \rangle  - r\|t\|_2^2$. 
For $\delta > 0$, define the annuli 
\[
T_0(\delta) \defn T \cap \delta B^d_2 
\qquad \mbox{and} \quad 
    T_k(\delta) \defn (T \cap 2^k \delta B^d_2) \setminus 
    (T \cap 2^{k-1} \delta B^d_2), \quad \mbox{for}~k\geq 1.
\]
If $\xi$
is a $\sigma$-sub-Gaussian random vector, then 
\[
\E \Big[\sup_{t \in T} X_t \Big]\leq 
    \frac{4\sigma^2}{r} \;
\inf_{\delta > 0}
    \log \Big( |T_0(\delta)| \e^{\tfrac{\delta^2 r^2}{32 \sigma^2}} + \sum_{k = 1}^\infty 
    |T_k(\delta)| \e^{-\tfrac{4^k \delta^2 r^2}{32 \sigma^2}}\Big).
\]
\end{lemma} 
\begin{proof}
    Fix $\delta > 0$; we use the shorthand $T_k \defn T_k(\delta)$.
    For the $k$th term, by sub-Gaussianity 
    and the definition of $T_k$, we obtain for $k \geq 1$
    \begin{equation}\label{ineq:termwise-bound}
    \E \, \sup_{t \in T_k} \e^{\gamma X_t} 
    % \leq |T_k| \max_{t \in T_k} \E \e^{\gamma X_t} 
    \leq 
    |T_k| \max_{t \in T_k} 
    \exp\Big( \frac{\gamma^2 \sigma^2}{2} \|t\|_2^2 - r \gamma \|t\|_2^2\Big)
    \leq 
    |T_k| \exp\Big( 4^{k} \delta^2 \big[\tfrac{\gamma^2 \sigma^2}{2} - \tfrac{r\gamma}{4}\big]\Big).
    \end{equation}
    To combine the annuli, note the concavity of the logarithm yields 
    \begin{equation}\label{ineq:peeling-bound}
    \E\Big[ \sup_{t \in T} X_t\Big]
    \leq 
    \inf_{\gamma > 0} \; \Big\{
    \frac{1}{\gamma} 
    \log \E \, \sup_{t \in T} \e^{\gamma X_t}\Big\}
    \leq 
    \inf_{\gamma > 0} \; \bigg\{
    \frac{1}{\gamma} 
    \log 
    \sum_{k = 0}^\infty 
     \E\Big[ \sup_{t \in T_k} \e^{\gamma X_t} \Big]\bigg\}.
    \end{equation}
    Combining inequalities~\eqref{ineq:termwise-bound} and~\eqref{ineq:peeling-bound} and taking $\gamma =  \tfrac{r}{4\sigma^2}$, we find
    \[
    \E\Big[ \sup_{t \in T} X_t\Big]
    \leq 
    \frac{4\sigma^2}{r} 
    \log \Big( \E \, \sup_{t \in T_0} \e^{\gamma X_t} + \sum_{k = 1}^\infty 
    |T_k| \e^{-\tfrac{4^k \delta^2 r^2}{32 \sigma^2}}\Big) 
    \leq 
    \frac{4 \sigma^2}{r}
        \log \Big( |T_0| \e^{\tfrac{\delta^2 r^2}{32 \sigma^2}} + \sum_{k = 1}^\infty 
    |T_k| \e^{-\tfrac{4^k \delta^2 r^2}{32 \sigma^2}}\Big).
    \]
    This bound holds for all $\delta > 0$. 
    Passing to the infimum over $\delta > 0$ yields the claim.
\end{proof}

The next result is standard, and relates local packing numbers on different scales.

\begin{lemma} [Standard packing]
\label{lem:packing}
For any $A \subset \R^\dimension$, any $x \in A$, and any $\gamma \geq \eps$, we have 
\[
\log M\Big(A \cap (x + 2 \gamma B_2^\dimension), \eps B_2^d\Big) 
\leq \ceil{\log_2(2\gamma/\eps)}
\; 
\log 
M^{\rm loc}_A(\eps).
\]
\end{lemma} 
\begin{proof}
Throughout the proof we set $B = B_2^\dimension$ and $B(x, \gamma) = x + \gamma B_2^d$.
Suppose for the moment that for any $x \in A$ and any $\delta \geq \eps > 0$ it holds that
\begin{equation}\label{eqn:main-ineq}
M\Big(A \cap B(x, 2^{k} \delta), \delta B\Big) 
\leq 
\Big[M^{\rm loc}_A(\eps)\Big]^k 
\quad \mbox{for all integers}~k \geq 1.
\end{equation}
Then, for $k = \ceil{\log_2(2\gamma/\eps)}$ we have by~\eqref{eqn:main-ineq} that
\[
M\Big(A \cap B(x, 2\gamma), \eps B\Big) 
\leq 
M\Big(A \cap B(x, 2^{k}\eps), \eps B\Big) 
\leq 
\Big[M^{\rm loc}_A(\eps)\Big]^k,
\]
which proves the claim. 
We now establish~\eqref{eqn:main-ineq}, for fixed $x \in A, \delta > 0$, by induction. The claim obviously holds with $k = 1$. Let 
$A_0$ denote a maximal 
$2^k \delta$-packing of $A \cap B(x, 2^{k+1} \delta)$. By the maximality of this packing, we have
\[
A \cap B(x, 2^{k+1} \delta) 
\subset \bigcup_{y \in A_0} 
B(y, 2^k\delta).
\]
Therefore, the maximal packing satisfies
\[
M(A \cap B(x, 2^{k+1} \delta), \delta B) 
\leq 
|A_0| \sup_{y \in A_0} 
M(A \cap B(y, 2^k \delta), \delta B)
\leq 
\Big[M^{\rm loc}_A(\eps)\Big]^{k+1},
\]
as desired. Note that above, we used the induction hypothesis to conclude 
\[
\sup_{y \in A_0} M(A \cap B(y, 2^k \delta), \delta B) 
\leq 
M^{\rm loc}_A(\eps)^k,~~\mbox{and}~~
|A_0| = 
M(A \cap B(x, 2^{k+1} \delta), 2^k \delta B)
\leq 
M^{\rm loc}_A(\eps).\qedhere
\]
\end{proof}

\begin{lemma}
\label{lem:trivial-lower-bound}
For every $\sigma > 0$, and any convex, bounded $T \subset \R^\dimension$, we have 
\[
\eps_\star(\sigma) \geq \frac{1}{4} \big(\twomin{\sigma}{\diam(T)}\big).
\]
\end{lemma} 
\begin{proof}
Resorting to an approximating sequence if necessary, we assume that $t, s \in T$ satisfy $\|t - s\|_2 = \diam(T)$. For $\alpha \defn \twomin{\sigma}{\|t-s\|_2}$, we can construct two points $t_{\alpha}, s_{\alpha}$ on the line segment connecting $t, s$ such that $\|t_\alpha - s_\alpha\|_2 = \alpha$. Given that $t_\alpha, s_\alpha \in T$, we may apply Le Cam's method and Pinsker's inequality~\cite[Chapter 15]{Wai19} to these two points, yielding
\[
\eps_\star^2(\sigma)\geq \frac{\|t_\alpha - s_\alpha\|_2^2}{8}
\Big\{1 - \sqrt{\frac{1}{4 \sigma^2} \|t_\alpha - s_\alpha\|_2^2}\Big\} 
\geq \frac{1}{16} \alpha^2 = 
\frac{1}{16} \min\{\sigma^2,\diam(T)^2\}.\qedhere
\]
\end{proof}

\begin{lemma} 
\label{lem:equivalent-versions-of-the-minimax-rate}
Let $\sigma > 0$ and fix a convex set $T \subset \R^\dimension$. Suppose $\delta > 0$ is such that for some constants $c_1, c_2, c_3 \geq 1$ we have 
\[
\frac{\delta^2}{c_1 \sigma^2} 
\leq \sup_{\mu \in T} 
\log M\Big(
T \cap (\mu + c_3 \delta B^d_2), \frac{\delta}{c_2} B^d_2\Big). 
\]
Then there is a constant $c_4 = c_4(c_1, c_2, c_3) > 0$ such that the inequality $\delta \leq c_4 \; \eps_\star(\sigma)$ holds.
\end{lemma} 
\begin{proof}
     Define 
    \[
    h(\delta) \defn \sup_{\mu \in T} 
\log M\Big(
T \cap (\mu + c_3 \delta B^d_2), \frac{\delta}{c_2} B^d_2\Big).
    \]
   Since $T$ is convex, we easily see that $h$ is nonincreasing. Therefore,
    \[
    \frac{\delta^2}{c_1 \sigma^2} \leq h(\delta) 
    \leq 
    h\Big(\frac{\delta}{c_3 \sqrt{c_1}}\Big) = 
    \sup_{\mu \in T}
    \log M\Big( 
    T \cap (\mu + \tfrac{\delta}{\sqrt{c_1}} B^d_2), 
   \frac{\delta}{c_2 c_3 \sqrt{c_1}} B^d_2\Big) 
    \defn \log M_\delta.
    \]

    We may assume that $M_\delta \geq 2$, otherwise the claim is trivial.
    Suppose first that $M_\delta \geq 5$. Then, Fano's inequality (e.g., see \cite[Proposition 15.12]{Wai19}), applied to the packing set implied by $M_\delta$ above, gives us 
    \begin{equation}\label{ineq:lower-bound-one}
    \eps_\star^2(\sigma)
    \geq 
    \frac{\delta^2}{4 c_2^2 c_3^2 c_1 }
    \Big\{1 - \frac{\log 2}{\log 5} - 
    \frac{1}{2} \frac{\delta^2/(c_1 \sigma^2)}{\log M_\delta}\Big\}
    \geq \underbrace{\frac{1}{4 c_2^2 c_3^2 c_1} 
    \Big(\half - \frac{\log 2}{\log 5}\Big)}_{\defn C_1} \, \delta^2
    = C_1 \delta^2. 
    \end{equation}
    Now suppose that instead we have $2 \leq M_\delta < 5$. 
    Then we can apply Le Cam's two-point lower bound (e.g., see \cite[Equation 15.14]{Wai19}) to 
    any pair of distinct points (say $\eta, \eta'$) in the packing.
    Denoting by $\|P - Q\|_{\rm TV}$ the total variation distance between probability measures $P, Q$ on the same space, this will give us
    \begin{multline}
    \eps_\star^2(\sigma)
    \geq \frac{\delta^2}{8 c_2^2 c_3^2 c_1}
    \Big\{1 - \|\Normal{\eta}{\sigma^2 I_\dimension} - 
    \Normal{\eta'}{\sigma^2 I_\dimension}\|_{\rm TV}
    \Big\} 
    \geq 
    \frac{\delta^2}{8 c_2^2 c_3^2 c_1}
    \Big\{1 - 
    \sqrt{1 - \e^{- \|\eta - \eta'\|_2^2/(2\sigma^2)}}\Big\} \\
    \geq 
    \frac{\delta^2}{8 c_2^2 c_3^2 c_1}
    \Big\{1 - 
    \sqrt{1 - \frac{1}{ M_\delta^2}}\Big\} 
    \geq\underbrace{\frac{1}{8 c_2^2 c_3^2 c_1}
    \Big\{1 - 
    \sqrt{\frac{15}{16}}\Big\}}_{\defn C_2} \delta^2
    = C_2 \delta^2.\label{ineq:lower-bound-two}
    \end{multline}
    Above, we bounded the total variation distance from above in terms of relative entropy, using the 
    Bretagnolle-Huber inequality~\cite[Lemma 2.1]{BreHub79}.
    By combining bounds~\eqref{ineq:lower-bound-one} and~\eqref{ineq:lower-bound-two}, we obtain
    with $C_3 = \twomin{C_1}{C_2}$ that
    \[
    \eps_\star^2(\sigma) \geq C_3 \, \delta^2.
    \]
    Taking $c_4 = \sqrt{1/C_3}$ completes the proof.
\end{proof}

\end{document}